\documentclass[12pt]{article}
\usepackage[small,compact]{titlesec}
\usepackage{amsmath,amssymb,amsfonts,mathabx,setspace,bm} 

\usepackage{graphicx,caption,epsfig,subfigure,epsfig,wrapfig}
\usepackage{url,color,verbatim,algorithmic}
\usepackage[sort,compress]{cite}
\usepackage[ruled,boxed]{algorithm2e}
\usepackage[scaled]{helvet}
\usepackage[T1]{fontenc}
\usepackage[bookmarks=true, bookmarksnumbered=true, colorlinks=true,   pdfstartview=FitV,
linkcolor=blue, citecolor=blue, urlcolor=blue]{hyperref}
\usepackage[sort,compress]{cite}
\usepackage{float}

\usepackage{multirow}
\usepackage{booktabs}
\usepackage{threeparttable}
\usepackage{diagbox}
\def\R{\mathbb{R}}

\def\Gbar{ {\overline{G}} }

\def\Abar{ {\overline{A}} }
\def\bbar{ {\overline{b}} }
\def\LGbar{ {\mathcal{L}_{\overline{G}}}  }

\newcommand{\innerp}[1]{\langle{#1}\rangle}

\newcommand{\floor}[1]{\lfloor{#1}\rfloor}

\newcommand{\argmin}[1]{\underset{#1}{\operatorname{arg}\operatorname{min}}\;}

\definecolor{dgreen}{RGB}{0,153,76}

\newtheorem{theorem}{Theorem}

\newtheorem{definition}[theorem]{Definition}

\newtheorem{lemma}[theorem]{Lemma}

\newtheorem{remark}[theorem]{Remark}
\newenvironment{proof}[1][Proof]{\noindent\textbf{#1.} }{\ \rule{0.5em}{0.5em}}

\numberwithin{equation}{section}
\numberwithin{theorem}{section}

\newtheorem{assumption}{Assumption}
\renewcommand{\theassumption}{H\arabic{assumption}} 

\newcommand\keywords[1]{\textbf{Keywords}: #1}

\begin{document}

\title{\textbf{Learning L\'evy density via adaptive RKHS regression with bi-level optimization
}}
\author{
Luxuan Yang\footnotemark[2]\;, 
Fei Lu\footnotemark[1]\;, 
Ting Gao\footnotemark[3]\;, Wei Wei\footnotemark[4] \;and Jinqiao Duan\footnotemark[5]
}
\footnotetext[2]{School of Mathematics and Statistics \& Center for Mathematical Sciences, Huazhong University of Science and Technology, Wuhan 430074, China. Email: \texttt{luxuan\_yang@hust.edu.cn}}
\footnotetext[1]{Department of Mathematics, Johns Hopkins University, Baltimore, USA. Email: 
\texttt{feilu@math.jhu.edu}}
\footnotetext[3]{School of Mathematics and Statistics \& Center for Mathematical Sciences, Huazhong University of Science and Technology, Wuhan 430074, China. Email:
\texttt{tgao0716@hust.edu.cn}}
\footnotetext[4]{College of Mathematics and Statistics, Chongqing University, Chongqing 401331, China. Email: \texttt{weiw@cqu.edu.cn}}
\footnotetext[5]{Department of Mathematics and Department of Physics, Great Bay University, Dongguan, Guangdong 52300, China. Email: \texttt{duan@gbu.edu.cn}}
\footnotetext[1]{is the corresponding author}
\date{}
\maketitle

\begin{abstract}

We propose a nonparametric method to learn the L\'evy density from probability density data governed by a nonlocal Fokker–Planck equation. We recast the problem as identifying the kernel in a nonlocal integral operator from discrete data, which leads to an ill-posed inverse problem. To regularize it, we construct an adaptive reproducing kernel Hilbert space (RKHS) whose kernel is built directly from the data. Under standard source and spectral decay conditions, we show that the reconstruction error decays in the mesh size at a near optimal rate. Importantly, we develop a generalized singular value decomposition (GSVD)-based bilevel optimization algorithm to choose the regularization parameter, leading to efficient and robust computation of the regularized estimator. Numerical experiments for several L\'evy densities, drift fields and data types (PDE-based densities and sample ensemble-based KDE reconstructions) demonstrate that our bilevel RKHS method outperforms classical L-curve and generalized cross-validation strategies and that the adaptive RKHS norm is more accurate and robust than $L^2_\rho$- and $\ell^2$-based regularization.

\vspace{2em}
\end{abstract}
\setcounter{tocdepth}{1}
\tableofcontents 

\section{Introduction}
\label{Introduction}
We consider the problem of estimating the  L\'evy density $\phi$ of a L\'evy motion $L_t$ appearing in the one-dimensional SDE
\begin{equation}
\label{eq:SDE}
d X_t=b\left(X_{t}\right) d t+ \sigma\left(X_{t}\right) d B_t+d L_t 
\end{equation}
from data consisting of the probability density $p(x,t)$ of $X_t$ on a discrete space-time mesh or from reconstructed densities based on observed sequences of sample ensembles. 
Here, the drift $b: \mathbb{R}^1 \rightarrow \mathbb{R}^1$ and the diffusion $\sigma: \mathbb{R}^1 \rightarrow \mathbb{R}^1$ are known functions, $L_t$ is a scalar and symmetric L\'evy motion with generating triplet $(0,0, \nu)$ independent of a scalar Brownian motion $B_t$ and its L\'evy jump measure $\nu$ admits a symmetric density $\phi$ with respect to Lebesgue measure, i.e. $d\nu(x) = \phi(x)\,dx$.

Identifying $\phi$ is practically important because it specifies the L\'evy jump measure $\nu$ and thereby characterizes the jump properties of the L\'evy process $L_t$, with applications to modeling volatility jumps in financial markets \cite{yu2011mcmc, kindermann2008identification, applebaum2009levy} and analyzing changes in climate systems \cite{zheng2020maximum, li2022extracting}.

A broad literature estimates L\'evy measures from trajectory data of the L\'evy process rather than from densities. In high-frequency (in time) settings, Comte and Genon-Catalot \cite{comte2009nonparametric} estimate $g(y)=y\phi(y)$ via frequency-domain and projection methods based on the characteristic function of increments, under asymptotics $\Delta t\to 0$ and $N\Delta t\to\infty$. However, recovering $\phi(y)=g(y)/y$ then becomes unstable near $y=0$.  
In low-frequency regimes, Neumann and Reiß \cite{neumann2009} estimate L\'evy–Khintchine characteristics from characteristic functions and their derivatives, but show that the diffusion coefficient cannot be uniformly consistently recovered without extra assumptions, which obstructs a unique reconstruction of $\phi$. Pathwise approaches based on the L\'evy–Itô decomposition, such as \cite{duval2021spectral}, extract jump information from thresholded high-frequency increments and then estimate a compound Poisson approximation; these methods require carefully tuned thresholds and high sampling rates to control contamination by small jumps. More recently, several works focus on estimating increment densities rather than $\phi$ itself, e.g., \cite{duval2024adaptive,duval2025nonparametric}.

 To address cases of density data without trajectory information, an alternative approach uses the Fokker–Planck equation and nonlocal Kramers–Moyal formulas. Li and Duan \cite{li2021data} fit parametric L\'evy families using small-time transition probabilities and conditional moments without explicitly solving \eqref{eq:FPE}, while Lu et al. \cite{lu2022extracting} approximate transition densities via normalizing flows and then infer a stability index within a prespecified L\'evy family. These approaches exploit the nonlocal generator but remain confined to parametric estimation. Parallel to this, there is a growing literature on learning nonlocal operators from data: many works treat the operator as a black box to be approximated by neural networks \cite{pang2020npinns,you2022nonlocal}, while others learn parametric kernels in nonlocal models (e.g., Bernstein polynomials with $\ell^2$ regularization \cite{you2021data,fan2023bayesian}). To improve flexibility, Lu et al. \cite{lu2022data} proposed a data-adaptive RKHS Tikhonov regularization framework for learning the kernel. However, the impact of mesh-dependent discretization errors and the principled selection of regularization parameters remain to be further explored.

In this work, we learn $\phi$ by exploiting the Fokker–Planck equation for $p(x,t)$:
\begin{equation}
    \label{eq:FPE}
    \partial_t p =-\partial_x\left(b p \right)+\frac{1}{2} \partial_{x x}(\sigma^2 p) +\int_{\mathbb{R}^1 \backslash\{0\}}\left[p(x+y, t)-p(x, t)\right] \phi(|y|) d y,
\end{equation}
and regard $\phi$ as an unknown kernel in a nonlocal operator to be inferred from data $\mathcal{D}=\{p_{t_i}(x_j)\}_{i,j=1}^{N,J}$ on a uniform spatial grid $\{x_j\}\subset\Omega_{R_0}$ and observation times $\{t_i\}$. To make the nonlocal integral numerically tractable, we assume that $\phi$ is supported on $(0,R_0]$ (a standard modeling choice; see, e.g., \cite{zhang2013risk, comte2009nonparametric}) and restrict $x$ to $\Omega_{R_0}:=\{z\in\Omega:\,[z-R_0,z+R_0]\subset\Omega\}$ so that $x\pm y \in \Omega$ for all $|y|\le R_0$. In this setting, recovering $\phi$ from discrete data and the numerical errors in estimating the derivatives of $p$ leads to an ill-posed inverse problem. This makes regularization indispensable and motivates a mesh-dependent error analysis.

Based on the above formulation, we develop a nonparametric regression framework to learn $\phi$ from data using the data-adaptive RKHS Tikhonov regularization of \cite{lu2022data} with a principled bi-level optimization scheme for selecting the regularization hyperparameter. Specifically, from data, we construct an adaptive reproducing kernel $\Gbar$ whose associated integral operator is precisely the Hessian of the quadratic loss for $\phi$. This leads to an RKHS $H_{\Gbar}$ that (i) is tailored to the analytic structure of the nonlocal operator, (ii) automatically restricts estimation to the subspace of identifiability, and (iii) provides a natural regularization norm.

Notably, we derive a mesh–dependent error bound for the RKHS‑regularized estimator and show that, with an optimally chosen regularization parameter, the reconstruction error decays like $(\Delta x)^{\frac{2\beta \varsigma}{2\beta \varsigma+4\varsigma+1}}$ before saturation, where $\beta$ is the smooth exponent in a source condition and $\varsigma$ is the power exponent of spectral decay. This extends the optimal order of Tikhonov regularization (e.g., \cite[Sec.~3.2]{engl1996regularization}) to a setting where the normal operator and the data are both discretized and the regularization uses an adaptive RKHS norm. When numerical quadrature errors are negligible, our rates become comparable (up to minor exponent differences due to the different effective dimension scalings) to the minimax optimal rates for inverse statistical learning in \cite{blanchard2018optimal} and the statistical inverse problem with DA–RKHS regularization in \cite{zhang2025minimax}, while explicitly quantifying how spatial discretization degrades the reconstruction accuracy.

Importantly, we develop a generalized singular value decomposition (GSVD)-based bi-level optimization scheme to select the regularization hyperparameter from data. In our bi-level optimization, the lower level solves a Tikhonov-regularized least-squares problem, for which GSVD yields both the closed-form solution and the inverse-Hessian needed to compute hypergradients in the upper-level update. Thus, while it is inspired by recent bi-level optimization approaches \cite{nelsen2025bilevel,chada2022consistency}, our \emph{bilevel-RKHS} approach avoids repeated computation inverse-Hessian and gives a stable, efficient hyperparameter selection procedure. 

Numerically, we show that the \emph{bilevel-RKHS} outperforms classical L-curve \cite{hansen1992analysis} and generalized cross-validation \cite{golub1979generalized} strategies and that the adaptive RKHS norm is superior to $L^2_\rho$- and $\ell^2$-based regularization across multiple L\'evy densities, drift fields, and data types (PDE-based density snapshots and sample ensemble-based KDE reconstructions).

Our contributions can be summarized as follows:
\begin{itemize}
    \item We formulate the recovery of L\'evy density $\phi$ from the nonlocal Fokker–Planck operator as an ill-posed inverse problem, and construct an automatic adaptive RKHS $H_{\Gbar}$ that encodes identifiability and provides a principled regularization norm.
    \item We establish mesh-dependent convergence rates for the RKHS-regularized estimator under suitable source and spectral decay conditions, explicitly capturing the effect of spatial discretization ($\Delta x$) on the reconstruction error.
    \item We design an efficient GSVD-based bi-level RKHS algorithm for selecting the regularization hyperparameter with closed-form hypergradients, and demonstrate its numerical advantages over L-curve and GCV, as well as over alternative norms.
\end{itemize}

Beyond the specific task of L\'evy density estimation, our framework contributes to the broader program of learning nonlocal operators and ill-posed inverse problems from data. It combines structure-exploiting operator learning, adaptive function spaces, and bilevel hyperparameter optimization and can be extended in several directions: to higher-dimensional L\'evy processes and to joint nonparametric recovery of drift, diffusion, and jump components. We view this work as a step toward data-adaptive learning of stochastic dynamics with jumps, and anticipate that the GSVD-based bilevel selection of hyperparameters will be useful in a range of nonlocal operator learning problems. 

Our paper is organized as follows. In Section \ref{Inverseproblem}, we formulate the inverse problem, define the adaptive RKHS and its discrete approximation. Section \ref{sec3} establishes convergence rates for the RKHS-regularized estimator. We present in Section \ref{secbilevel} the bi-level optimization algorithm, and report numerical experiments for different L\'evy densities and data-generation scenarios in Section \ref{sec4}. Section \ref{sec:conclusion} concludes and discusses possible extensions.

\paragraph{Preliminaries and notations.} Here, we briefly introduce the definitions and notation used throughout the paper.

\noindent\textbf{L\'evy Process: }
Let  $X=(X(t), t \geq 0)$  be a stochastic process defined on a probability space  $(\Omega, \mathcal{F}, P)$. We say that  $X$ is a \emph{L\'evy process} if:
(i)  $X(0)=0$  (a.s);
(ii)  $X$  has independent and stationary increments;
(iii)  $X$  is stochastically continuous, i.e. for all  $a>0$  and all  $s \geq 0$ , $\lim _{t \rightarrow s} P(|X(t)-X(s)|>a)=0$.

\noindent\textbf{L\'evy Measure: }
Let $v$ be a Borel measure defined on $\mathbb{R}^d \backslash\{0\}$. We say that it is a \emph{L\'evy measure} if $\int_{\mathbb{R}^d-\{0\}}\left(|y|^2 \wedge 1\right) \nu(d y)<\infty.$ If the L\'evy measure $\nu$ is absolutely continuous with respect to Lebesgue measure, the Radon-Nikodym derivative $\phi=d \nu / d x$ is called the \emph{L\'evy density}. Then, when L\'evy Process is symmetric, L\'evy Measure $\nu$ is symmetric so that  L\'evy density is .

\noindent\textbf{Symmetric L\'evy process: } A L\'evy process $X$ is called \emph{symmetric} if $L_t = -L_t$ for all $t\ge 0$.
In this case, the L\'evy measure is symmetric in the sense that $\nu(A)=\nu(-A)$ for all Borel sets
$A\subset \mathbb{R}\setminus\{0\}$. Consequently, If the L\'evy measure $\nu$ is absolutely continuous, the L\'evy density satisfies
$\phi(y)=\phi(-y)$ for a.e.\ $y\neq 0$.

\noindent The whole paper follows the notation in Table \ref{tab:notaion}.
\begin{table}[H]
    \centering
    \caption{Table of notations}
    \begin{tabular}{cc}
        \toprule
       $\mathbf{\Abar}$ and $\mathbf{\Abar}^M$	& normal matrices of continuum and discrete data  \\
        \hline
        $\mathbf{\bbar}$ and $\mathbf{\bbar}^M$  & normal vectors of continuum and discrete data \\
        \hline
        $\bar{\mathbf{G}}$, $\mathbf{Q}$, $\mathbf{U}$, $\mathbf{X}$, $\mathbf{K}$, $\mathbf{M}$ & matrix or array by bold capital letters  \\
        \hline
        $\mathbf{f}$, $\mathbf{c}$ & vector by bold lowercase letters \\
        \hline
        $\gamma$, $\lambda$  & scalar by Greek letters\\
        \bottomrule
    \end{tabular}
    \vspace{0cm}
    \label{tab:notaion}
\end{table}

\section{Adaptive kernel regression and regularization
}
\label{Inverseproblem}
Assuming continuum-in-space data, we first formulate the estimation of the L\'evy density as a variational inverse problem. We then show that the problem is ill-posed and propose an adaptive RKHS framework for regression and regularization. Throughout, we assume that the L\'evy density $\phi$ is symmetric and has bounded support in $(-R_0, R_0)$, a standard assumption (see, e.g., \cite{zhang2013risk}) and we focus on the computationally practical case $R_0 < \infty$. 

\subsection{Regression estimator} 
We recover the L\'evy density by solving an inverse problem of learning the convolution kernel in the integral operator of the Fokker-Planck equation. Specifically, with $p_t$ denoting the probability density of $X_t$, we define $ R_\phi[p_t]$ the a nonlocal operator parametrized by the L\'evy density $\phi$: 
\begin{equation}
\label{eq:nonop}
\begin{aligned}
R_{\phi}[p_t](x):=&\int_{\Omega} \phi(|y|)[p_t(y+x)-p_t(x)] d y =  \int_{(0,R_0]} \phi(r)Q[p_t](x, r) d r, 
\end{aligned}
\end{equation}
for $x\in \Omega_{R_0}: = \{z\in \Omega: [z-R_0,z+R_0]\in \Omega \}$, and the functional $Q[p_t]:\Omega_{R_0}\times (0,R_0]\to \R$ is 
\begin{equation}
\label{eq:Q[p]}
Q[p_t](x, r) = p_t(x + r) + p_t(x - r)- 2p_t(x), r\in (0,R_0], x\in \Omega_{R_0}. 
\end{equation}
Note that the evaluation of $R_\phi[p_t](x)$ with $x\in \Omega_{R_0}$ uses the values of $p_t$ in $\Omega$.

The task is to estimate the function $\phi$ of the nonlocal operator $R_\phi$ in the following reformulation of the Fokker-Planck equation 
\begin{equation}
   \label{eq:fdata} 
 R_\phi[p_t](x) = f_t(x):= \frac{\partial p_t(x)}{\partial t}+ \frac{\partial}{\partial x}\left[b(x) p_t(x)\right]- \frac{1}{2} \frac{\partial^2 \left[\sigma(x)^2p_t(x)\right]}{\partial x^2}.  
\end{equation}

 We assume first that we have continuum data of function pairs:
\begin{equation}
\label{eq:densitydata}
\mathcal{T}=\left\{\left(p_{t_i}(x)_{x\in \Omega}, f_{t_i}(x) _ {x\in \Omega_{R_0} }\right)  \right\}_{i=1}^N.  
\end{equation} 

Consider the minimizer of the loss function of mean square error:
\begin{equation}
\label{eq:meanloss}
    \widehat{\phi}_{H_n}=\underset{\phi\in H_n}{\arg \min } \, \mathcal{E}_\infty (\phi),  \quad \text{where} \quad \mathcal{E}_\infty (\phi)=\frac{1}{N} \sum_{i=1}^N \int_{\Omega_{R_0}}\left|R_{\phi}\left[p_{t_i}\right](x)-f_{t_i}(x)\right|^2 dx.
\end{equation}
Here, the hypothesis space 
$H_n= \operatorname{span}\left\{{\phi}_k\right\}_{k=1}^n	
$ with basis functions $\left\{{\phi}_k\right\}$ will be selected in the next sections using an adaptive RKHS. For each $\phi=\sum_{k=1}^n c_k {\phi}_k \in H_n$, noticing that $R_{\phi}=\sum_{k=1}^n c_k R_{{\phi}_k}$, we can write the loss function as :
\begin{equation}
\label{eq: loss}
\begin{aligned}
\mathcal{E}_\infty \left(\phi\right)
  & =\frac{1}{N} \sum_{i=1}^N \int_{\Omega_{R_0}} \left|\sum_{k=1}^n c_k R_{{\phi}_k}\left[p_{t_i}\right](x)-f_{t_i}(x)\right|^2 dx\\
  & = \mathbf{c}^\top  \mathbf{\Abar}  \mathbf{c} - 2 \mathbf{c}^\top \mathbf{\bbar} + C_N^f =:  \mathcal{E}(\mathbf{c}), 
\end{aligned}
\end{equation}
where $\mathbf{c}= (c_1,\ldots,c_n)^\top$, the normal matrix and vectors are 
\begin{equation}
\label{eq:Aandb}
\begin{aligned}
 \mathbf{\Abar}(k, k')
    & = \frac{1}{N} \sum_{i=1}^N \int_{\Omega_{R_0}} R_{{\phi}_k}\left[p_{t_i}\right](x)R_{{\phi}_{k'}}\left[p_{t_i}\right](x) dx, \quad 1\leq k,k'\leq n 
\\
\mathbf{\bbar}(k) & =\frac{1}{N}\sum_{i=1}^N \int_{\Omega_{R_0}}R_{{\phi}_k}\left[p_{t_i}\right](x) f_{t_i}(x)  dx , \quad 1\leq k\leq n,  
\end{aligned}
\end{equation}
and $C_N^{f}=\frac{1}{N} \sum_{i=1}^N \int_{\Omega_{R_0}} \left|f_{t_i}(x)\right|^2 dx $ is a constant. Hence, a minimizer for \eqref{eq:meanloss}, which is a least squares estimator(LSE) with a minimal coefficient norm,  is  $\widehat{\phi}_{H_n}=\sum_{k=1}^n \widehat{c}_k {\phi}_k$ with  $\widehat{\mathbf{c}}=\mathbf{\Abar}^{\dag} \mathbf{\bbar}$, where $\mathbf{\Abar}^{\dag}$ is pseudo-inverse of $\mathbf{\Abar}$. 

Regularization is necessary in practice since the normal matrix $\mathbf{\Abar}$ is often ill-conditioned and the normal vector $\mathbf{\bbar}$ is perturbed by sampling error or numerical integration error. In fact, as we will show in the next section, the underlying inverse problem is ill-posed.  

Before picking a proper regularization norm, we must first address a fundamental issue: defining the function space for estimating $\phi$. Subsequently, we define function spaces adaptive to this inverse problem, characterize its ill-posedness, and propose an adaptive RKHS for regularization. 

\subsection{An adaptive RKHS for regression and regularization}

Since we have little prior information about the L\'evy density, we define function spaces adaptive to this variational approach through the lens of statistical learning, as in \cite{lu2023nonparametric,lu2022data,lang2022learning}. In particular, we characterize the ill-posedness of this variational inverse problem using the loss function's Hessian and define an adaptive RKHS for regression. Moreover, we assume uniform boundedness of continuum data, which often holds, for example, when the diffusion term is uniformly elliptic so the density is smooth \cite{chen2017heat}.

\begin{assumption}\label{assump:data}  
The continuum data $\{p_{t_i}\}_{i=1}^N$ sastisy that  
 \begin{equation}
 \begin{aligned}
 C_{max}:=  \max_{i} \sup_{x\in \Omega} p_{t_i}(x) <\infty, \quad 
 Z      := \int_0^{R_0} \frac{1}{N} \sum_{i=1}^N \int_{\Omega_{R_0}} \left|Q[p_{t_i}](x,r)\right| d x d r<\infty, 
 \end{aligned}
 \end{equation}
 where $Q[p_t]$ is defined in \eqref{eq:Q[p]}. 
\end{assumption}

Given data $\{p_{t_i}(x)\}_{i=1}^N$ satisfying {\rm Assumption \ref{assump:data}}, let $\rho$ be the measure on $(0,R_0]$ whose density function with respect to the Lebesgue measure is  
\begin{equation} \label{eq:rho}
   \rho(r)= \frac{1}{Z N} \sum_{i=1}^N \int_{\Omega_{R_0}} \left|Q[p_{t_i}](x,r)\right| d x. 
 \end{equation}
The measure $\rho$ quantifies the strength that the data explores $\phi$. In particular, outside of its support, the data provides no information about $\phi$. Thus, we call it an \emph{exploration measure}, and set the default function space of estimating $\phi$ to be $L^2_\rho$.

We define a function $\Gbar:[0,R_0]\times [0,R_0]\to \R$ as 
\begin{equation}
\label{eq:barG}
\Gbar(r, s)=\frac{G(r, s)}{\rho(r) \rho(s)}\mathbf{1}_{\{ \rho(r) \rho(s)>0 \} }, \quad G(r, s)=\frac{1}{N} \sum_{i=1}^N \int Q[p_{t_i}](x, r)Q[p_{t_i}](x, s) d x, 
\end{equation}
where $\rho$ is the exploration measure defined in \eqref{eq:rho}. 

The next lemma shows that $\Gbar$ is a positive semi-definite and square-integrable function, so its integral operator $\LGbar$ is compact. In particular, the variational inverse problem of minimizing the quadratic loss function in \eqref{eq:meanloss} is ill-posed since it involves the inversion of the operator $\LGbar$. We postpone its proof to Appendix \ref{append:A}. 
\begin{lemma} \label{lemma:rk}
Given data $\{p_{t_i}\}_{i=1}^N$ satisfying {\rm Assumption \ref{assump:data}}, the function $\Gbar$ in \eqref{eq:barG} is a positive semi-definite function and $\Gbar\in L^2(\rho\otimes \rho)$. Consequently, its integral operator $\LGbar:L^2_\rho\to L^2_\rho$ 
\begin{equation}\label{eq:LGbar}
\LGbar \phi (r)=\int_0^{R_0} \phi(s) \Gbar(r, s) \rho(s) d s, \quad \forall \phi\in L^2_\rho
\end{equation}   
is compact, positive semi-definite. Furthermore, the loss function in \eqref{eq:meanloss} can be written as 
\begin{equation} \label{eq:lossFn_LG}
\begin{aligned}
\mathcal{E}_\infty \left({\phi}\right) 
  & = \innerp{\LGbar \phi, \phi}_{L^2_\rho} - 2\innerp{\phi^D, \phi}_{L^2_\rho} + C, 
\end{aligned}
\end{equation}
and its minimizer in $\mathrm{Null}(\LGbar)^\perp$ is $\widehat \phi = \LGbar^{-1} \phi^D$. Here, the function $\phi^D\in L^2_\rho$ is the Riesz representation of the bounded linear functional  
\begin{equation}
 \label{eq:phiD}
\left\langle \phi^D, \phi\right\rangle_{L^2_{\rho}}=\frac{1}{N} \sum_{i=1}^N \int  R_{\phi}\left[p_{t_i}\right](x) f_{t_i}(x) d x, \forall \phi \in L^2_{\rho} .
\end{equation}
\end{lemma}

Therefore, the inverse problem of minimizing the loss function is ill-posed, and it is crucial to regularize the problem by seeking solutions in $\mathrm{Null}(\LGbar)^\perp$. We introduce the following RKHS for regularization.     

\begin{definition}[Adaptive RKHS] \label{def:auto-kernel}
We call $\Gbar$ in \eqref{eq:barG} the \emph{automatic reproducing kernel} for the inverse problem of learning the density of the L\'evy measure $\phi$ in the operator $R_{\phi}$ in \eqref{eq:nonop} from data $\{p_{t_i}\}_{i=1}^N$ satisfying {\rm Assumption \ref{assump:data}}.  Its associated RKHS, denoted by $H_{\Gbar}$, is called an adaptive RKHS. 
 \end{definition}
 The RKHS $H_{\Gbar}$ is data-adaptive since its reproducing kernel $\Gbar$ adjusts to the continuum data. It is well-known that the RKHS can be expressed as $H_{\Gbar}:=\LGbar^{1/2}(L^2_\rho)$; see e.g., \cite{cucker2007learning,lu2022data,zhang2025minimax}.

 Notably, the closure of $H_{\Gbar}$ coincides with $\mathrm{Null}(\LGbar)^\perp$, within which the loss function admits a unique minimizer, making it an ideal space for regularization. Consequently, the RKHS $H_{\Gbar}$ naturally leads to the derivation of a regularized estimator: 
\begin{equation}
\label{eq:wholeloss}
\widehat{\phi}_\lambda =\argmin{\phi \in H_{\Gbar}}   \mathcal{E}_\lambda(\phi), \quad 
\mathcal{E}_\lambda(\phi): =\mathcal{E}_\infty(\phi) +\lambda \|\phi\|_{H_{\Gbar}}^2. 
\end{equation}
Here, $\lambda$ is a hyperparameter controlling the strength of regularization.

In particular, we can also use $\Gbar$ kernel regression. That is, setting the basis function of the hypothesis $H_n= \mathrm{span}\{\phi_k\}_{k=1}^n$ to be 
\begin{equation}\label{eq:Gbar_basisFn}
\phi_k(r)=  \Gbar(r_k,r) , \,1\leq k\leq n,  
\end{equation}
where $\{r_k\}_{k=1}^n\subset (0,R_0]$ are mesh points, we have, by the reproducing property, for any  $\phi (r) =\sum_{k=1}^n c_k \Gbar(r_k,r) \in H_n$, 
\begin{equation}
\label{regularkernel}
\begin{aligned}
\|{\phi}\|_{H_{\Gbar}}^2
 & =\sum_{k=1}^n \sum_{m=1}^n c_{k} c_{m} \left\langle \Gbar(r_{k},r), \Gbar(r_{m},r) \right\rangle_{H_{\Gbar}} =\sum_{k=1}^n \sum_{m=1}^n  c_{k} c_{m}\Gbar(r_{k},r_{m})=\mathbf{c}^{\top}\bar{\mathbf{G}}\mathbf{c},
\end{aligned}
\end{equation}
where $\bar{\mathbf{G}} = \begin{bmatrix}
\Gbar(r_1,r_1) & \cdots  &  \Gbar(r_1,r_n)\\
  \vdots & \ddots  & \vdots\\
  \Gbar(r_n,r_1) & \cdots  &  \Gbar(r_n,r_n)
\end{bmatrix}$, similar to \cite{li2025automatic}. Then, the loss \eqref{eq: loss} can be written as: 
\begin{equation}
\label{eq:regularkernelloss}
\begin{aligned}
\mathcal{E}_\lambda(\phi) &=\mathcal{E}_\infty(\phi)+\lambda\|{\phi}\|_{H_{\Gbar}}^2 = \mathbf{c}^{\top}\left( \mathbf{\Abar} + \lambda \bar{\mathbf{G}} \right)\mathbf{c}-2 \mathbf{c}^{\top} \mathbf{\bbar}  + C_N^f, 
\end{aligned}
\end{equation}
where $\mathbf{\Abar}$ and $\mathbf{\bbar}$ have the same form as in \eqref{eq:Aandb} with $\phi_k=\bar{G}\left(r_k, \cdot\right)$ and $\phi_{k^{\prime}}=\bar{G}\left(r_{k^{\prime}}, \cdot\right)$.

\subsection{Data-based approximation of the regression and regularization} 
\label{sec31}

Given discrete data $\left\{p_{t_i}\left(x_j\right), j=1, \ldots, M\right\}_{i=1}^N$, we approximate $\mathbf{\Abar}$, $\mathbf{\bbar}$ and $\bar{\mathbf{G}}$ in \eqref{eq:regularkernelloss} to compute the estimator. We first write the data as input–output pairs of the nonlocal operator: 
\begin{equation}
\label{eq: discretedata}
\mathcal{D}=\left\{\left(p_{t_i}\left(x_j\right), \tilde{f}_{t_i}\left(x_j\right)\right), j=1, \ldots, M\right\}_{i=1}^N,
\end{equation}
where the value of $\tilde{f}_{t_i}(x_j)$ is computed via finite difference. That is, for the interior grid points $x_j \in (-L+R_0, L-R_0)$, 
\begin{equation}
\label{eq: discretef}
\begin{gathered}
\tilde{f}_{t_i}\left(x_j\right) \doteq \frac{p_{t_{i+1}}\left(x_j\right)-p_{t_{i}}\left(x_j\right)}{\Delta t}+\frac{b\left(x_{j+1}\right) p_{t_{i}}\left(x_{j+1}\right)-b\left(x_{j-1}\right) p_{t_{i}}\left(x_{j-1}\right)}{2 \Delta x} \\
-\frac{1}{2} \frac{p_{t_{i}}\left(x_{j+1}\right)-2 p_{t_{i}}\left(x_{j}\right)+p_{t_{i}}\left(x_{j-1}\right)}{\Delta x^2}.
\end{gathered}
\end{equation}
For the left/right boundary points, i.e., at $x_j=-L+R_0$ or $x_j=L-R_0$, we approximate the advection term using a three-point forward/backward difference and the diffusion term using a four-point forward/backward difference. Both one-sided schemes for boundary points exhibit $O\left(\Delta x^2\right)$ numerical error. 

We use Riemann sum to approximate the integrals in the reproducing kernel $\Gbar$ in \eqref{eq:barG} and in $R_{\Gbar(r_k,\cdot)}[p_{t_i}](x_j)$. The process starts from piecewise constant approximations of the functions $Q\left[p_{t_i}\right]\left(x_j, r\right)$:  
\begin{equation}
\label{eq: discreteQ}   
\widehat{Q}\left[p_{t_i}\right]\left(x_j, r\right)=\sum_{k=1}^n Q\left[p_{t_i}\right]\left(x_j, r_k\right) \mathbf{1}_{I_k}(r),
\end{equation}
for $x_j\in \Omega_{R_0}=\left[-L+R_0, L-R_0\right]$. Here, $\{r_k\}$ is a partition of $(0,R_0]$ and are set to be $r_k=k \Delta x, k=$ $1,2, \ldots, n$ with $n=\floor{R_0/\Delta x}$, and $\mathbf{1}_{I_k}(r)$ is the the indicator function  for the interval $I_k=((k-1) \Delta x, k \Delta x]$. Accordingly, the discrete counterpart of the exploration measure $\rho$ in \eqref{eq:rho} and the reproducing kernel $G(r, s)$ in \eqref{eq:barG}, denonted by $\widehat{\rho}$ and $G^M$, are given by
\begin{equation}
\label{eq: discreteG}
\begin{aligned}
\widehat{\rho}(r) &= \frac{1}{Z N} \sum_{i, j, k}\left|Q\left[p_{t_i}\right]\left(x_j, r_k\right)\right| \Delta x \mathbf{1}_{I_k}(r),\\
G^M(r, s) & =\frac{\Delta x}{N} \sum_{i=1}^N \sum_{m, \ell, \ell^{'}} Q\left[p_{t_i}\right]\left(x_m, r_{\ell}\right) Q\left[p_{t_i}\right]\left(x_m, s_{\ell^{'}}\right) \mathbf{1}_{I_{\ell}}\left(r\right) \mathbf{1}_{I_{\ell^{'}}}\left(s\right), 
\end{aligned}
\end{equation}
and the basis function $\phi_k(r)= \Gbar(r_k,r)$ is 
\begin{equation}
\label{eq:disbasis}   
\phi_k^M(r)=\Gbar^M\left(r_k, r\right)=\sum_{\ell=1}^n \Gbar^M\left(r_k, r_{\ell}\right) \mathbf{1}_{I_{\ell}}(r), \;\text{with}\; \Gbar^M(r, s)  =\frac{G^M(r, s)}{\widehat{\rho}(r) \widehat{\rho}(s)}.
\end{equation}

 With the notation $\Phi^M(r)=\left(\phi_1^M(r), \ldots, \phi_n^M(r)\right)^{\top}$, the estimator $\hat{\phi}_\lambda^{n, M}$ in the finite-dimensional space $H_{\Gbar^M}= \operatorname{span}\left\{\phi_k^M\right\}_{k=1}^n $ admits the representation 
 \begin{equation}
\label{eq:defphiMma}  
\hat\phi_\lambda^{n,M}(r)
= \bigl(\hat{\mathbf{c}}_\lambda^{n,M}\bigr)^\top \Phi^M(r),\qquad
\hat{\mathbf{c}}_\lambda^{n,M}
= \bigl(\mathbf{\Abar}^M + \lambda \mathbf{\Gbar}^M\bigr)^{\dagger}\mathbf{\bbar} ^M.  
\end{equation} 
 Here, $\mathbf{\Abar}^M$ and $\mathbf{\bbar} ^M$ approximate $\mathbf{\Abar}$ and $\mathbf{\bbar}$ in \eqref{eq:regularkernelloss} and are given by 
\begin{equation}
\label{eq:matrix}
\begin{aligned}
 \mathbf{\Abar}^M =\mathbf{\Gbar}^M\mathbf{G}^M\mathbf{\Gbar}^M (\Delta x)^2, \quad \mathbf{\bbar} ^M = \mathbf{\Gbar}^M\mathbf{Q}^\top \mathbf{f}\Delta x, 
\end{aligned}
\end{equation}
where $\mathbf{G}^M$ 
$\mathbf{\Gbar}^M$, $\mathbf{Q}$ and $\mathbf{f}$ are given by  
\begin{equation}
\label{eq:matrixexpression}
\begin{aligned}
\mathbf{Q}_{(i, j), k} & =
\sqrt{\frac{\Delta x}{N}}Q\left[p_{t_i}\right]\left(x_j, r_k\right) \in \mathbb{R}^{N M \times n}, \quad
&\mathbf{f}_{(i, j)} & =\sqrt{\frac{\Delta x}{N}} \tilde{f}_{t_i}\left(x_j\right) \in \mathbb{R}^{N M},\\
\mathbf{G}^M(k, \ell) & =\sum_{i=1}^N \sum_{j=1}^M  \mathbf{Q}_{(i, j), k} \mathbf{Q}_{(i, j), \ell}
& \mathbf{\Gbar}^M(k, \ell) &= \frac{\mathbf{G}^M(k, \ell)}{\widehat{\rho}(r_k)\widehat{\rho}(r_\ell)} \in \mathbb{R}^{n \times n}.
\end{aligned}
\end{equation}

In particular, to improve the numerical stability,  instead of using the normal matrix $\overline{\mathbf{A}}^M$, we use the original regression matrix to solve for $\hat{\mathbf{c}}_\lambda^{n, M}$. Specifically, we solve $\hat{\mathbf{c}}_\lambda^{n, M}$ by minimizing 
\begin{equation}
\label{eq:discreteloss}
\begin{aligned}
 \mathcal{E}_\lambda(\mathbf{c})&:=\| \mathbf{Q}\mathbf{\Gbar}^M\mathbf{c} (\Delta r) - \mathbf{f}\|^2+ \lambda\| \mathbf{c}^\top  \sqrt{\mathbf{\Gbar}^M}\|^2, 
\end{aligned}
\end{equation}
which approximates the loss function in \eqref{eq:regularkernelloss}.

\section{Convergence of the RKHS-regularized estimator}
\label{sec3}
\subsection{Main results}
Under suitable smoothness conditions on the true kernel and the data, we show that the regularized RKHS estimator converges as the uniform mesh grid becomes finer. 

To ensure that the numerical errors in \eqref{eq: discretedata} are controlled by $\Delta t$ and $\Delta x$, we strengthen Assumption \ref{assump:data} and impose higher uniform regularity on $p_{t_i}$ and $\left.\partial_t p_t\right|_{t=t_i}$.

{
\renewcommand{\theassumption}{H1$^\prime$}
\begin{assumption}\label{H:boundf}
For each $i$, the probability density $p_{t_i}$ is supported on $\Omega_{R_0}\subset \R$ such that
$$
p_{t_i}\in W^{3,\infty}(\Omega_{R_0}), \qquad \left.\partial_t p_t\right|_{t=t_i}\in W^{1,\infty}(\Omega_{R_0}),
$$
with the corresponding $L^\infty(\Omega_{R_0})$–based Sobolev norms bounded uniformly in $i$.  
\end{assumption}
}
\addtocounter{assumption}{-1}

Since $Q\left[p_t\right](x, 0)=0$, we have $\rho(0)=0$. In other words, the data carries no information about $\phi(0)$, so $\phi(0)$ is not identifiable. Thus, we restrict our analysis to $r \in\left[\delta, R_0\right]$, where $\rho(r) \geq \rho_0>0$, i.e., the exploration measure is uniformly positive over the domain of interest. To further simplify the analysis, we assume that $\rho$ is a positive constant over $[\delta, R_0]$.

\begin{assumption}{\label{H:constant}}  $\rho(r) \equiv \rho_0>0$ for $\forall r \in [\delta, R_0]$ with $\delta >0$.  
\end{assumption}
This assumption enables the next lemma that bridges the discrete matrix formulation of the estimator in \eqref{eq:matrix} with a continuous operator formulation, and provides a basis for the convergence analysis in the next section. Its proof is postponed to Appendix \ref{append:A}.

\begin{lemma}[Operator form $\Leftrightarrow$ Matrix form]\label{lem:opma} 
Under Assumption {\rm \ref{H:constant}}, let $\mathcal{L}_{\Gbar^M}:L_\rho^2\to L_\rho^2$ be the integral operator with kernel $\Gbar^M$ in \eqref{eq: discreteG} and $H_{\Gbar^M} := \operatorname{span}\{\phi_k^M\}_{k=1}^n$ with $\{\phi_k^M\}$ in \eqref{eq:disbasis}. Then, the matrices $\mathbf{\Abar}^M, \mathbf{\Gbar}^M$ and vector $\mathbf{\bbar} ^M$ in \eqref{eq:matrix} satisfy 
$$
\mathbf{\Abar}^M(k, k')
= \langle \mathcal{L}_{\Gbar^M}\phi_k^M,\phi_{k'}^M\rangle_{L_\rho^2},\qquad
\mathbf{\Gbar}^M(k, k')
= \langle \mathcal{L}_{\Gbar^M}^{-1}\phi_k^M,\phi_{k'}^M\rangle_{L_\rho^2},\qquad
\mathbf{\bbar} ^M(k)
= \langle \phi^{D,M},\phi_k^M\rangle_{L_\rho^2}.
$$
Furthermore, the matrix form estimator in \eqref{eq:defphiMma} has the operator form 
\begin{equation}
\label{eq:defphiMop}    
\hat\phi_\lambda^{n,M} = (\mathcal{L}_{\Gbar^M} + \lambda \mathcal{L}_{\Gbar^M}^{-1})^{-1}\phi^{D,M} =  (\mathcal{L}_{\Gbar^M}^2 + \lambda I)^{-1}\mathcal{L}_{\Gbar^M}\phi^{D,M}, 
\end{equation}
where $\phi^{D,M}$ is the Riesz representation of the bounded linear functional:  
\begin{equation}
\label{eq:RieszM}
\begin{aligned} 
\left\langle \phi^{D,M},\phi \right\rangle_{L_\rho^2}
& = \frac{1}{N}\sum_{i=1}^N\sum_{m=1}^M 
  R_\phi^M[p_{t_i}](x_m)\tilde f_{t_i}(x_m)\Delta x, \, \\
R_\phi^M\left[p_{t_i}\right](x):&=\sum_{\ell=1}^n \int_0^{R_0} Q\left[p_{t_i}\right](x, r_{\ell})\mathbf{1}_{I_{\ell}}(r) \phi(r) d r.
\end{aligned}
\end{equation}
\end{lemma}

Moreover, we impose a source condition on the true L\'evy density $\phi^*$ relative to the normal operator $\mathcal{L}_{\Gbar}$, with $\beta$ quantifying its smoothness and thereby governing the convergence rates. 
\begin{assumption}[Source condition]
\label{H:source_condi}
$\phi^*=\mathcal{L}_{\Gbar}^{\beta / 2} w$ for some $\beta \geq 0, w \in L_\rho^2$ with $\|w\|_{L_\rho^2} \leq R$. 
\end{assumption}

We assume that the eigenvalues of the normal operator $\LGbar$ decays polynomially. This assumption is supported by the regularity conditions on $p_{t_i}$ in Assumption \ref{H:boundf} and Theorem 1.1 in \cite{volkov2024optimal}, which together imply that $\LGbar$ has polynomial spectral decay.  
\begin{assumption}[Spectral decay]
\label{H:spectral_decay}
There exists $a\geq 0$ and $b>0$ such that the eigenvalues $\lambda_k$ of $\LGbar$ in \eqref{eq:LGbar} satisfy 
\[
ak^{-2\varsigma} \leq \lambda_k \leq b k^{-2\varsigma}, \forall k.  
\] 
\end{assumption}
Here, the polynomial spectral decay exponent $\varsigma$ quantifies the ill-posedness of the inverse problem, and a larger $\varsigma$ indicates a more severe ill-posedness.

We quantify the complexity of the inverse problem at scale~$\lambda$ through a regularized effective dimension. In analogy with classical kernel regression (see, e.g., \cite{caponnetto2007optimal,zhang2005learning}), but adapted here to the normal operator $\mathcal{L}_{\bar{G}}$ and the DA-RKHS regularization, we define
\begin{equation}\label{eq:defNeff}
\mathcal{N}(\lambda)
:= \operatorname{Tr}\bigl(\mathcal{L}_{\bar{G}}(\mathcal{L}_{\bar{G}}+\lambda \mathcal{L}_{\bar{G}}^{\dagger})^{-1}\bigr)
= \operatorname{Tr}\bigl(\mathcal{L}_{\bar{G}}^2(\mathcal{L}_{\bar{G}}^2+\lambda I)^{-1}\bigr).
\end{equation}
Under Assumption~\ref{H:spectral_decay}, $\mathcal{N}(\lambda)$ grows polynomially as $\lambda \downarrow 0$ with a rate $O(\lambda^{-\frac{1}{4\varsigma}})$  (see Lemma \ref{lemma:effdim_Bd}). This quantity will play a central role in the convergence analysis of our RKHS-regularized estimator, stated in Theorem~\ref{thm:conv}, by governing the trade-off between approximation and regularization errors as the mesh is refined. 

Our main result is the following convergence theorem for the DA-RKHS regularized estimator. The proof is given in Section \ref{sec:error_decomposition}.   
\begin{theorem}\label{thm:conv}
Under Assumptions \ref{H:boundf}, \ref{H:constant} and \ref{H:source_condi} and with $\mathcal{N}(\lambda)$ in \eqref{eq:defNeff}, for $\varsigma > 1/4$ and $0<\lambda \leq 1$,
the estimator in \eqref{eq:defphiMma} has the error bound
$$
\begin{aligned}
\left\|\hat{\phi}_\lambda^{n, M}-\phi^*\right\|_{L_\rho^2} 
\leq & C\left(\beta, C_1\right) R \lambda^{\min \{\beta / 4,1\}} + C_{a}\bigl(\Delta t+\Delta x+\Delta x^2\bigr) \left( \sqrt{\frac{\mathcal{N}(\lambda)}{\lambda}} + \sqrt{\frac{\Delta x}{\lambda^2}}\right) \\
 + & C_{b}\,\frac{\Delta x}{\lambda}
+ C_{d}\,\frac{\Delta x^2}{\lambda}\sqrt{\frac{\mathcal{N}(\lambda)}{\lambda}}
+ C_e \frac{\Delta x}{{\lambda}^{1/2}} +C_f \frac{\Delta x ^2}{\lambda^{3/2}}+ C_{g}\,\frac{\Delta x^{5/2}}{\lambda^2}
\end{aligned}
$$
with constants $ C\left(\beta, C_1\right), C_{a}, C_{b}, C_{c}, C_{d}, C_{e},C_{f},C_g$ independent of $\lambda, \Delta x, \Delta t$. 

Consequently, with $\Delta t \leq \Delta x$ and Assumption \ref{H:spectral_decay}, taking the optimal regularization  hyperparameter $\lambda ^* \asymp(\Delta x)^\alpha$ with $\alpha=\frac{1}{\min \{\beta / 4,1\}+1}$, 
we have, as $\Delta x\to 0$,  
$$
\left\|\hat{\phi}_\lambda^{n, M}-\phi^*\right\|_{L^2_{\rho}} \lesssim(\Delta x)^{\alpha \min \{\beta / 4,1\}}, \quad  \alpha \min \{\beta / 4,1\}= \begin{cases}\frac{\beta}{ \beta +4}, & 0<\beta \leq 4;  \\ \frac{1}{2}, & \beta>4. \end{cases}
$$ 
\end{theorem}

The error comes from three sources: the regularization error (the term $O(\lambda^{\min \{\beta / 4,1\}})$), the numerical error and the finite-dimensional RKHS approximation error (the rest terms); see Section \ref{sec:error_decomposition}.  Here, the numerical error arises from the discretization of the integrals in the normal matrix $\mathbf{\Abar}$ and normal vector $\mathbf{\bbar}$ in \eqref{eq:regularkernelloss} via $\mathbf{\Abar}^M$ and $\mathbf{\bbar} ^M$. In particular, the dominating term $O(\Delta x/\lambda)$ comes from the perturbation in $\mathcal{L}_{\bar{G}^M}$ due to the Riemann-sum discretization in $x$ with mesh $\Delta x$, which is amplified by the regularized inverse $(\mathcal{L}_{\bar{G}^M}^2+\lambda I)^{-1}$ with norm bounded by $1/\lambda$. One may reduce this numerical error by using higher-order quadrature rules for the integrals in $\mathbf{\Abar}$ and $\mathbf{\bbar}$, which would lead to higher-order convergence in $\Delta x$.
The finite-dimensional RKHS approximation error stems from approximating the infinite-dimensional DA-RKHS $H_{\Gbar}$ by its finite-dimensional subspace $H_{\Gbar^n}$, which introduces additional errors of order  
$O(\Delta x\,\sqrt{\frac{\mathcal{N}(\lambda)}{\lambda}} + 
 \frac{\Delta x^{3/2}}{\lambda}\;+ \frac{\Delta x}{\lambda^{1/2}})$. Note that the approximation error and the numerical error are both controlled by the mesh size $\Delta x$ since the mesh $\Delta r =\Delta x$.

 The regularization error decreases as $\lambda$ decreases, while the numerical and approximation errors increase as $\lambda$ decreases, which motivates balancing the bias $O\left(\lambda^{\min \{\beta / 4,1\}}\right)$ against the dominant term $O\left(\Delta x /\lambda\right)$ from the numerical error. This balance yields the optimal regularization parameter $\lambda^* \asymp (\Delta x)^{\frac{1}{\min \{\beta / 4,1\}+1}}$ and the convergence rates stated in Theorem \ref{thm:conv}.  Notably, when the true kernel is sufficiently smooth with $\beta>4$, the convergence rate saturates at $O\left((\Delta x)^{1/2}\right)$, reflecting the limited ability of the RKHS to regularize overly smooth kernels in this inverse problem setting. 

\begin{remark}[Relation to minimax rates]
\label{rem:minimax}
When the numerical error is negligible when using high-order quadrature, the error bound in Theorem \ref{thm:conv} would lead to a convergence rate close to the minimax optimal rate for the statistical inverse problems of learning kernels in operators with DA-RKHS regularization in {\rm\cite{zhang2025minimax}} when these rates are adapted to our deterministic setting. Specifically, ignoring the numerical error terms in Theorem \ref{thm:conv}, we have
$$
\left\|\hat{\phi}_\lambda^{n, M}-\phi^*\right\|_{L^2_{\rho}} 
\lesssim (\Delta x)^{\alpha \min \{\beta / 4,1\}}, 
\quad  
\alpha \min \{\beta / 4,1\}
= \begin{cases}
    \frac{2 \beta \varsigma}{2 \beta \varsigma+4 \varsigma+1}, & 0<\beta \leq 4;  \\ 
       \frac{8 \varsigma}{12 \varsigma+1}, & \beta>4 .
    \end{cases}
$$ 
These rates come from the trade-off between the regularization bias $O\left(\lambda^{\min \{\beta / 4,1\}}\right)$ and the variance-like term $O\left(\Delta x \sqrt{\frac{\mathcal{N}(\lambda)}{\lambda}}\right) = O\left(\Delta x  \lambda^{-(1+4\varsigma) /(8\varsigma)} \right) $, which leads to $\lambda^*= O\left((\Delta x)^{\frac{8 \varsigma}{\min \{\beta/4, 1\} 8 \varsigma+1+4 \varsigma}}\right)$. When $\beta\leq 4$, the above rate is close to the minimax rate for $ O\left((\frac{1}{\sqrt{M}})^{-\frac{2 \beta \varsigma}{2 \beta \varsigma+2 \varsigma+1} } \right) $ with $M$ being the samples size in {\rm\cite{zhang2025minimax}}, where the exponent differs slightly due to the different nature of the variance-like term in our deterministic setting compared to the statistical setting there. 
\end{remark}

\begin{remark}[Regularity of $p_{t_i}$ and uniform bounds for $f_{t_i}$]\label{rem:regularity}
Under Assumption~\ref{H:boundf}, for each $t_i>0$ the map $x \mapsto p_{t_i}(x)$ is Lipschitz on $\Omega_{R_0}$ with a constant independent of $i$. In particular, letting $C'_{\max}:=\max_{1\le i\le N}\|\partial_x p_{t_i}\|_{L^\infty(\Omega_{R_0})},$ we have
$$
|p_{t_i}(x+u)-p_{t_i}(x)| \le C'_{\max}\,|u|\quad\text{for all }x,u\text{ with }x,x+u\in\Omega_{R_0},
$$
since $p_{t_i}\in W^{1,\infty}(\Omega_{R_0})$. Moreover, for $f_t(x)$ defined in \eqref{eq:fdata}, the product rule and Assumption~\ref{H:boundf} imply that there exist constants $C_0,C_1<\infty$, independent of $i$, such that
$$
\|f_{t_i}\|_{L^\infty(\Omega_{R_0})}\le C_0,\qquad
\|\partial_x f_{t_i}\|_{L^\infty(\Omega_{R_0})}\le C_1.
$$
These constants depend only on $\|b\|_{W^{2,\infty}}$, $\|\sigma^2\|_{W^{3,\infty}}$, and the uniform $W^{3,\infty}$ in $x$ and $W^{1,\infty}$ in $t$ bounds of the densities, but not on the index $i$.
\end{remark}

\subsection{Error decomposition and proof of the main theorem}\label{sec:error_decomposition}
We decompose the estimator's error into three parts: regularization error, approximation error due to a finite-dimensional hypothesis space, and numerical error arising from the use of discrete data to approximate integrals. That is, 
$$
\left\|\hat{\phi}_\lambda^{n, M}(r)-\phi^*(r)\right\|_{L_\rho^2} \leq \underbrace{\left\|\hat{\phi}_\lambda^{n, M}(r)-\hat{\phi}_\lambda^n(r)\right\|_{L_\rho^2}}_{\text {numerical error }}+\underbrace{\left\|\hat{\phi}_\lambda^n(r)-{\phi}_\lambda^{\infty}(r)\right\|_{L_\rho^2}}_{\text {approximation error }}+\underbrace{\left\|{\phi}_\lambda^{\infty}(r)-\phi^*(r)\right\|_{L_\rho^2}}_{\text {regularization error }}, 
$$
where the three estimators are defined as 
\begin{equation}
\label{eq:defphi}
\begin{aligned}
    {\phi}_\lambda^{\infty} :=&\left(\mathcal{L}_{\Gbar}^2+\lambda I\right)^{-1} \mathcal{L}_{\Gbar}^2\phi_*, \\
 	   \hat{\phi}_\lambda^n :=&\left(\mathcal{L}_{\Gbar^n}^2+\lambda I\right)^{-1} \mathcal{L}_{\Gbar^n} \phi^{D,n}, \\
 	   \hat{\phi}_\lambda^{n,M} =& (\mathcal{L}_{\Gbar^M}^2 + \lambda I)^{-1}\mathcal{L}_{\Gbar^M}\phi^{D,M}. 
\end{aligned}
 \end{equation}
Here, the operator $\mathcal{L}_{\Gbar^n}$ is the integral operator whose reproducing kernel $ \Gbar^n$ is a piecewise constant approximation of $\Gbar$,  
  \begin{equation}\label{eq:Gbar^n}
    \Gbar^n(r, s)=\sum_{\ell^{'}, \ell=1}^n \Gbar \left(r_{\ell^{'}}, r_{\ell}\right) \mathbf{1}_{I_{\ell^{'}}}(r) \mathbf{1}_{I_{\ell}}(s), \quad I_l:=((l-1) \Delta r, l \Delta r],  
   \end{equation}
 and the function $\phi^{D, n}$ is the Riesz representation of the bounded linear functional  
\begin{equation}
\label{eq:Rieszn}
\left\langle\phi^{D, n}, \phi\right\rangle_{L^2_{\rho}} = \frac{1}{N} \sum_{i=1}^N \int R^n_{\phi}\left[p_{t_i}\right](x) f_{t_i}(x) d x, \quad 
R_\phi^n\left[p_{t_i}\right](x):=\sum_{\ell=1}^n \int_\delta^{R_0} Q\left[p_{t_i}\right](x, r_{\ell})\mathbf{1}_{I_{\ell}}(r) \phi(r) d r. 
\end{equation}
In other words, ${\phi}_\lambda^{\infty}$ is the regularized estimator in the continuous space $H_{\Gbar}$; $\hat{\phi}_\lambda^n$ is an estimator in the finite-dimensional RKHS $H_{\Gbar^n}$ when $\Gbar$ is approximated by a piecewise constant function $\Gbar^n$ in $r$; and $\hat\phi_\lambda^{n,M}$ is the fully discrete regularized estimator based on $\Gbar^M$, which is discretized in both $r$ and $x$.

We bound these errors by the lemmas below, whose proofs are postponed to Appendix~\ref{sec:tech_lemmas}.

\begin{lemma}[Regularization Error] \label{lem: regerror}
Under Assumption \eqref{H:source_condi} on the source and Assumptions \ref{H:boundf} which implies the operator bound for $\mathcal{L}_{\bar{G}}$ in Lemma \ref{Lem: HSoperator}, the regularization error satisfies
$$
\left\|{\phi}_\lambda^{\infty}-\phi^*\right\|_{L_\rho^2} \leq C\left(\beta, C_1\right)  R  \lambda^{\min \{\beta / 4,1\}},
$$
where $C\left(\beta, C_1\right) = \left(\frac{\beta}{4}\right)^{\beta / 4}\left(1-\frac{\beta}{4}\right)^{1-\beta / 4}$ if $0<\beta\leq 4$ and $C\left(\beta, C_1\right)= C_1^{\frac{\beta}{2}-2}$ if $\beta>4$. 

\end{lemma}

\begin{lemma}[Approximation Error] \label{lem:approxerror}
Under Assumptions \eqref{H:source_condi} and \ref{H:constant} on the source, for $0<\lambda \leq 1$,  the error between ${\phi}_\lambda^{\infty}$ and $\hat{\phi}_\lambda^n$ defined \eqref{eq:defphi} has a bound 
$$
\left\|\hat{\phi}_\lambda^n-\phi_\lambda^{\infty}\right\|_{L_\rho^2} \leq C_L \Delta x \sqrt{\frac{\mathcal{N}(\lambda)}{\lambda}}+C_L \sqrt{2 C_1 C_2} \frac{\Delta x^{3 / 2}}{\lambda}+C_\beta\frac{\Delta x}{\sqrt{\lambda}},
$$
where $C_L = 2 C_0\sqrt{C}\left|\Omega_{R_0}\right| C'_{\max} \sqrt{\frac{R_0}{\rho_{0}}}$ and $C_\beta: = C_2 C_1^{\beta / 2} R + \frac{C_2}{2} C\left(\beta, C_1\right)R$. 
\end{lemma}

\begin{lemma}[Numerical error]\label{lem:numerroroperator}
Under Assumption \ref{H:boundf} on $p_{t_i}$, the fully discrete estimator
$\hat{\phi}_\lambda^{n,M}$ satisfies, for all $\lambda>0$,
$$
\begin{aligned}
\|\hat{\phi}_\lambda^{n,M}-\hat{\phi}_\lambda^n\|_{L^2_\rho}
&\le \|T_1\|_{L^2_\rho} + \|T_2\|_{L^2_\rho} + \|T_3\|_{L^2_\rho} \\
& \leq K_1\left(\Delta t+\Delta x+\Delta x^2\right)\left(\sqrt{\frac{\mathcal{N}(\lambda)}{\lambda}}+\sqrt{\frac{\Delta x}{\lambda^2}}\right) \\ & +K_2 \frac{\Delta x}{\lambda}+K_3 \frac{\Delta x^2}{\lambda} \sqrt{\frac{\mathcal{N}(\lambda)}{\lambda}}+K_4 \frac{\Delta x^{5 / 2}}{\lambda^2} + K_5\frac{\Delta x^{2}}{\lambda^{3/2}}
\end{aligned}
$$
Here, constants $K_i>0$ are independent of $\Delta x,\,\Delta t,\,\mathcal{N}(\lambda)$ and $\lambda$.
\end{lemma}


\begin{proof}[Proof of Theorem \ref{thm:conv}]
According to Lemma \ref{lem: regerror}, Lemma \ref{lem:approxerror} and Lemma \ref{lem:numerroroperator}, we get
$$
\begin{aligned}
\left\|\hat{\phi}_\lambda^{n, M}-\phi^*\right\|_{L_\rho^2} \leq & \left\|\hat{\phi}_\lambda^{n, M}-\hat{\phi}_\lambda^n\right\|_{L^2_{\rho}}+\left\|\hat{\phi}_\lambda^n-\hat{\phi}_\lambda^{\infty}\right\|_{L_\rho^2}+\left\|\phi_\lambda^{\infty}-\phi^*\right\|_{L_\rho^2} \\
\leq & C\left(\beta, C_1\right) R \lambda^{\min \{\beta / 4,1\}} + C_L \Delta x \sqrt{\frac{\mathcal{N}(\lambda)}{\lambda}}+C_L \sqrt{2 C_1 C_2} \frac{\Delta x^{3 / 2}}{\lambda}+C_\beta\frac{\Delta x}{\sqrt{\lambda}}\\
 + & K_1\bigl(\Delta t+\Delta x+\Delta x^2\bigr) \left( \sqrt{\frac{\mathcal{N}(\lambda)}{\lambda}} + \sqrt{\frac{\Delta x}{\lambda^2}}\right) + K_2\,\frac{\Delta x}{\lambda}
+ K_3 \frac{\Delta x^2}{\lambda} \sqrt{\frac{\mathcal{N}(\lambda)}{\lambda}}\\
& + K_4 \frac{\Delta x^{5 / 2}}{\lambda^2}+K_5 \frac{\Delta x^2}{\lambda^{3 / 2}}
\end{aligned}
$$
By Lemma \ref{lemma:effdim_Bd} and the spcetral decay condition \ref{H:spectral_decay},  the effective dimension exhibits polynomial decay $\mathcal{N}(\lambda) \asymp \lambda^{-1 /(4 \gamma)} $, so 
$\sqrt{\frac{\mathcal{N}(\lambda)}{\lambda}} \asymp \lambda^{- \frac{1+ 4 \gamma}{8 \gamma}} $. 
Ignoring constant factors, with
sufficiently small mesh $\Delta x$ and under a suitable Courant-Friedrichs-Lewy condition $\Delta t \leq \Delta x^2$, all higher-order terms in $\Delta x$ are absorbed into the leading $\Delta x$-order contributions and the bound reduces to
$$
\left\|\hat{\phi}_\lambda^{n, M}-\phi^*\right\|_{L_\rho^2} \lesssim \lambda^{\min \{\beta / 4,1\}}+\Delta x \lambda^{-\frac{1+4 \gamma}{8 \gamma}}+\Delta x\lambda^{-1/2}+\Delta x\lambda^{-1}.
$$
Since $\gamma>1/4$ , $\{-\frac{1+4 \gamma}{8 \gamma}, -1/2\}>-1 $, so the $\Delta x / \lambda$ term dominates the $\Delta x$-dependent contributions. Consequently, the leading trade-off is between the bias and $\Delta x / \lambda$, namely
$$
\lambda^{\min \{\beta / 4,1\}} \sim \Delta x \lambda^{-1}
$$
Hence, the optimal regularization hyperparameter $\lambda^*$ has
$$
\lambda^* \sim(\Delta x)^\alpha, \;\text { with } \; \alpha=\frac{1}{\min \{\beta / 4,1\}+1} = \begin{cases}\frac{4}{\beta+4}, & 0<\beta \leq 4 \\ \frac{1}{2}, & \beta>4\end{cases}
$$
Then we obtain the convergence rate
$\left\|\hat{\phi}_\lambda^{n, M}-\phi^*\right\|_{L_o^2} \lesssim(\Delta x)^{\alpha \min \{\beta / 4,1\}}$.
\end{proof}

\section{Hyperparameter selection via Bi-level optimization}
\label{secbilevel}

Selection of the regularization parameter is crucial for obtaining a stable solution for this ill-posed inverse problem. An overly large parameter can cause underfitting, leading to inaccurate solutions, while an excessively small parameter may result in overfitting, producing unstable solutions. While classical methods such as Generalized Cross-Validation (GCV) \cite{golub1979generalized} and the L-curve \cite{hansen1992analysis} are widely used, they may suffer from numerical instability \cite{lu2022data,li2025automatic}.

To address these limitations, we investigate a bi-level optimization \cite{chada2022consistency,holler2018bilevel} that selects the hyperparameter by 
     minimizing generalization error on validation data. This approach ensures robust parameter selection and can be easily extended to the simultaneous tuning of multiple hyperparameters.    
 Its ``upper-level'' selects the hyperparameter by minimizing the validation error, and the ``lower-level'' solves the regularized regression problem. 
 From \eqref{eq:discreteloss}, the bi-level optimization reads:
\begin{equation}
\label{eq:bilevel}
\begin{aligned}
\text{Upper level:} \quad \gamma_{\text {opt }} & =\underset{\gamma}{\operatorname{argmin}} F\left(\gamma\right), \quad F\left(\gamma\right):= \mathcal{E}(\mathbf{c}(\gamma)) = \| \mathbf{Q_U}\bar{\mathbf{G}}\mathbf{c} (\gamma) (\Delta r) - \mathbf{f_U}\|^2;  \\
\text { Lower level: }\quad  \mathbf{c}(\gamma) & =\underset{\mathbf{c}}{\operatorname{argmin}} \Psi\left(\gamma,\mathbf{c}\right), \quad 
\Psi\left(\gamma,\mathbf{c}\right) := \| \mathbf{Q_L}\bar{\mathbf{G}}\mathbf{c} (\Delta r) - \mathbf{f_L}\|^2+ 10^{\gamma} \| \mathbf{c}^\top  \sqrt{\bar{\mathbf{G}}}\|^2.
\end{aligned}
\end{equation}
Here, we divide the data into a training dataset and a validation dataset by collecting samples at distinct times, and denote them by $\left(\mathbf{f_L}, \mathbf{Q_L}\right)$ and $\left(\mathbf{f_U}, \mathbf{Q_U}\right)$, respectively. We parameterize $\lambda$ as $\lambda=10^\gamma$ to accommodate the typical scale of $\lambda$ and to optimize over an unconstrained variable. We refer to the outer minimization as the upper-level (UL) problem and the inner minimization as the lower-level (LL) problem.

The lower-level optimization problem is a Tikhonov-regularized least-squares problem. 
Since \(\Psi(\gamma,\mathbf{c})\) is quadratic in \(\mathbf{c}\), it admits the explicit solution
\begin{equation}
\label{eq:LLOP}
\mathbf{c}(\gamma)
= \big[(\Delta r)^{\top}\bar{\mathbf{G}}^{\top}\mathbf{Q}_{\mathbf{L}}^{\top}\mathbf{Q}_{\mathbf{L}}\bar{\mathbf{G}}(\Delta r)
+ 10^{\gamma}\bar{\mathbf{G}}\big]^{-1}\bar{\mathbf{G}}\mathbf{Q}_{\mathbf{L}}^{\top}\mathbf{f}_{\mathbf{L}}\,\Delta r .
\end{equation}
For any fixed regularization hyperparameter, it has a unique minimizer.

The upper-level optimization is solved by gradient descent with the update
\begin{equation}
\label{eq:gradient}
\gamma_{k+1}=\gamma_k-\eta_k\,\partial_\gamma F(\gamma_k),
\qquad 
\end{equation}
where $\eta_k= \eta/\sqrt{k}$ is the learning rate \cite{kingma2014adam}, with $\eta$ the initial learning rate and $k$ the iteration index.
By the chain rule,
\begin{equation}
\label{eq:chain}
\partial_\gamma F(\gamma)=\big(\partial_{\mathbf{c}}\mathcal{E}(\mathbf{c}(\gamma))\big)^{\top}\partial_\gamma \mathbf{c}(\gamma).
\end{equation}
Thus, the key task is to evaluate $\partial_\gamma \mathbf{c}(\gamma)$. Existing approaches fall into two classes \cite{liu2021investigating}: (i) explicit gradients, obtained by automatic differentiation through all lower-level iterations, which are memory-intensive; and (ii) implicit gradients, based on the implicit function theorem and the inverse Hessian, typically approximated by a Neumann series \cite{lorraine2020optimizing} or conjugate gradients \cite{rajeswaran2019meta}. In our setting, the inverse Hessian can be expressed in closed form using the generalized singular value decomposition (GSVD) \cite{hansen1994regularization}, which we already compute to solve the lower-level problem.

For convenience, set
\begin{equation}
\label{eq:ALU}
\bar{\mathbf{A}}_{\mathbf{L}} \doteq \mathbf{Q}_{\mathbf{L}}\,\bar{\mathbf{G}}\,(\Delta r),
\quad
\bar{\mathbf{A}}_{\mathbf{U}} \doteq \mathbf{Q}_{\mathbf{U}}\,\bar{\mathbf{G}}\,(\Delta r), \quad \bar{\mathbf{K}}\doteq \sqrt{\bar{\mathbf{G}}}.
\end{equation}
Then \eqref{eq:LLOP} can be rewritten as
\begin{equation}
\label{eq:clower}
\mathbf{c}(\gamma)
= \big(\bar{\mathbf{A}}_{\mathbf{L}}^{\top}\bar{\mathbf{A}}_{\mathbf{L}}+10^{\gamma}\bar{\mathbf{G}}\big)^{-1}
\,\bar{\mathbf{G}}\,\mathbf{Q}_{\mathbf{L}}^{\top}\mathbf{f}_{\mathbf{L}}\,\Delta r
= \big(\bar{\mathbf{A}}_{\mathbf{L}}^{\top}\bar{\mathbf{A}}_{\mathbf{L}}+10^{\gamma}\bar{\mathbf{K}}^{\top}\bar{\mathbf{K}}\big)^{-1}
\,\bar{\mathbf{G}}\,\mathbf{Q}_{\mathbf{L}}^{\top}\mathbf{f}_{\mathbf{L}}\,\Delta r.
\end{equation}
Let the GSVD of the pair \((\bar{\mathbf{A}}_{\mathbf{L}},\bar{\mathbf{K}})\) be
\[
\bar{\mathbf{A}}_{\mathbf{L}}=\mathbf{U}_{\bar{\mathbf{A}}_{\mathbf{L}}}\,\Sigma_{\bar{\mathbf{A}}_{\mathbf{L}}}\,\mathbf{X}^{-1},
\qquad
\bar{\mathbf{K}}=\mathbf{V}_{\bar{\mathbf{K}}}\,\Sigma_{\bar{\mathbf{K}}}\,\mathbf{X}^{-1},
\]
where \(\mathbf{U}_{\bar{\mathbf{A}}_{\mathbf{L}}}\) and \(\mathbf{V}_{\bar{\mathbf{K}}}\) are orthogonal, \(\mathbf{X}\) is invertible, and 
\(\Sigma_{\bar{\mathbf{A}}_{\mathbf{L}}}=\mathrm{diag}(\sigma_i)\), \(\Sigma_{\bar{\mathbf{K}}}=\mathrm{diag}(\mu_i)\) satisfy \(\sigma_i^{2}+\mu_i^{2}=1\). Then the lower-level solution admits the filter-factor representation \cite[Eq.~(2.19)]{hansen1994regularization}
\begin{equation}
\label{eq:lowersolver}
\mathbf{c}(\gamma)
=\mathbf{X}\sum_{i=1}^{n}\frac{\sigma_i}{\sigma_i^{2}+10^{2\gamma}\mu_i^{2}}\,
\big(\mathbf{u}_i^{\top}\bar{\mathbf{G}}\,\mathbf{Q}_{\mathbf{L}}^{\top}\mathbf{f}_{\mathbf{L}}\big)\,e_i \Delta r,
\end{equation}
where \(\mathbf{u}_i\) is the \(i\)-th column of \(\mathbf{U}_{\bar{\mathbf{A}}_{\mathbf{L}}}\) and \(e_i\) is the \(i\)-th canonical basis vector.

To compute the hypergradient, we apply the implicit function theorem to the lower-level optimality condition and obtain
\begin{equation}
\label{eq:cgamma}
\partial_{\gamma}\mathbf{c}(\gamma)
= -\big(\nabla_{\mathbf{c}}^{2}\Psi(\gamma,\mathbf{c})\big)^{-1}
\,\partial_{\gamma\mathbf{c}}^{2}\Psi(\gamma,\mathbf{c}).
\end{equation}
Combining \eqref{eq:chain} and \eqref{eq:cgamma} and using the definitions in \eqref{eq:ALU} yield
\begin{equation}
\label{partialFim}
\partial_{\gamma}F(\gamma)
= -\,2\ln(10)\cdot 10^{\gamma}\,
\big(\bar{\mathbf{A}}_{\mathbf{U}}^{\top}\bar{\mathbf{A}}_{\mathbf{U}}\mathbf{c}-\bar{\mathbf{A}}_{\mathbf{U}}^{\top}\mathbf{f}_{\mathbf{U}}\big)^{\top}
\big(\bar{\mathbf{A}}_{\mathbf{L}}^{\top}\bar{\mathbf{A}}_{\mathbf{L}}+10^{\gamma}\bar{\mathbf{K}}^{\top}\bar{\mathbf{K}}\big)^{-1}
\bar{\mathbf{G}}\,\mathbf{c}.
\end{equation}
The matrix inverse in \eqref{partialFim} can be expressed via the GSVD already computed for the lower-level problem. Indeed,
\begin{equation}
\label{eq:outerinverse}
\big(\bar{\mathbf{A}}_{\mathbf{L}}^{\top}\bar{\mathbf{A}}_{\mathbf{L}}+10^{\gamma}\bar{\mathbf{K}}^{\top}\bar{\mathbf{K}}\big)^{-1}
=\mathbf{X}\,\big(\Sigma_{\bar{\mathbf{A}}_{\mathbf{L}}}^{\top}\Sigma_{\bar{\mathbf{A}}_{\mathbf{L}}}
+10^{\gamma}\Sigma_{\bar{\mathbf{K}}}^{\top}\Sigma_{\bar{\mathbf{K}}}\big)^{-1}\mathbf{X}^{\top},
\end{equation}
where the middle matrix is diagonal, so its inverse is obtained by inverting its diagonal entries. Substituting \eqref{eq:outerinverse} into \eqref{partialFim} yields the desired gradient for the upper-level update. 

To improve numerical stability and convergence speed, we use a momentum scheme with Nesterov acceleration \cite{nesterov1983method} to update $\gamma$ and apply gradient normalization \cite{baydin2018online} to mitigate potential gradient explosion. The resulting upper-level update is
\begin{equation}\label{eq:gradientnew}
\gamma_{k+1}=\gamma_k-v_{k+1},\qquad 
v_{k+1}=\iota  v_k+\eta_k\frac{\nabla_\gamma F(\gamma_k-\iota  v_k)}{\|\nabla_\gamma F(\gamma_k-\iota  v_k)\|+\epsilon},
\end{equation}
where $F(\gamma)$ is the upper-level objective, $k$ is the iteration index, $v_k$ is the momentum variable with coefficient $\iota  \in[0,1)$, and $d_k=\nabla_\gamma F\left(\gamma_k-\iota  v_k\right)$ is the look-ahead gradient. Here $\eta_k$ is the learning rate, $\|\cdot\|$ is the Euclidean norm, and $\epsilon>0$ ensures numerical stability in the gradient normalization.

We summarize the complete procedure in Algorithm~\ref{alg:one} and refer to it as \emph{bilevel-RKHS}.

\SetKwInOut{KwData}{Initialize}
\begin{algorithm}
\caption{GSVD-based bi-level optimization (bilevel-RKHS)}\label{alg:one}
\KwIn{Sample $\left\{x_j\right\}_{j=1}^J$ and $\left\{\left(p_{t_i}\left(x_j\right), f_{t_i}\left(x_j\right)\right): j=1, \ldots, J\right\}_{i=1}^N$.}
\KwData{initial learning rate $\eta >0$, momentum coefficient $\iota \in(0,1)$, max iterations $K$, tolerance $\epsilon>0$; set $v_0 = 0$, $\gamma_0 = 0$}
\textbf{Generate Dataset:} Get $(\bar{\mathbf{A}}_\mathbf{L}, \bar{\mathbf{A}}_\mathbf{U}, \mathbf{f_U}, \mathbf{f_L}, \bar{\mathbf{G}})$ via equation \eqref{eq:matrixexpression}; set $\bar{\mathbf{K}} = \sqrt{\bar{\mathbf{G}}}$.\\

$(\mathbf{U},\texttt{sm},\mathbf{X}) = \mathrm{gsvd}(\bar{\mathbf{A}}_\mathbf{L},\bar{\mathbf{K}})$\\

\For{$k=1, \ldots, K$: }{
    $\eta_k = \eta/\sqrt{k}$\\
    $\gamma=\gamma_k-\iota v_k$\\
   \textbf{Obtain the explicit solution of Lower-Level problem:}\\
   $\mathbf{c}(\gamma) =  \mathrm{tikhonov}(\mathbf{U},\texttt{sm},\mathbf{X},\mathbf{f_L},10^{\gamma})$ with equation \eqref{eq:lowersolver}\\
  \textbf{Solve the Upper-Level problem with a gradient descent method:}\\
  $v_{k+1} = \iota v_k+\eta_k \frac{\frac{d F}{d \gamma}}{\left\|\frac{d F}{d \gamma}\right\|+\epsilon}$ using equation \eqref{partialFim} \\ $\gamma_{k+1}=\gamma_k-v_{k+1}$\\ 
  }
$\lambda = 10^{\gamma_{K+1}} , \hat{\mathbf{c}} =  \mathrm{tikhonov}(\mathbf{U},\texttt{sm},\mathbf{X},\mathbf{f_L},\lambda )$, $\hat{\phi} = \bar{\mathbf{G}}\hat{\mathbf{c}}$\\
\KwOut{$\hat{\phi},\lambda$}
\end{algorithm}

In addition, to avoid unnecessary computation, we invoke an early-stopping criterion when updates of the hyperparameter $\gamma$ become negligible or the loss fails to decrease. Let $\left\{\gamma\right\}_{k \geq 1}$ denote the sequence of hyperparameters and $F\left(\gamma_k\right)$ the upper-level loss (treated as a validation loss) evaluated on the dataset $( \bar{\mathbf{A}}_\mathbf{U}, \mathbf{f_U}, \bar{\mathbf{G}})$. Given thresholds $\varepsilon_\gamma> 0, \varepsilon_F>0$ and a window length $W \in \mathbb{N}$, define $
\Delta \gamma_k:=\left|\gamma_k-\gamma_{k-1}\right|$ and $\Delta F_k:=\left|F\left(\gamma_k\right)-F\left(\gamma_{k-1}\right)\right| \quad(k \geq 1) $.
We stop at the first index
$$
K^{*}:=\inf \left\{k \geq W: \forall t \in\{k-W+1, \ldots, k\}, \Delta \gamma_t<\varepsilon_\gamma \text { and } \Delta F_t<\varepsilon_F\right\} .
$$
The window length $W$ guards against premature stopping and improves stability.

\section{Numerical experiments}
\label{sec4}
In this section, we present numerical experiments to assess the proposed bilevel-RKHS method for estimating the L\'evy density from two types of data. The first type is generated by numerical iterative schemes that solve the Fokker-Planck equation, allowing us to validate our method in a controlled setting, since the discretization error of these schemes is explicitly known. The second type is reconstructed from discrete trajectory data via kernel density estimation (KDE), illustrating the applicability of our method to real-world observations. The codes are available at: \href{https://github.com/senyuanya/bilevel-RKHS-main}{https://github.com/senyuanya/bilevel-RKHS-main}.

\paragraph{Performance evaluation:} To evaluate performance, we use the upper-level loss and the relative $L^2_\rho$ error, defined as:
\begin{equation}
\label{eq: LossError}
\begin{gathered}
\operatorname{Loss}=\| \mathbf{Q_U}\bar{\mathbf{G}}\mathbf{c} (\gamma) (\Delta r) - \mathbf{f_U}\|^2 \\
L^2_\rho\text { Error }=\sqrt{d r \cdot \rho(r) \cdot\left(\hat{\phi}^{n,M}_{\lambda}(r)-\phi^*(r)\right)^2}.
\end{gathered}
\end{equation}
Here, $\hat{\phi}^{n,M}_{\lambda} $ represents the estimator, while $\phi^*$ denotes true L\'evy density. 

We benchmark our approach against existing methods, including the L-curve and GCV, and evaluate different regularization norms within the bilevel framework.  First, we compare our RKHS–Bilevel method with RKHS–LC and RKHS–GCV, which select the regularization parameter via the L-curve and generalized cross-validation (GCV), respectively. Second, previous work shows that the RKHS norm outperforms other norms under both L-curve and GCV selection rules \cite{li2025automatic, lu2022data}. Accordingly, within our bilevel framework, we evaluate regularizers based on the RKHS norm, the $L^2_{\rho}$-norm, and the $\ell^2$-norm. Table \ref{tab:method} summarizes the methods and regularization norms considered when using $\Gbar$ in \eqref{eq:Gbar_basisFn} as the basis function. The numerical experiments also demonstrate convergence with respect to the spatial mesh size, consistent with Theorem \ref{thm:conv}.
\begin{table}[H]    
    \centering
    \caption{ Methods and Regularization Norms}
    \begin{threeparttable}
    
    \begin{tabular}{c|cc}
        \toprule
        Methods & Norms &$\|\cdot\|_*^2$  \\
        \midrule
        L-curve (LC) & $L^2_\rho$ &  $\mathbf{c}^{\top}\bar{\mathbf{G}} \operatorname{diag}(\boldsymbol{\rho}) \bar{\mathbf{G}}^{\top}\mathbf{c}$ \tnote{*}\\
  
        Generalized Cross‑Validation (GCV) & $\ell^2$ & $\mathbf{c}^{\top} \mathbf{I} \mathbf{c}$ \tnote{*}\\

        Bilevel optimization (Bilevel) & RKHS  & $\mathbf{c}^{\top}\bar{\mathbf{G}}\mathbf{c}$\\
        \bottomrule
    \end{tabular}
    \vspace{0cm}
     \begin{tablenotes}
        \footnotesize
        \item[*] We define: $\operatorname{diag}(\boldsymbol{\rho})=\operatorname{diag}\left(\Delta r \rho\left(r_1\right), \Delta r \rho\left(r_2\right), \ldots, \Delta r \rho\left(r_M\right)\right)$
        \item[*] Here, $\mathbf{I}$ denotes the identity matrix, and $\|\cdot\|_{\ell^2}$ represents the Euclidean norm.
      \end{tablenotes}
    \end{threeparttable}
      \label{tab:method}
\end{table}

\subsection{ Estimation from data consisting of PDE solutions}
\label{subSec:iterative}
We first consider the setting in which the data consists of numerical solutions of the Fokker--Planck equation at discrete times. The probability density $p^j_{i}= p(x_j,t_i)$ is generated by the finite-difference scheme
\begin{equation}
\label{eq: DFM}
\begin{aligned}
\frac{p_{i+1}^{j+1}-p_{i}^j}{\Delta t}= & -\frac{b\left(x_{j+1}\right) p_{i}^{j+1}-b\left(x_{j-1}\right) p_{i}^{j-1}}{2 \Delta \tilde{x}}+\frac{1}{2} \frac{p_{i}^{j+1}-2 p_{i}^j+p_{i}^{j-1}}{\Delta \tilde{x}^2} \\
& +\sum_{k=0}^M \phi\left(r_k\right)\left(p_{i}^{j+j_k}+p_{i}^{j-j_k}-2 p_{i}^j\right) \Delta \tilde{r},
\end{aligned}
\end{equation}
with time step $\Delta t = 0.000025$ and spatial steps $\Delta \tilde{x} = \Delta \tilde{r} = 0.005$ chosen to satisfy the CFL condition. We take $\Omega = [-5,5]$ and $R_0 = 2$, so the nonlocal term is evaluated only for $t \in [0,1]$ and $x_j \in [-3,3]$ with $0\leq j\leq M= R_0/\Delta \tilde{r}$, thereby avoiding truncation error of $p(x,t)$ in the regression. From the resulting numerical solution, we construct the dataset $\mathcal{D}$ in \eqref{eq: discretedata} on coarser meshes with $\Delta x \in \{0.01, 0.02, 0.025, 0.04\}$. In all experiments we use $N = 30$ uniformly spaced time snapshots. For the bilevel optimization, $N/2$ snapshots are used for training and $N/2$ for validation.

We consider two L\'evy measure densities and two drift functions:
\begin{itemize}
\item[-] \textbf{L\'evy Measure Density}: (a) $\phi(|y|)=e^{-|y|^2}$ (b)  $\phi(|y|) = e^{-2|y|}$
\item[-]  \textbf{Drift Function}: (a) $b(x) = -0.5x$ (b) $b(x) = \sin x$
\end{itemize}
The two light-tailed L\'evy densities with Gaussian and exponential decay reduce quadrature error when truncating to a compact domain and the globally Lipschitz drifts promote convergence of the scheme and stable numerical differentiation in the Fokker–Planck equation. In our experiments, we optimize using SGD with momentum, using  initial learning rates $\eta \in \{0.004,0.009\}$ with momentum coefficient $\iota =0.99$ for the linear-drift case and a learning rate $\eta=0.005$ with momentum coefficient $\iota=0.95$ for the nonlinear-drift case.

\begin{figure}[thb]
    \centering       \subfigure{\includegraphics[width=0.32\textwidth,height=0.25\textwidth]{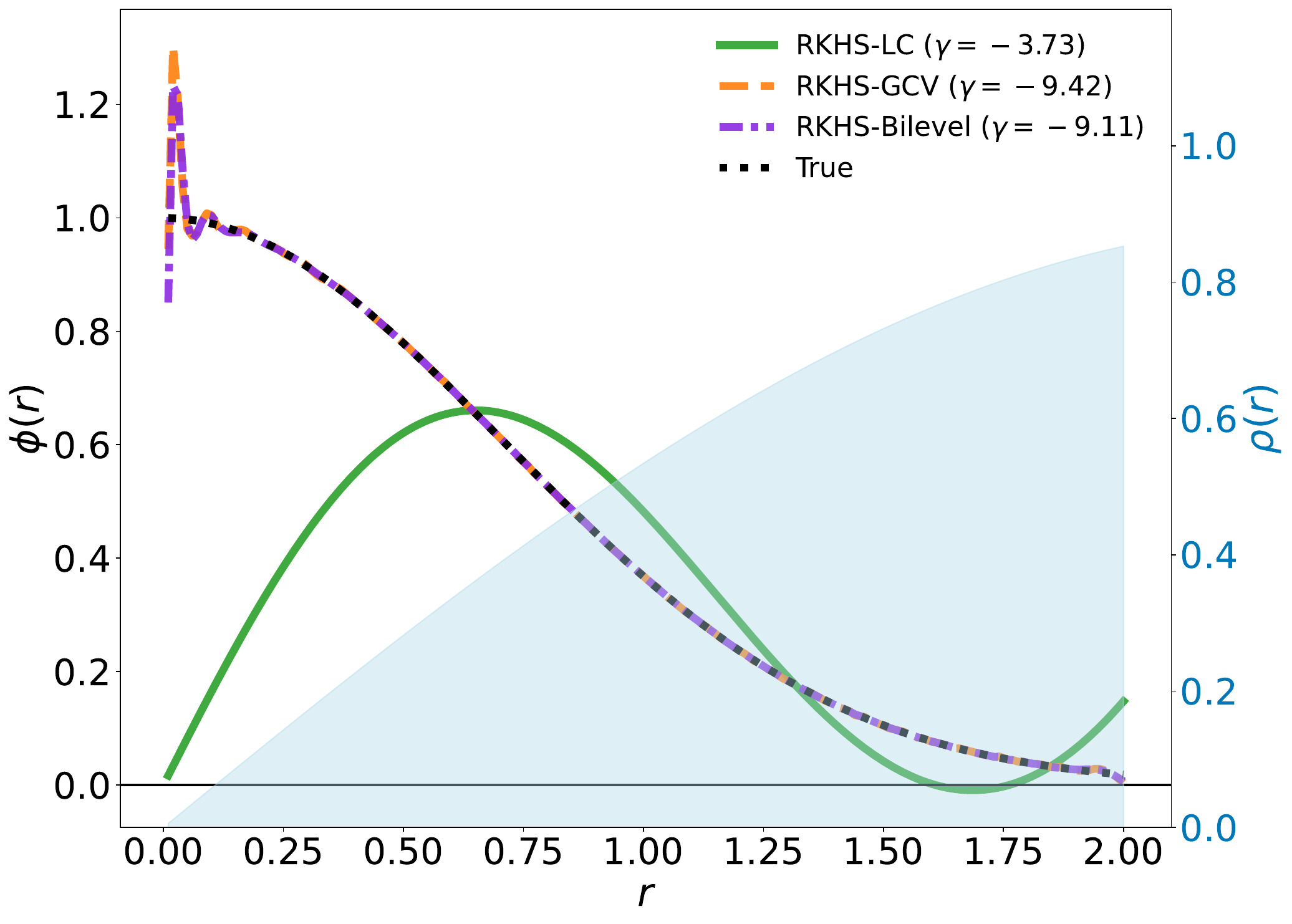}
    \label{fig:LG1}}     \subfigure{\includegraphics[width=0.32\textwidth,height=0.25\textwidth]{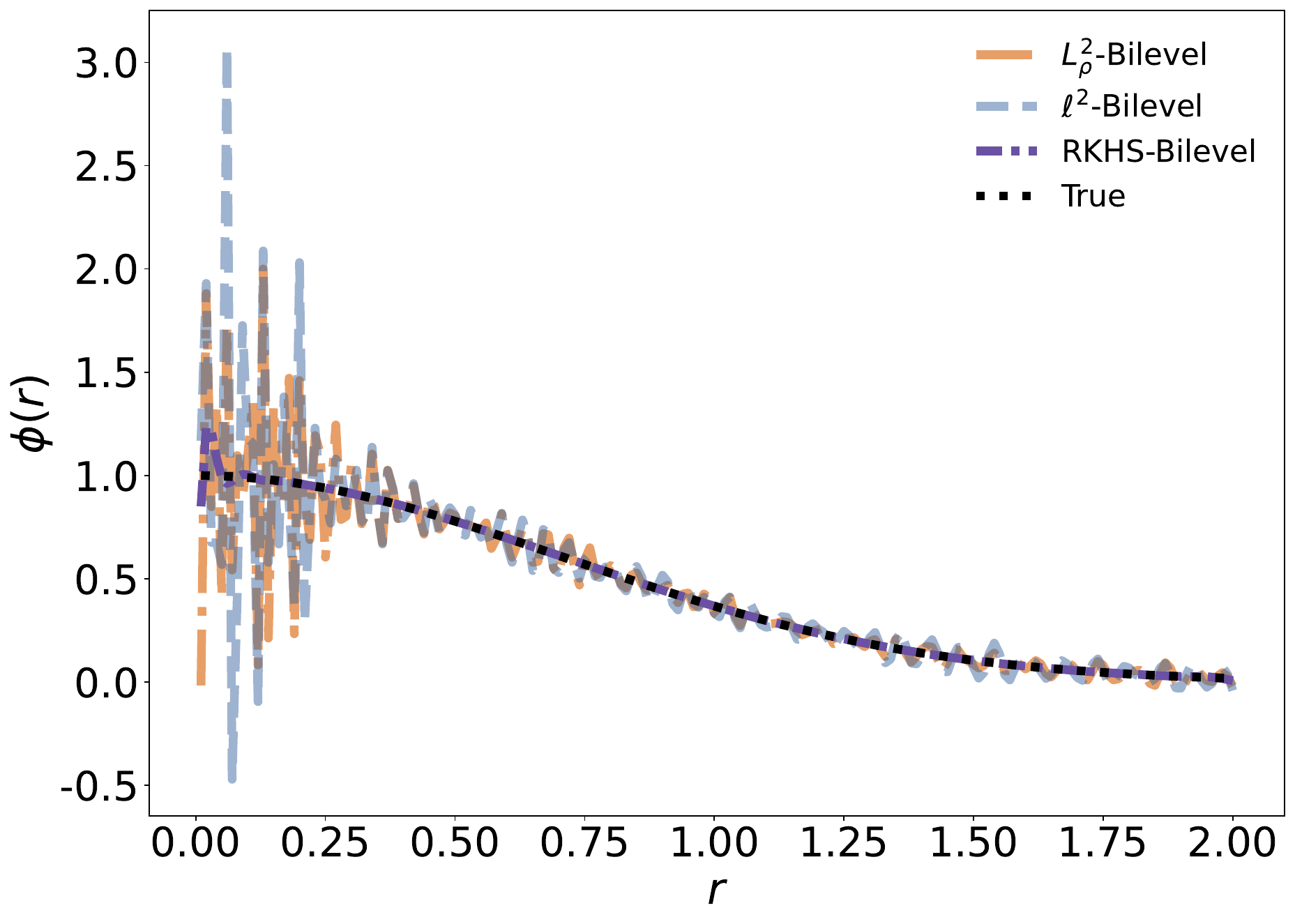}
    \label{fig:LG2}}    \subfigure{\includegraphics[width=0.32\textwidth,height=0.25\textwidth]{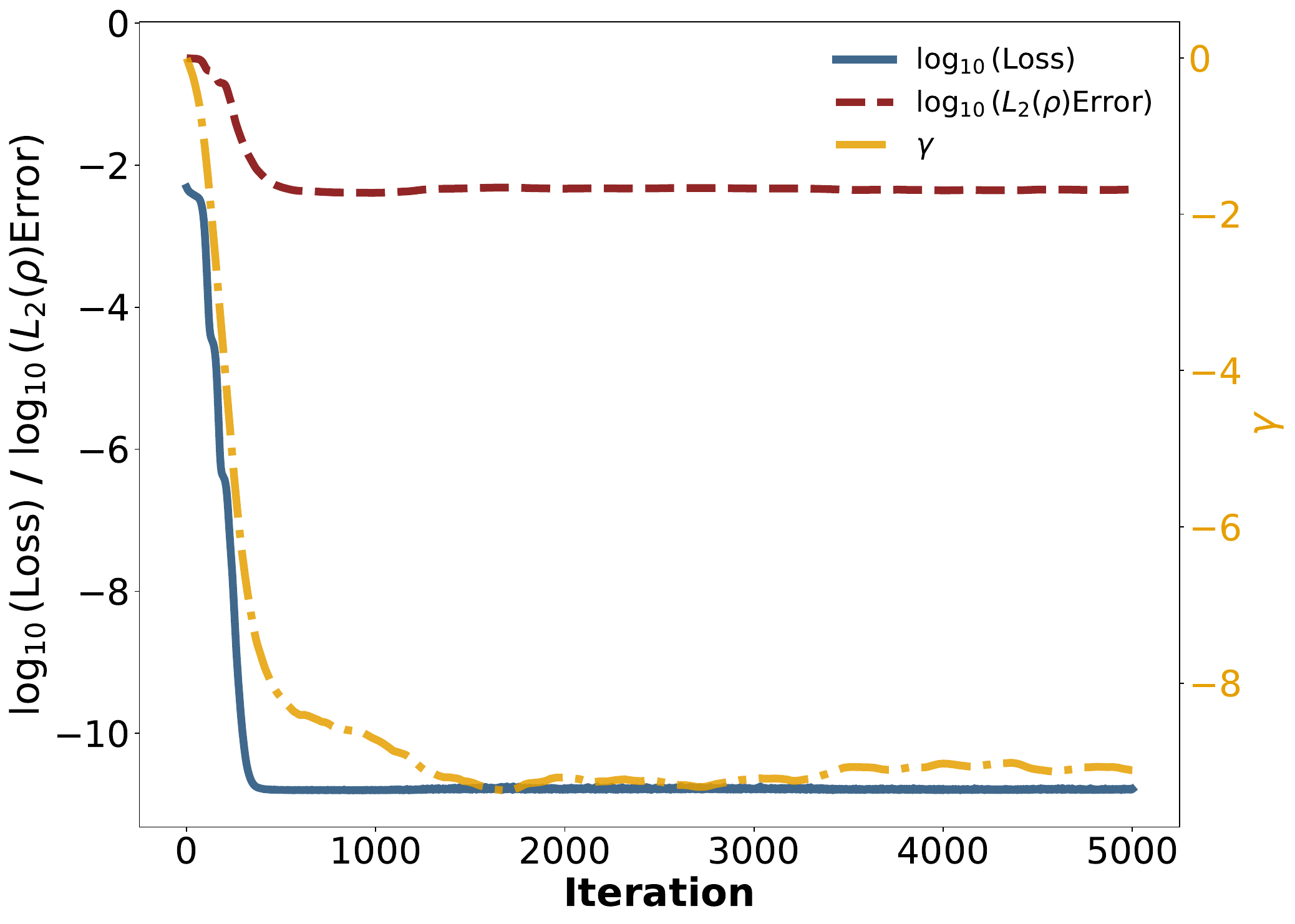}
    \label{fig:LG3}}\vspace{-2mm}
     \caption*{ (a) Linear drift: $b(x)=-0.5 x$.}
     \subfigure{\includegraphics[width=0.32\textwidth,height=0.25\textwidth]{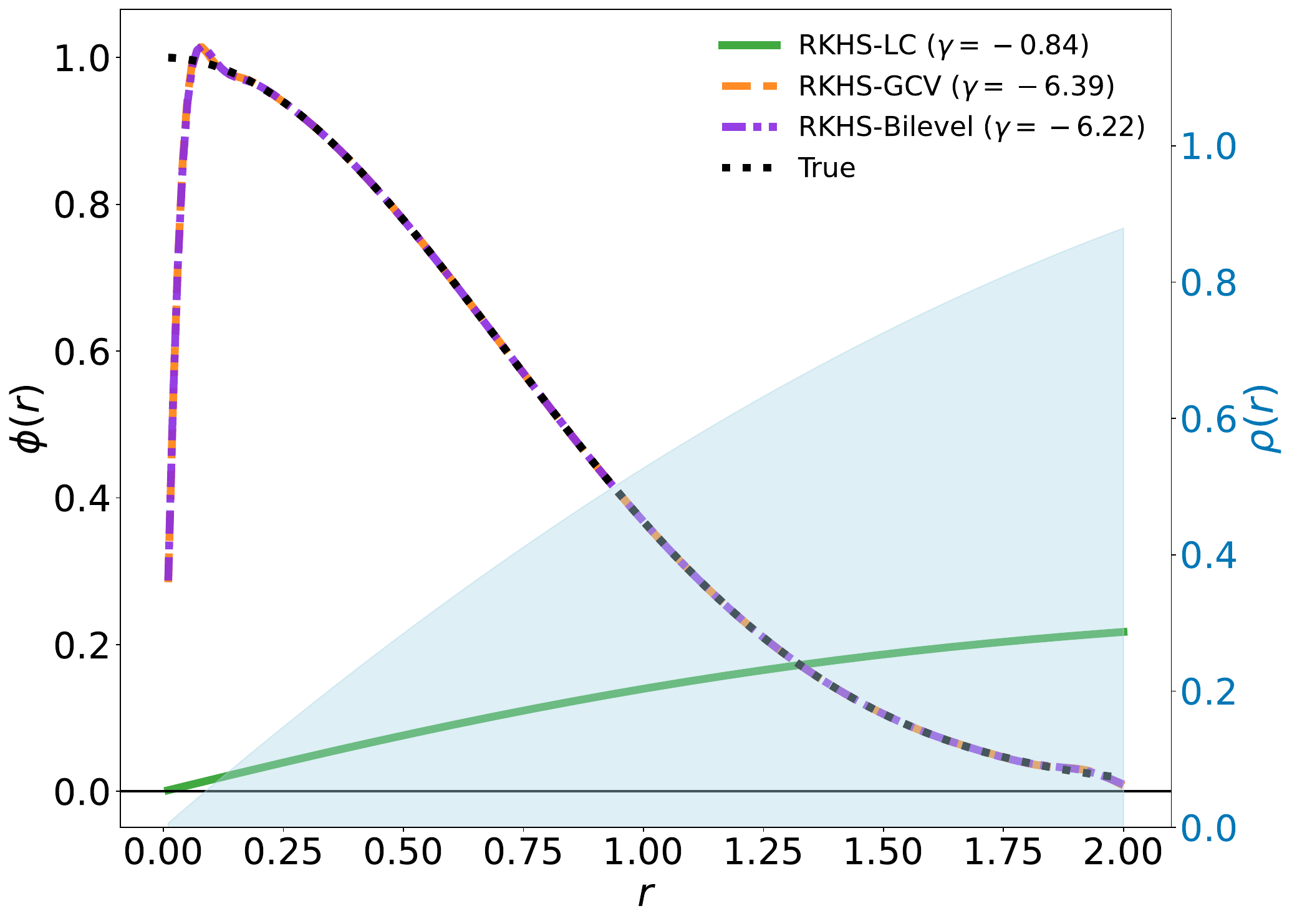}
    \label{fig:NG1}}    \subfigure{\includegraphics[width=0.32\textwidth,height=0.25\textwidth]{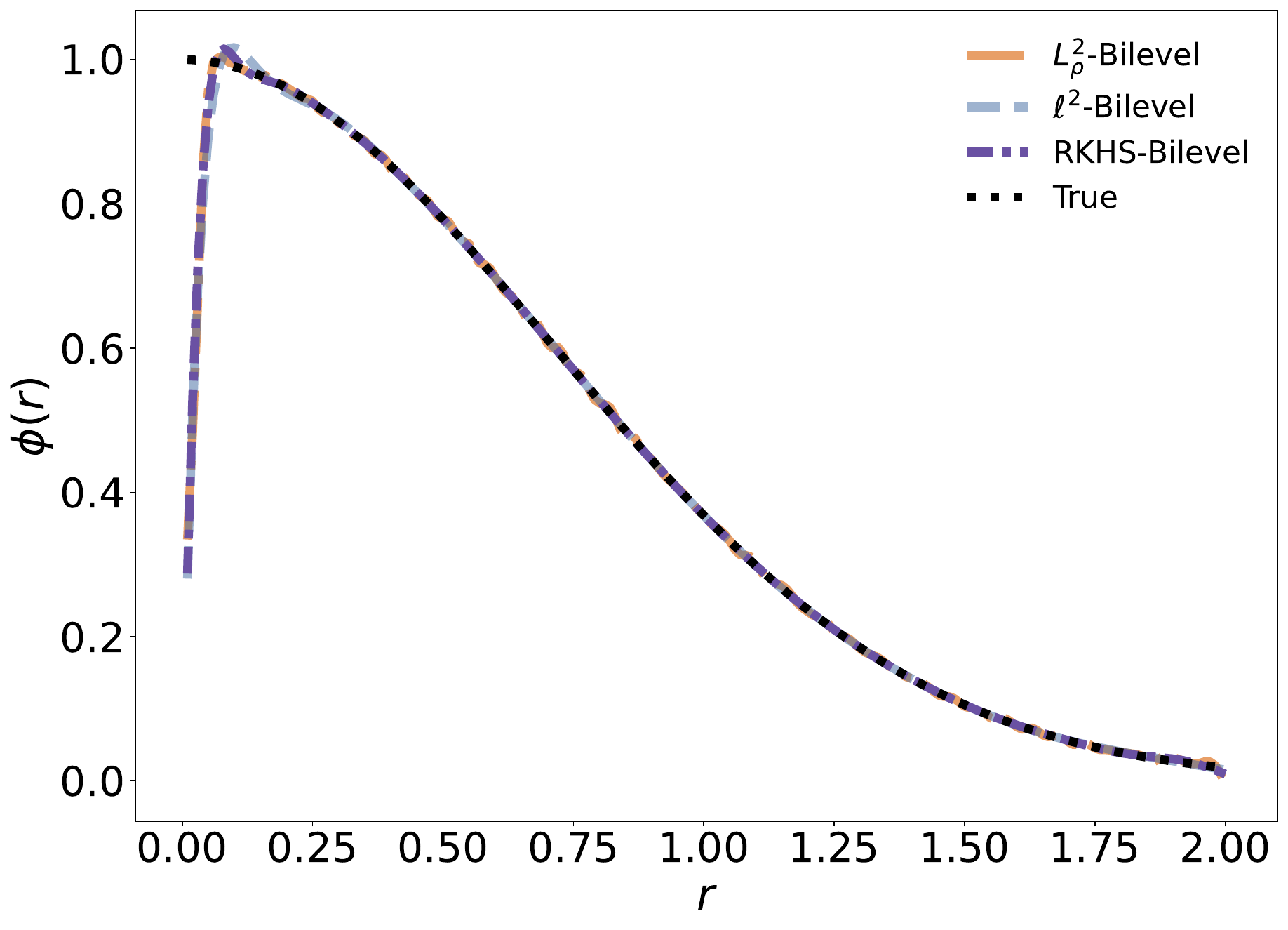}
    \label{fig:NG2}}     \subfigure{\includegraphics[width=0.32\textwidth,height=0.25\textwidth]{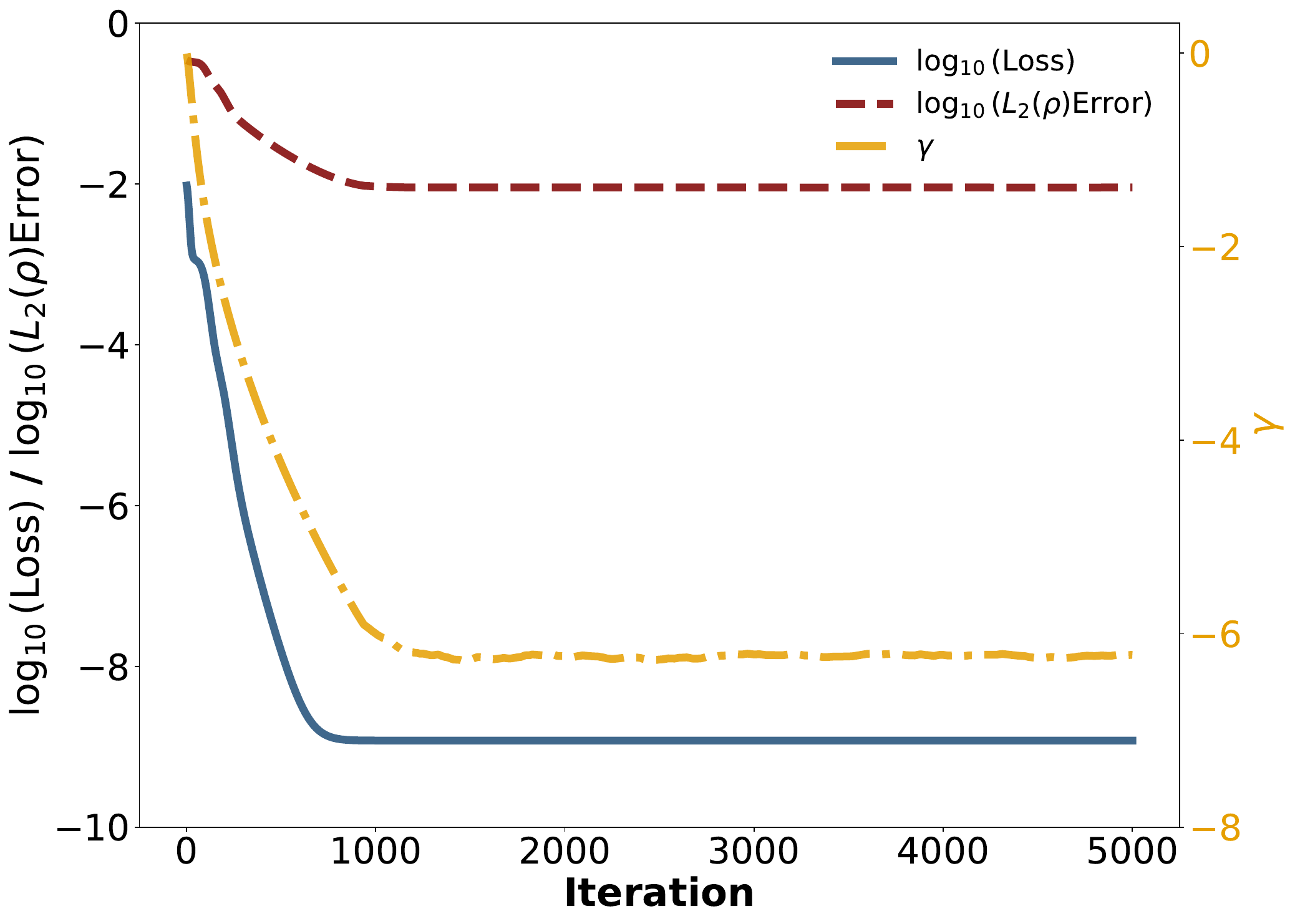} \label{fig:NG3}}\vspace{-2mm}
    \caption*{ (b) Nonlinear drift: $b(x)=\sin x $.} 
    \caption{Estimators under linear and nonlinear drift when data with mesh $\Delta x = 0.01$. Left: estimated L\'evy density $\phi(|y|)=e^{-|y|^2}$ using different methods; RKHS-Bilevel matches RKHS-GCV and outperforms RKHS-LC. The blue-shaded region indicates the exploration measure $\rho$. Middle: estimators with different regularization norms, showing RKHS norm is more stable than $L_\rho^2$ and $\ell^2$. Right: $L^2_\rho$ error, loss, and $\gamma$ versus iterations, showing convergence.
    }
    \label{fig:estimatorG}
\end{figure}

\subsubsection{L\'evy jump measure with Gaussian decay}

Figure \ref{fig:estimatorG} and Tables \ref{tab:errorNormG}-\ref{tab:errorG} show the estimation results for the Gaussian-decay L\'evy density $\phi(|y|)=e^{-|y|^2}$ under two drifts: $b(x) = -0.5 x$ and $b(x) = \sin x$.  

Across both drift settings, Figures~\ref{fig:estimatorG} and Tables~\ref{tab:errorNormG}–\ref{tab:errorG} show that the RKHS–Bilevel estimator consistently achieves the smallest or near–smallest $L_\rho^2$ error. In the linear case $b(x)=-0.5x$, RKHS–Bilevel slightly improves on RKHS–GCV and substantially outperforms RKHS–LC, while maintaining stable accuracy as $\Delta x$ increases. In the nonlinear case $b(x)=\sin x$, RKHS–Bilevel remains accurate and robust, whereas RKHS–GCV becomes noticeably more sensitive to mesh refinement.

The middle panels of Figure~\ref{fig:estimatorG} and Table~\ref{tab:errorNormG} confirm that, within the bilevel framework, RKHS–norm regularization is more effective than either $L_{\rho}^2$ or $\ell^2$ regularization: the RKHS norm yields the lowest relative errors in both drift scenarios.

The right panels of Figure~\ref{fig:estimatorG} further indicate that, as the outer iterations proceed, the loss, the estimation error, and the hyperparameter $\gamma$ all converge to stable values, demonstrating reliable convergence of the bilevel optimization procedure.

\begin{table}[H]
    \centering
    \caption{Relative errors of estimators for $\phi(|y|)=e^{-|y|^2}$ under different regularization norms.}
    \begin{tabular}{c|ccc}
        \toprule
         Drift term & $L^2_{\rho}$-Bilevel & $\ell^2$-Bilevel & \textbf{RKHS-Bilevel}\\
        \hline
         $-0.5x$ & 0.0763 &0.1001 & \textbf{0.0045} \\
         
         $\sin x$&0.0100 &0.0097 &\textbf{0.0090} \\
        \bottomrule
    \end{tabular}
    \vspace{0cm}
    \label{tab:errorNormG}
\end{table}

\begin{table}[H]
    \centering
    \caption{Relative $L^2_\rho$ errors for $\phi(|y|)=e^{-|y|^2}$ under  $b(x)=-0.5x$ and $b(x)=\sin x$, when the data has varying mesh size $\Delta x$. The RKHS-Bilevel method consistently achieves the lowest or near-lowest errors.}
    \begin{tabular}{c c|cccc}
        \toprule
        Case & Methods & $\Delta x = 0.01$ & $\Delta x = 0.02$ & $\Delta x = 0.025$ & $\Delta x = 0.05$ \\
        \midrule
        \multirow{3}{*}{$b(x)=-0.5x$} 
            & RKHS-LC      & 0.1418 & 0.1420 & 0.1416 & 0.1412 \\
            & RKHS-GCV     & 0.0049 & \textbf{0.0083} & \textbf{0.0215} & 0.0407 \\
            & RKHS-Bilevel & \textbf{0.0045} & {0.0212} & {0.0232} & \textbf{0.0258} \\
        \midrule
        \multirow{3}{*}{$b(x)=\sin x$} 
            & RKHS-LC      & 0.3279 & 0.0178 & \textbf{0.0210} & \textbf{0.0305} \\
            & RKHS-GCV     & 0.0090 & 0.0801 & 0.0962 & 0.1431 \\
            & {RKHS-Bilevel} & \textbf{0.0090} & \textbf{0.0197} & {0.0247} & {0.0353} \\
        \bottomrule
    \end{tabular}
    \vspace{0cm}
    \label{tab:errorG}
\end{table}

\subsubsection{L\'evy jump measure with Exponential decay}
In this subsection, we estimate from data the L\'evy density $\phi(|y|)=e^{-2|y|}$ under both a linear drift $b(x)=-0.5 x$ and a nonlinear drift $b(x)=\sin x$. Figure~\ref{fig:estimatorE} shows the estimators for mesh size $\Delta x=0.01$, while Tables~\ref{tab:errorNormE} and \ref{tab:errorE} summarize the relative $L^2_\rho$ errors for different norms and mesh sizes. In both drift settings, the RKHS–Bilevel estimator is either the best or very close to the best among all methods: it clearly improves over RKHS–LC and typically matches or slightly outperforms RKHS–GCV, while remaining stable as $\Delta x$ increases. Within the bilevel framework, RKHS–norm regularization yields the smallest $L^2_\rho$ errors, outperforming both the $L^2_\rho$ and $\ell^2$ norms, which confirms that the adaptive RKHS is a more suitable function space for this inverse problem. The right column of Figure~\ref{fig:estimatorE} further shows again that, as iterations proceed, the loss, the $L^2_\rho$ error, and the hyperparameter $\gamma$ all converge to steady values, indicating reliable convergence of the bilevel optimization scheme.

\begin{figure}[H]
    \centering   
    \subfigure{\includegraphics[width=0.32\textwidth,height=0.25\textwidth]{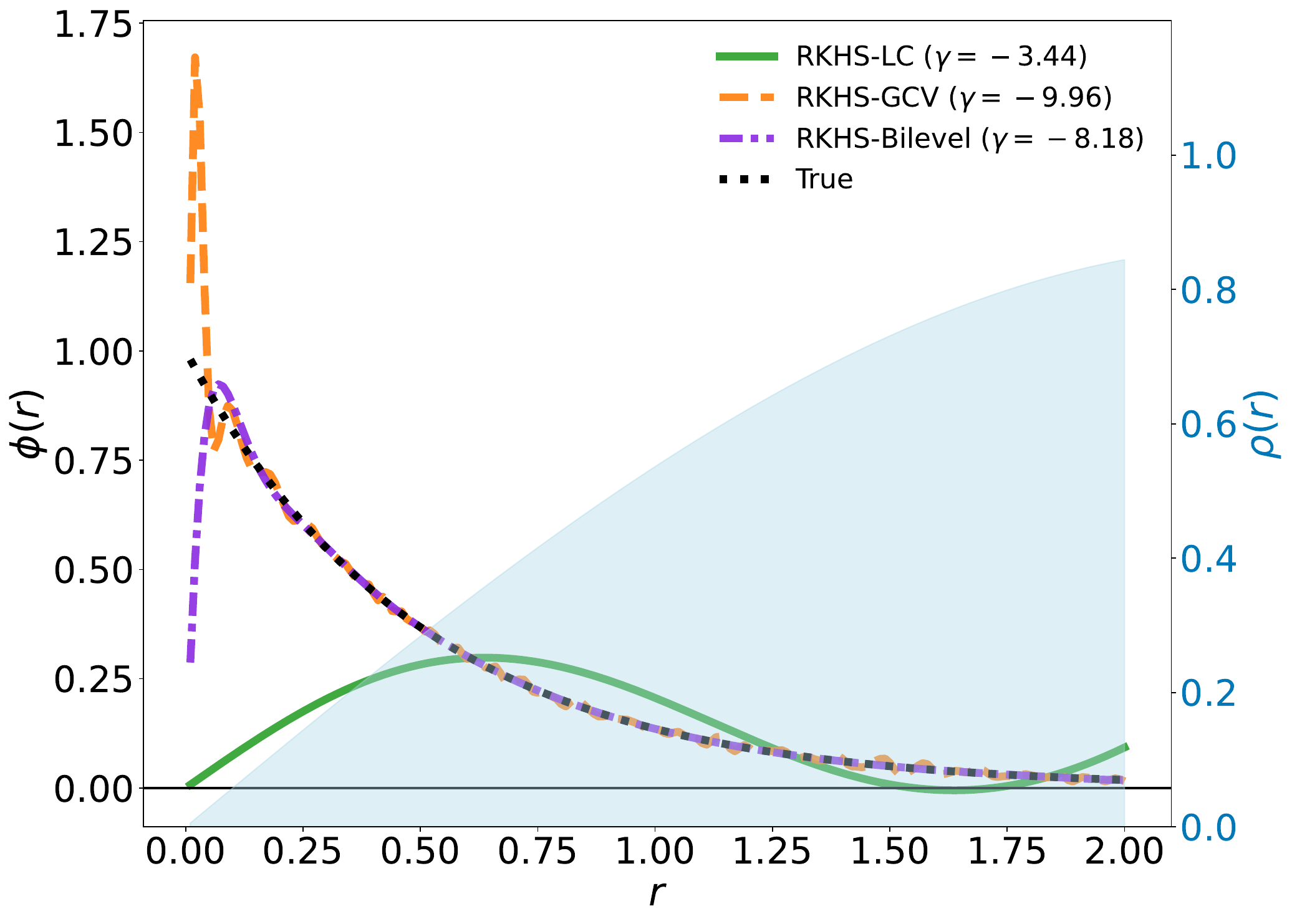}
    \label{fig:LE1}}
    \subfigure{\includegraphics[width=0.32\textwidth,height=0.25\textwidth]{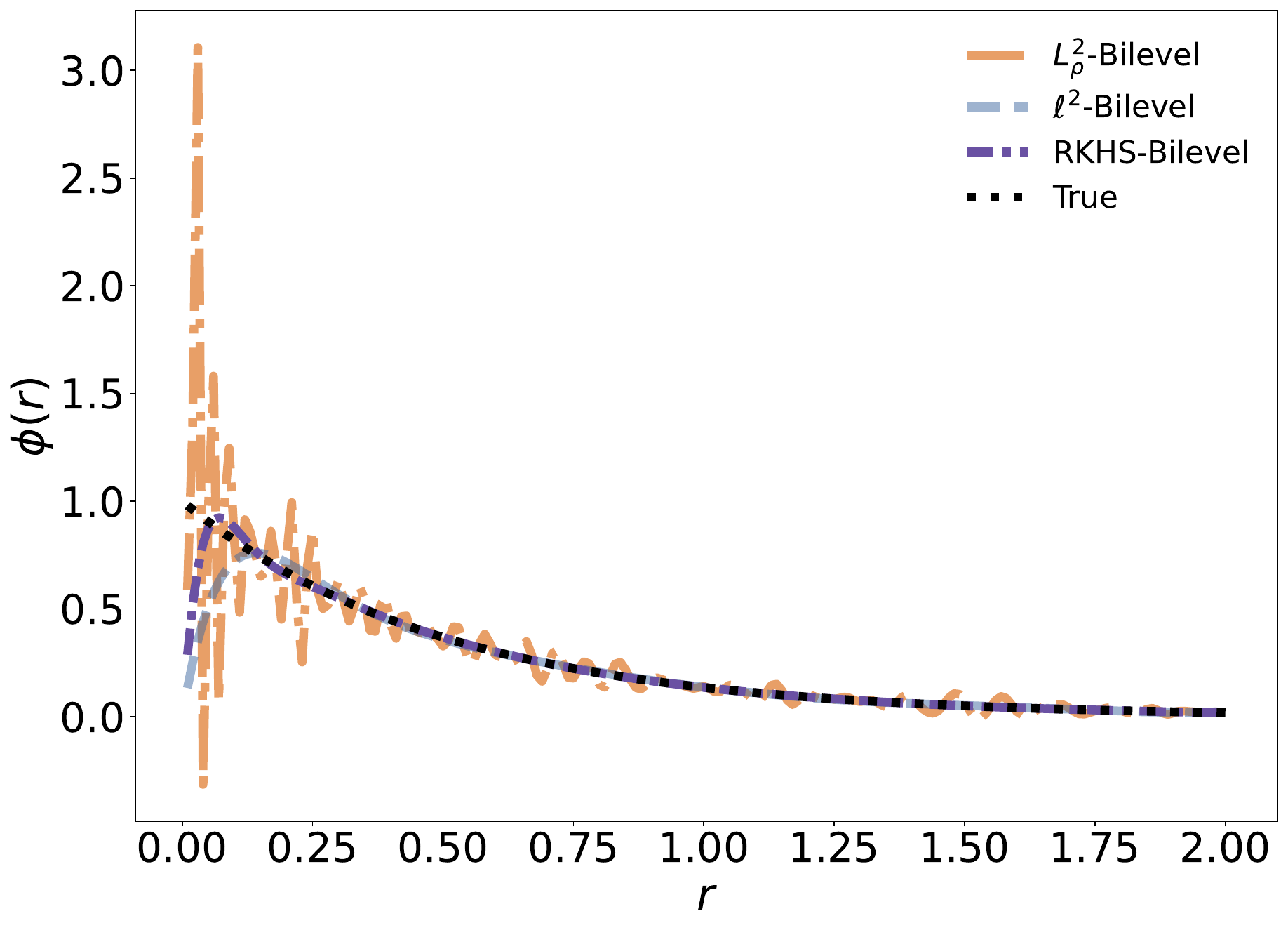}
    \label{fig:LE2}}
    \subfigure{\includegraphics[width=0.32\textwidth,height=0.25\textwidth]{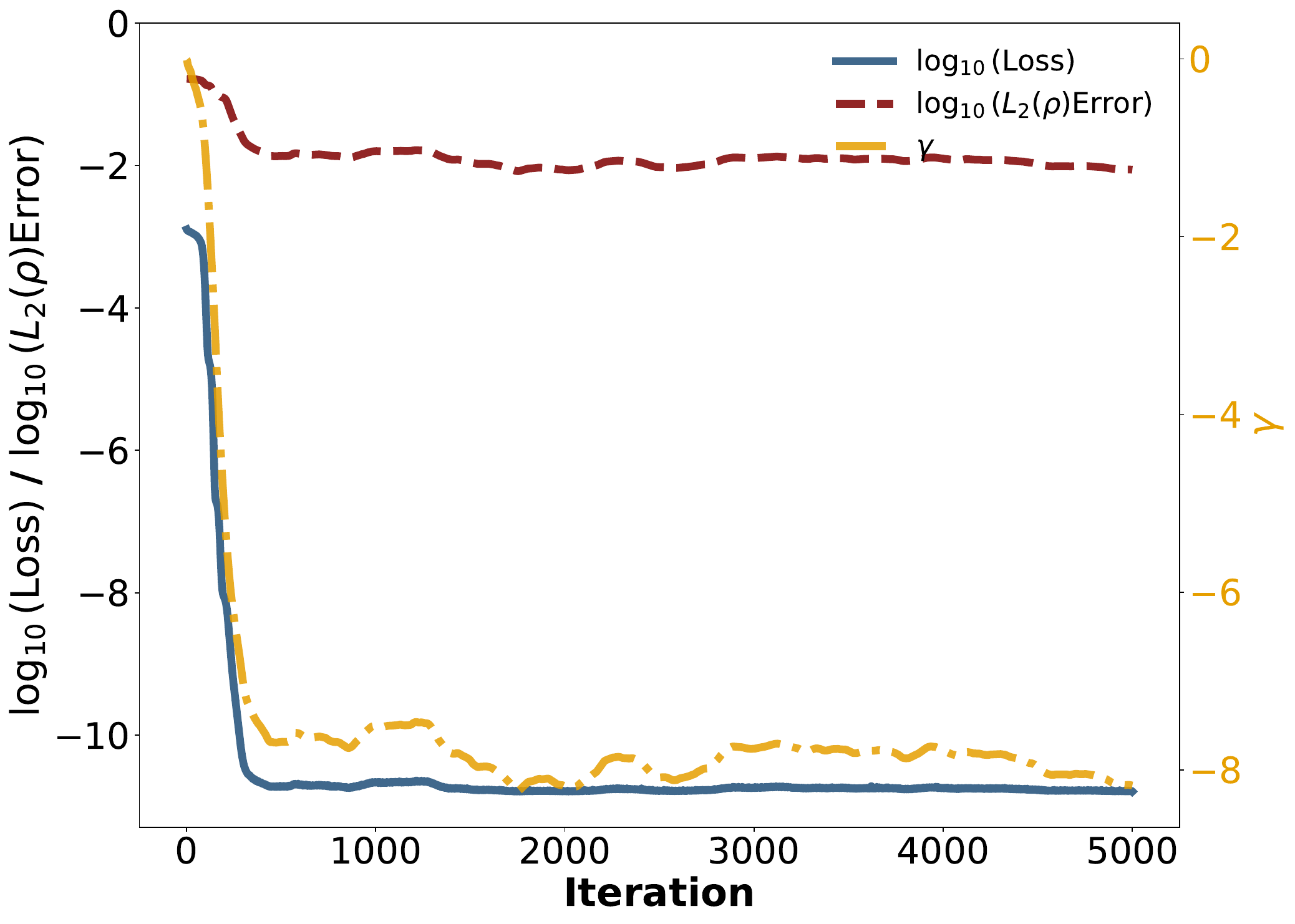}
    \label{fig:LE3}}\vspace{-3mm}
     \caption*{ (a) Linear drift: $b(x)=-0.5 x$.}
     \subfigure{\includegraphics[width=0.32\textwidth,height=0.25\textwidth]{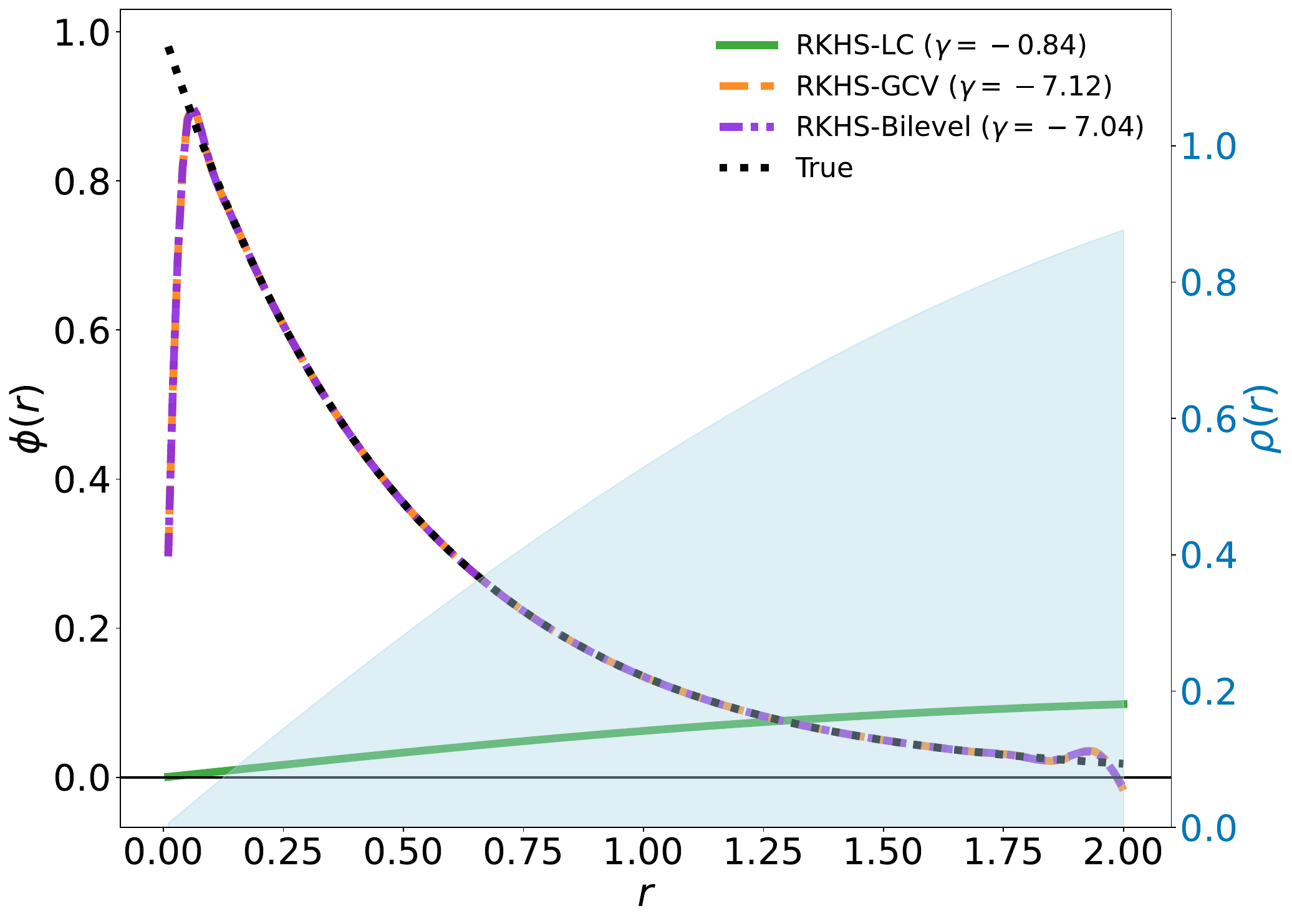}
    \label{fig:NE1}}
    \subfigure{\includegraphics[width=0.32\textwidth,height=0.25\textwidth]{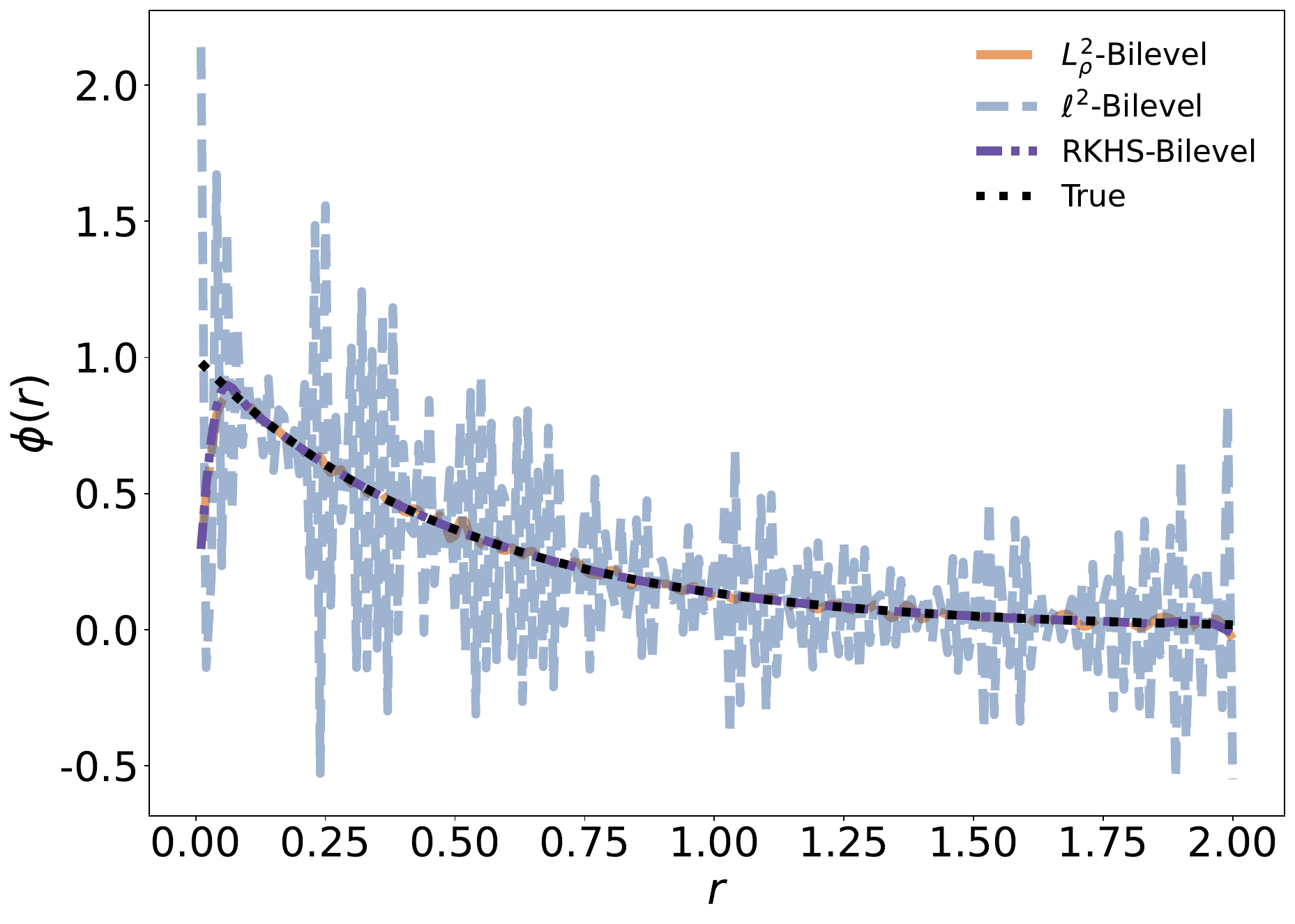}
    \label{fig:NE2}}
    \subfigure{\includegraphics[width=0.32\textwidth,height=0.25\textwidth]{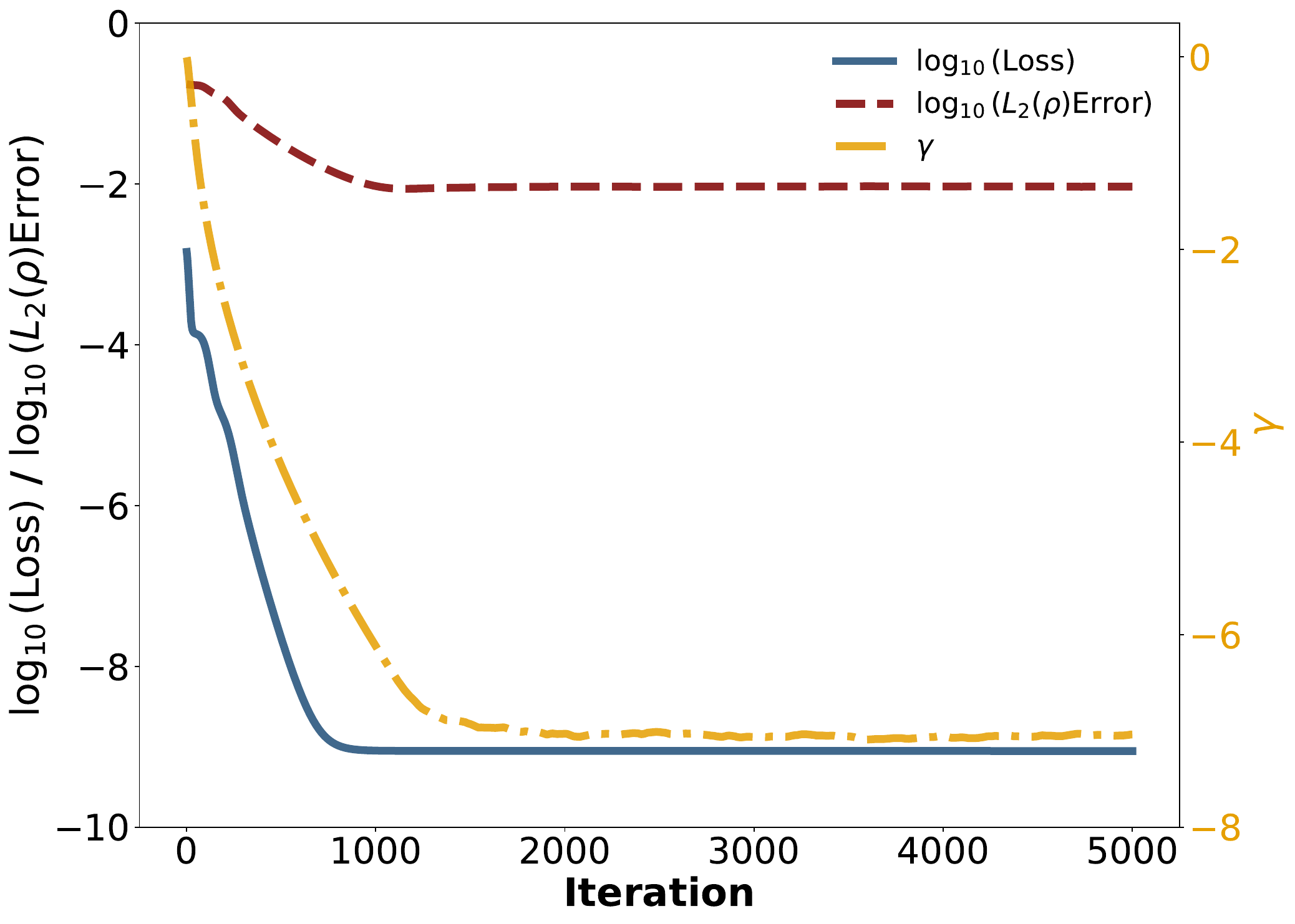}
    \label{fig:NE3}} \vspace{-3mm}
    \caption*{ (b) Nonlinear drift: $b(x)=\sin x $.}\vspace{-2mm}
    \caption{Comparison of estimators under linear and nonlinear drift. Left: estimated L\'evy density $\phi(|y|)=e^{-2|y|}$ for different methods. RKHS-Bilevel improves over RKHS–LC and matches or slightly improves upon RKHS–GCV. Blue shading: numerically computed $\rho$. Middle: effect of regularization norms, with RKHS outperforming $L_\rho^2$ and $\ell^2$. Right: $L^2_\rho$ error, loss, and hyperparameter $\gamma$ over iterations, indicating convergence.}
    \label{fig:estimatorE}
\end{figure}

\begin{table}[thb]
    \centering
    \caption{Relative errors of estimators for $\phi(|y|)=e^{-2|y|}$ under different regularization norms.  }
    \begin{tabular}{c|ccc}
        \toprule
         Drift term & $L^2_{\rho}$-Bilevel & $\ell^2$-Bilevel & \textbf{RKHS-Bilevel}\\
        \hline
         $-0.5x$ &0.0564 & 0.0196 &\textbf{0.0088} \\
         
         $\sin x$ & 0.0164 &0.2856 &\textbf{0.0093} \\
        \bottomrule
    \end{tabular}
    \vspace{0cm}
    \label{tab:errorNormE}
\end{table}

\begin{table}[thb]
    \centering
    \caption{Relative $L^2_\rho$ errors for $\phi(|y|)=e^{-2|y|}$ under  $b(x)=-0.5x$ and $b(x)=\sin x$.}
    \begin{tabular}{c c|cccc}
        \toprule
        Case & Methods & $\Delta x = 0.01$ & $\Delta x = 0.02$ & $\Delta x = 0.025$ & $\Delta x = 0.05$ \\
        \midrule
        \multirow{3}{*}{$b(x)=-0.5x$} 
            & RKHS-LC      & 0.1088 & 0.1086 & 0.1086 & 0.1082 \\
            & RKHS-GCV     & 0.0136 & 0.1271 & 0.0944 & 0.1623 \\
            & {RKHS-Bilevel} & \textbf{0.0088} & \textbf{0.0212} & \textbf{0.0220} & \textbf{0.0240} \\
        \midrule
        \multirow{3}{*}{$b(x)=\sin x$} 
            & RKHS-LC      & 0.1700 & 0.1696 & 0.1694 & 0.1688 \\
            & RKHS-GCV     & 0.0095 & 0.0645 & 0.0797 & 0.1230 \\
            & {RKHS-Bilevel} & \textbf{0.0093} & \textbf{0.0196} & \textbf{0.0248} & \textbf{0.0377} \\
        \bottomrule
    \end{tabular}
    \vspace{0cm}
    \label{tab:errorE}
\end{table}

\subsubsection{Convergence as the observation mesh size}
\label{sec:converge}

We next examine the convergence of the regularized estimators as $\Delta x$ increases; see Fig.~\ref{fig:estimatoC}. On the log–log scale, the $L^2_\rho$–error decays with an empirical slope close to $1$ in all four test cases. These rates are higher than the rate $(\Delta x)^{1/2}$ in Theorem~\ref{thm:conv}, which is proved under the worst–case source condition and only polynomial spectral decay, in the saturated regime $\beta>4$. This apparent discrepancy is not a contradiction, but reflects that the theorem gives a rate on what can be guaranteed uniformly over all admissible problems, whereas our examples are much smoother than required by the theory (i.e., $\phi^\star$ is smooth, corresponding to $\beta=\infty$). In particular, the densities $\{p_{t_i}\}$ are smooth and dependent, so the normal matrix has multiple zero eigenvalues, not satisfying the polynomial spectral decay condition. In this setting, the bilevel choice of $\lambda$ drives the solution into a regime where the dominant contribution to the error comes from discretization (finite differences in $x$ and the piecewise-constant approximation in $r$), which is first-order in $\Delta x$. Hence, the observed $L^2_\rho$–error behaves like $O(\Delta x)$, i.e., with slope close to $1$.

\begin{figure}[thb]
    \centering   
    \subfigure[$ \phi(|y|)=e^{-|y|^2}$]{\includegraphics[width=0.3\textwidth,height=0.28\textwidth]{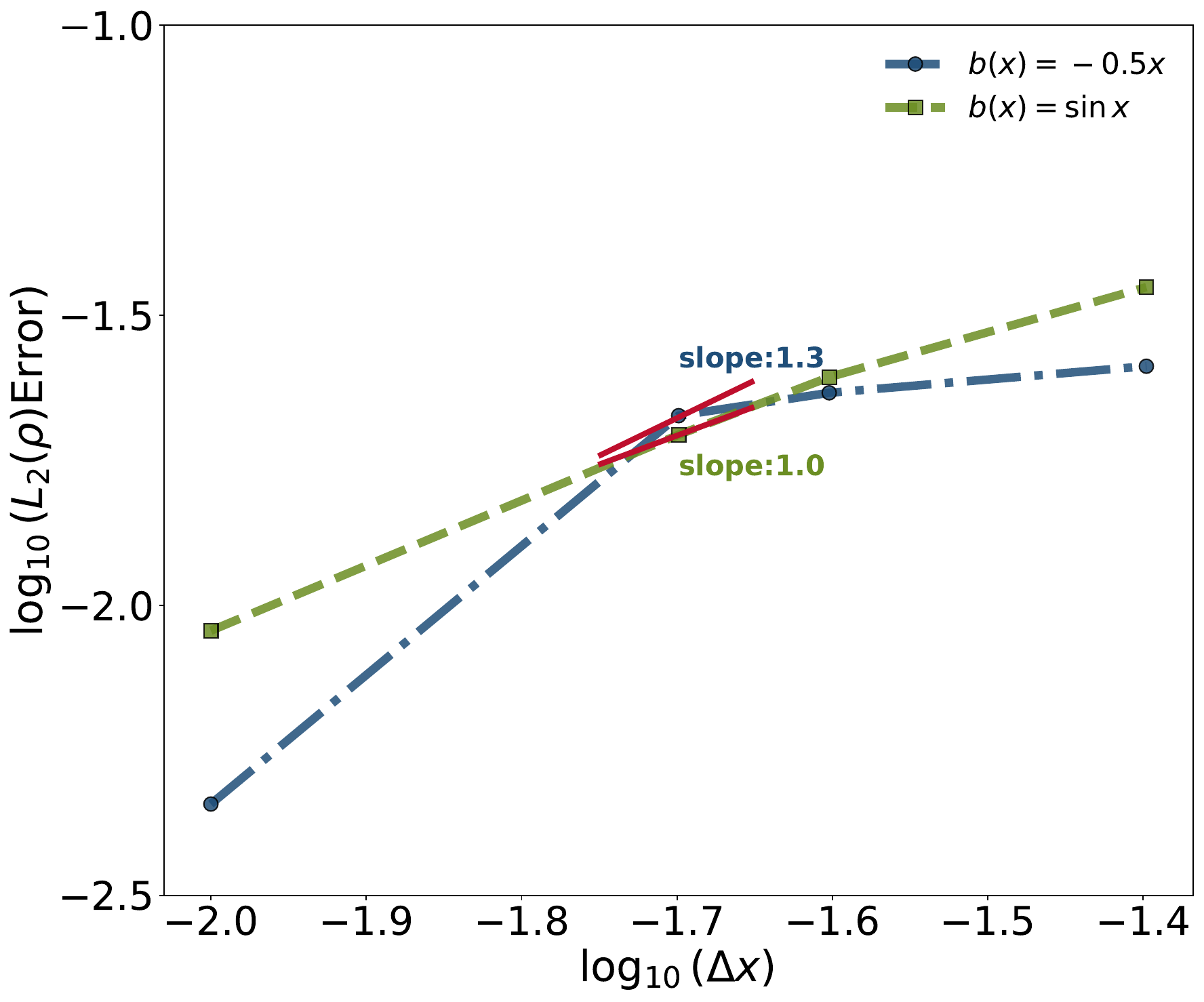}
    \label{fig:LG}
    }
    \subfigure[$\phi(|y|)=e^{-2|y|}$]{\includegraphics[width=0.3\textwidth,height=0.28\textwidth]{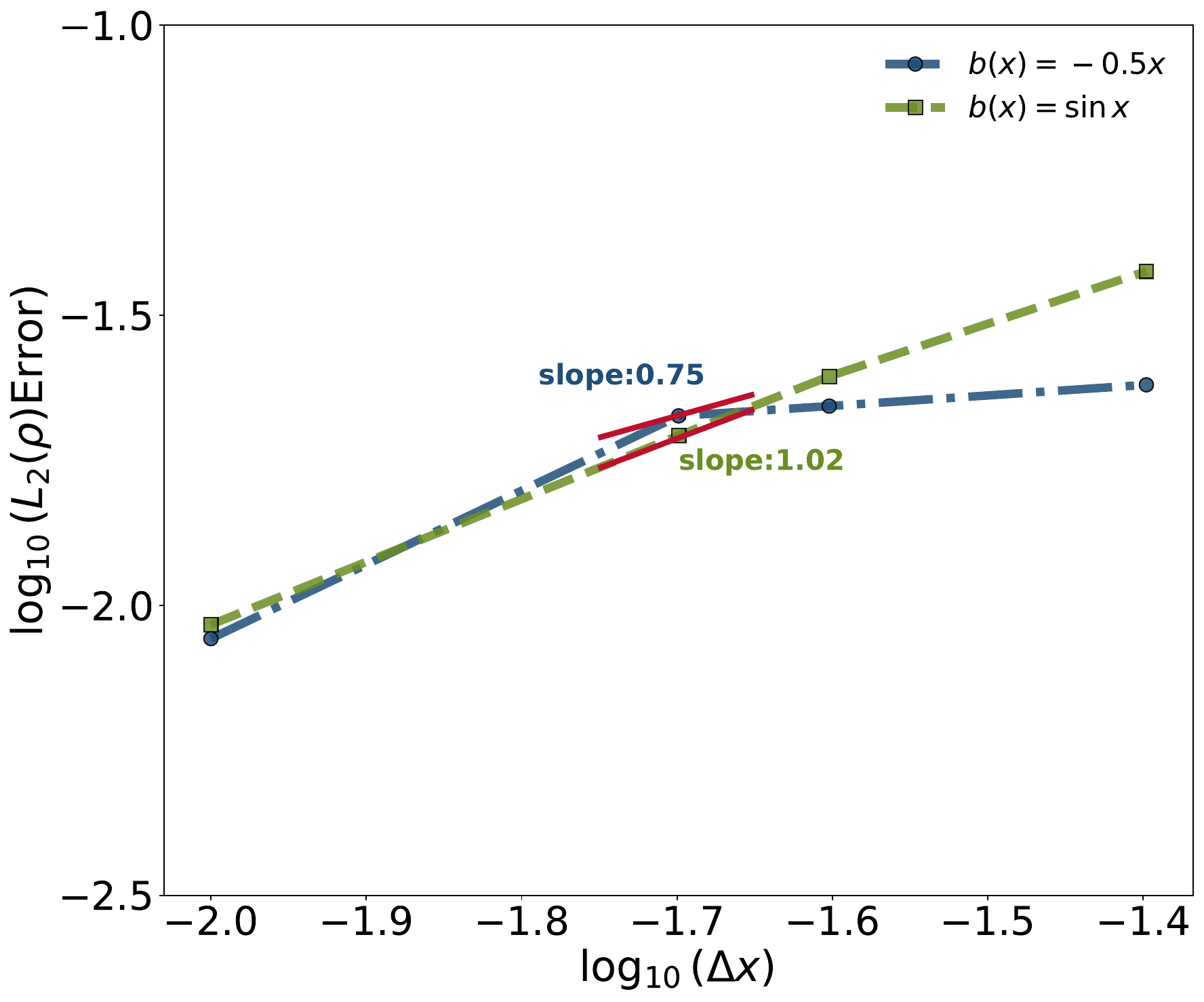}
    \label{fig:LE}
    }
    \caption{Convergence of estimator error and loss as mesh refines. 
    The $L^2_\rho$ error decays approximately like $O(\Delta x)$ (slope close to $1$), faster than the worst-case theoretical rate, due to high smoothness of data in the examples. 
    }
    \label{fig:estimatoC}
\end{figure}

\subsection{L\'evy density estimation from a sequence of ensembles}

Another practical scenario is to estimate the L\'evy density $\phi$ directly from a sequence of sample ensembles of the underlying SDE, without access to the density data. In this setting, we first estimate the densities $p_{t_i}$ from the ensemble data using kernel density estimation (KDE) and then use the estimated densities to construct the dataset $\mathcal{D}$ in \eqref{eq: discretedata}.

Following \cite{comte2009nonparametric}, we consider the simplest setting of estimating the L\'evy density $\phi$ of the L\'evy process $L_t$: 
\begin{equation}
\label{eq: SDEexampleone}
d X_t=d L_t,
\end{equation}
from data consisting of sequence of ensembles $\left\{\mathcal{S}_i\right\}_{i=1}^N$ with $\mathcal{S}_i=\left\{X_{t_i}^{(j)}\right\}_{j=1}^J$, where $X_{t_i}^{(j)}$ is the $j$-th sample at time $t_i$. 

We generate the data using the Euler–Maruyama scheme. A key step is to compute the increment $\Delta L_{t_i}=L_{t_{i+1}}-L_{t_i}$, where the driving L\'evy process is taken to be compound Poisson. By \cite{protter1997euler}[Theorem 3.1], we simulate these increments in two steps: (i) compute the jump arrival rate $\lambda= \int_{\mathbb{R} \backslash\{0\}} \nu(d y)$, and the distribution of jump $\frac{\nu(d y)}{\lambda}$; (ii) simulate the compound Poisson process $L_t$ of the form $L_t :=\sum_{i=1}^{\infty} V^i {\mathrm {1}}_{\left[t \geq T_i\right]},$ where $T_i$ are Poisson arrival times of intensity $\lambda$ and $V_i$ are i.i.d. with law $\mu:=(1 / \lambda) \nu$.

\paragraph{Numerical settings:} To obtain discrete observations $\left\{\mathcal{S}_i\right\}_{i=1}^N$, we simulate $J= 10^6$ trajectories over $t \in[0,5]$ with time step $\Delta t=0.05$ via the explicit Euler scheme. The L\'evy increments are constructed in accordance with the two types of L\'evy densities introduced in Subsection \ref{subSec:iterative}:
\begin{itemize}
\item[(a)] \textbf{Gaussian decay:} $\phi(|y|)=e^{-|y|^2}$; $\lambda = \sqrt{\pi}$; $\mu(y) = \frac{e^{-|y|^2}}{\sqrt{\pi}} $; $V_i \sim \mathcal{N}(0, \frac{1}{2})$
\item[(b)] \textbf{Exponential decay:} $\phi(|y|) = e^{-2|y|}$; $\lambda = 1$; $\mu (y)= e^{-2|y|}$; $V_i \sim \text { Laplace }(0, \frac{1}{2})$.
\end{itemize}

From the ensemble data, we estimate the density $p\left(\cdot, t_i\right)$ at each time point $t_i$ using KDE. For each $t_i$, we evaluate the KDE on a uniform grid $x_k \in[-5,5]$ with spacing $\Delta x=0.01$. To stabilize numerical differentiation, we smooth the KDE values with a Savitzky-Golay (SG) filter \cite{savitzky1964smoothing} and then apply finite differences to the smoothed PDF, thereby constructing dataset $\mathcal{D}$ in \eqref{eq: discretedata}. Unless otherwise specified, we use a KDE bandwidth $h=0.105$, corresponding to $c=0.5$ in the rule $h=c\left(J \Delta t^2\right)^{-1 / 5}$.

For the bilevel optimization, we fix $\eta=0.005$ and $\iota=0.85$ in our algorithm.

Figure \ref{fig:kdgestimatorP} summarizes the estimators obtained from ensemble data for both L\'evy densities, $\phi(|y|)=e^{-|y|^2}$ and $\phi(|y|)=e^{-2|y|}$. In each case, the RKHS–Bilevel method produces the visually best fit, while RKHS–LC is slightly less accurate and RKHS–GCV can become unstable (in particular for the exponential-decay density, where the GCV curve is highly oscillatory). The right panels show that, for both densities, the loss, the $L^2_\rho$ error and the hyperparameter $\gamma$ converge as the outer iterations proceed.

Table \ref{tab:KDEerrorNormMerged} reports the corresponding $L^2_\rho$ errors. They confirm that RKHS–Bilevel attains the smallest error among the three hyperparameter–selection methods, and also compare bilevel optimization under different regularization norms. Across both densities, RKHS–norm regularization clearly outperforms $L^2_\rho$–Bilevel (which is very unstable) and $\ell^2$–Bilevel (which yields larger errors), reinforcing the advantage of the adaptive RKHS norm in this setting.

\begin{figure}[H]
    \centering   
    \subfigure[]{\includegraphics[width=0.38\textwidth,height=0.28\textwidth]{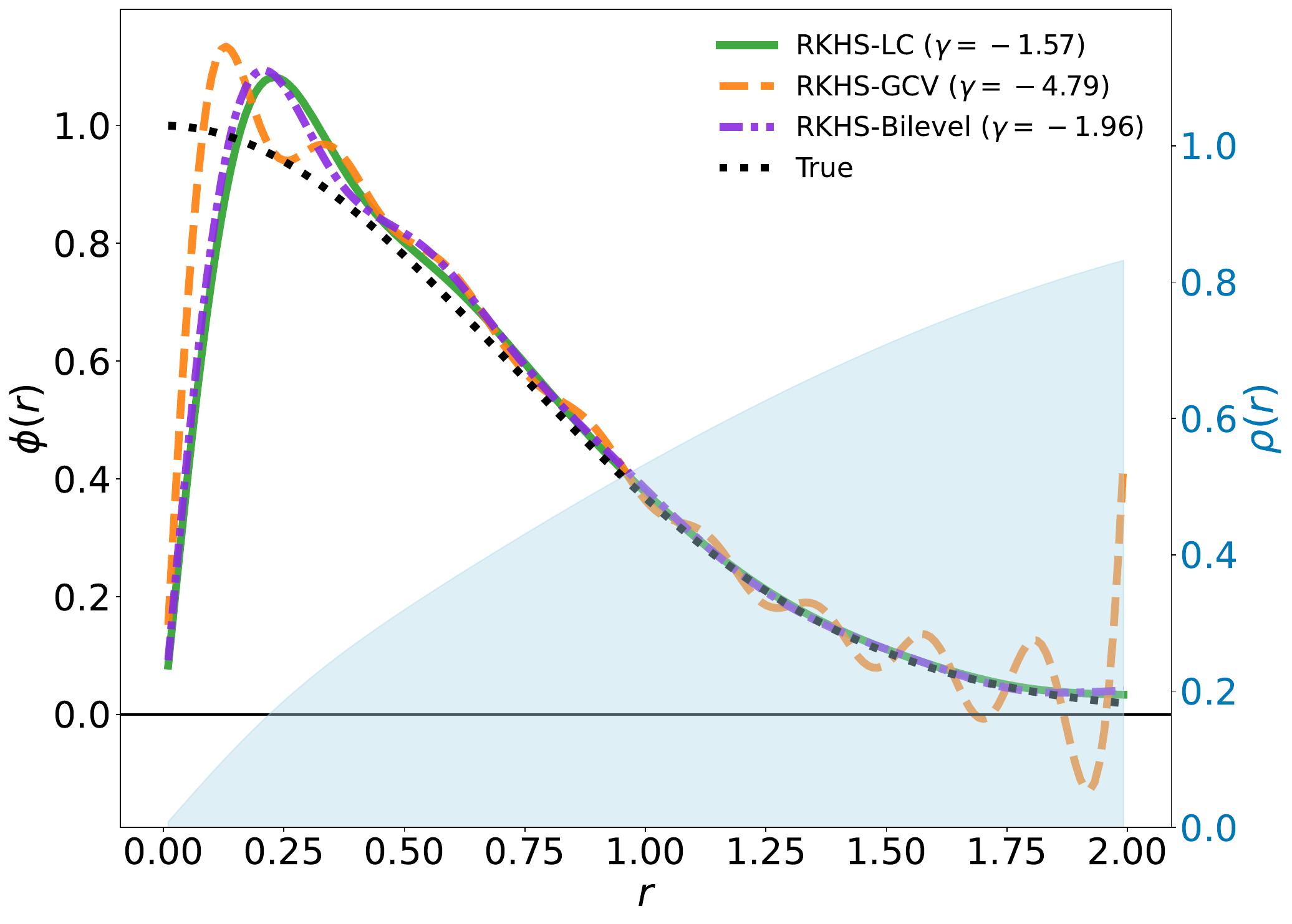}
    \label{fig:KDEP1}}
    \subfigure[]{\includegraphics[width=0.38\textwidth,height=0.28\textwidth]{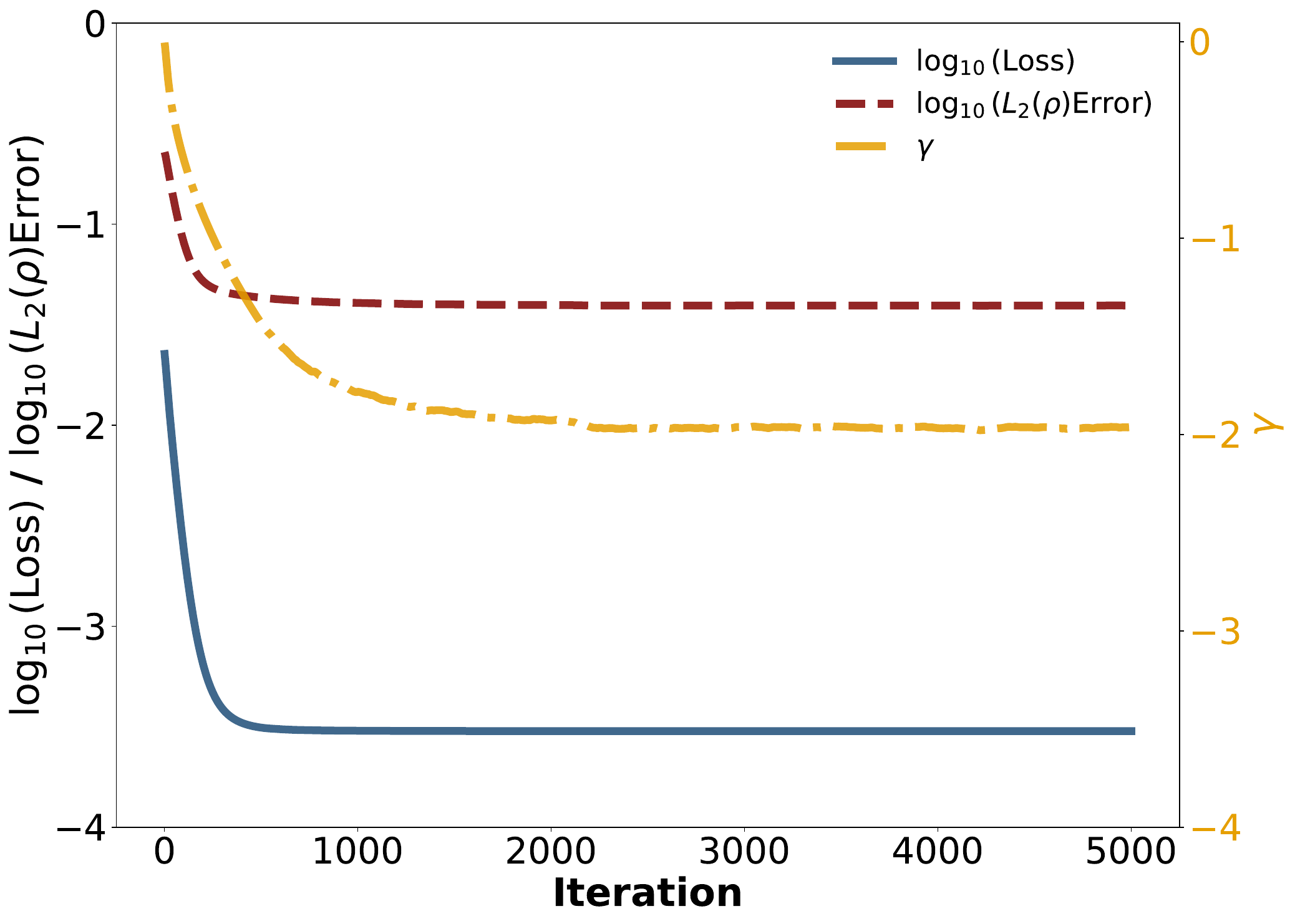}
    \label{fig:KDEP2}}\\
        \subfigure[]{\includegraphics[width=0.38\textwidth,height=0.28\textwidth]{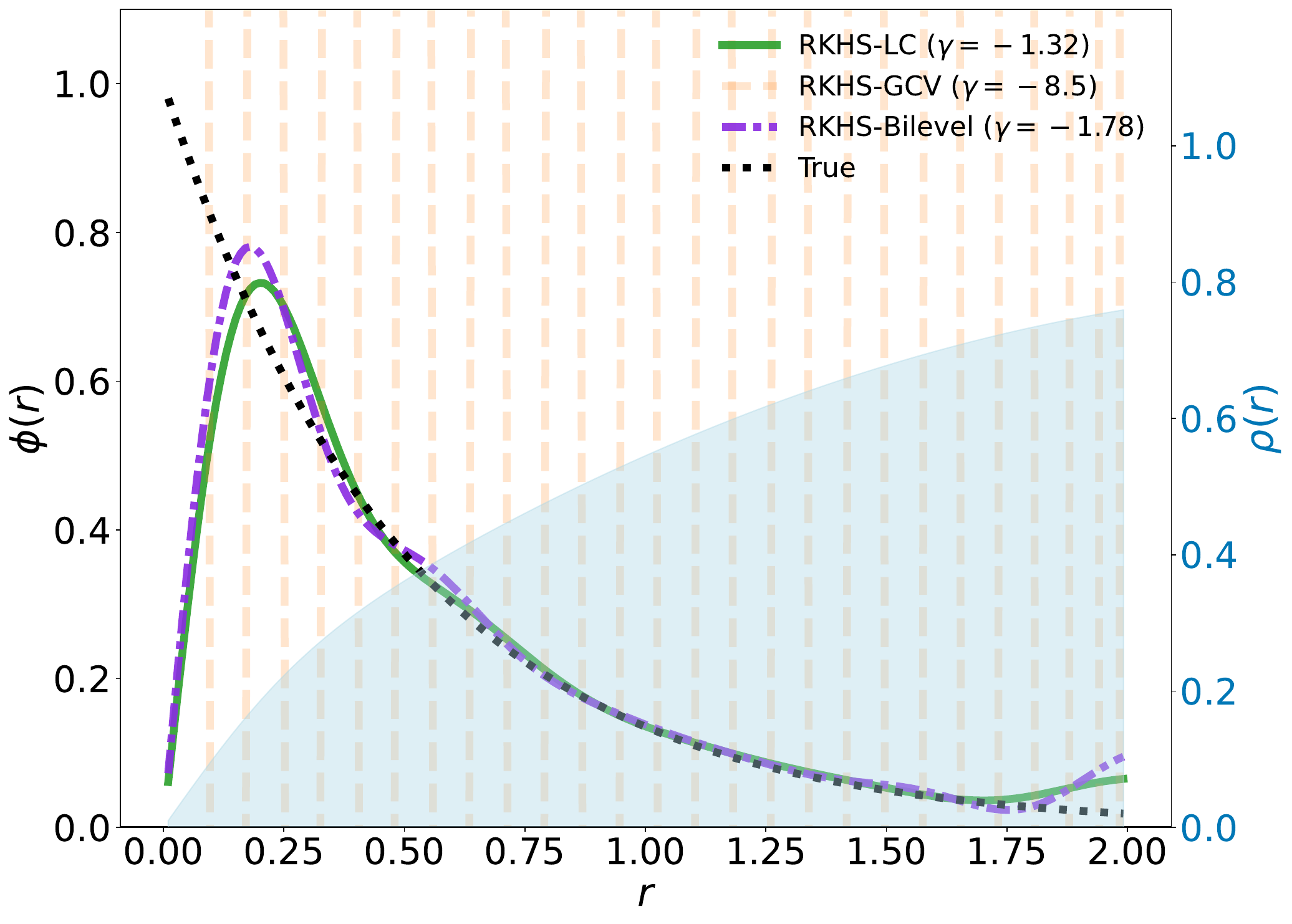}
    \label{fig:KDEP3}}
    \subfigure[]{\includegraphics[width=0.38\textwidth,height=0.28\textwidth]{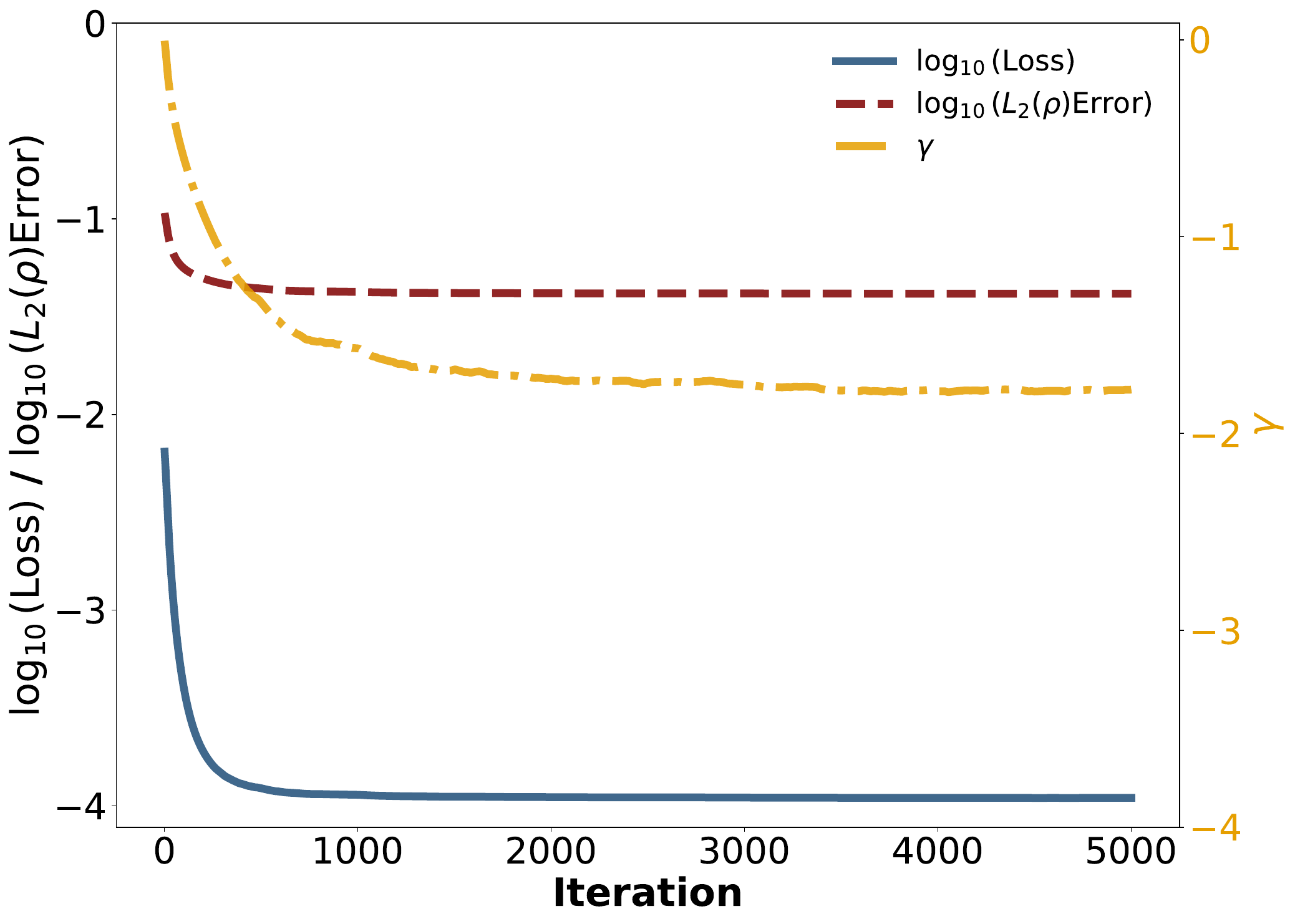}
    \label{fig:KDEP4}}
    \caption{Estimators from ensemble data. (a) \& (c): Regularized estimators with hyperparameter selected by L-curve, GCV and bi-level methods. RKHS-Bilevel outperforms RKHS–GCV and slightly improves upon RKHS–LC. (b)\& (d): Decay of the $L_\rho^2$ error, loss, and regularization hyperparameter $\gamma$ over iterations, illustrating the convergence behavior of RKHS-Bilevel.}
    \label{fig:kdgestimatorP}
\end{figure}

\begin{table}[H]
    \centering
    \caption{Relative $L^2_\rho$ errors across methods and norms for two L\'evy densities}
    \label{tab:KDEerrorNormMerged}
    \begin{tabular}{l l c l c}
        \toprule
        $\phi$ & \multicolumn{2}{c}{Different methods} & \multicolumn{2}{c}{Different norms} \\
        \cmidrule(r){2-3}\cmidrule(l){4-5}
                     & Setting & Error & Setting & Error \\
        \midrule
        \multirow{3}{*}{$e^{-|y|^2}$}
            & RKHS--LC          & 0.0421  & $L^2_{\rho}$--Bilevel & 5.3434 \\
            & RKHS--GCV         & 0.0687  & $\ell^2$--Bilevel     & 0.0460 \\
            & \textbf{RKHS--Bilevel} & \textbf{0.0394} & & \\
        \midrule
        \multirow{3}{*}{$e^{-2|y|}$}
            & RKHS--LC          & 0.0442  & $L^2_{\rho}$--Bilevel & 55.2078 \\
            & RKHS--GCV         & 11.9383 & $\ell^2$--Bilevel     & 0.0479 \\
            & \textbf{RKHS--Bilevel} & \textbf{0.0414} & & \\
        \bottomrule
    \end{tabular}
\end{table}

Figure~\ref{fig:cbandwidth} illustrates the impact of the KDE bandwidth $h$ (parameterized by $c$ in the formula $h=c(J\Delta t^2)^{-1/5}$) on both the estimated PDFs and the recovered L\'evy densities. In panels (a) and (c), smaller bandwidths yield less-smoothed, more oscillatory KDEs, while larger bandwidths produce smoother but more biased density estimates, especially near the origin. Panels (b) and (d) show the corresponding L\'evy density estimates: when $h$ is too small the estimates inherit strong oscillations from the noisy PDFs, whereas for overly large $h$ the estimates become overly smooth and lose accuracy. These results highlight the crucial role of choosing an appropriate bandwidth: it must be small enough to preserve the structure of $p_t$ for accurate recovery of $\phi$, yet large enough to stabilize numerical differentiation and suppress spurious oscillations.

\begin{figure}[htb]
    \centering   
    \subfigure[]{\includegraphics[width=0.38\textwidth,height=0.28\textwidth]{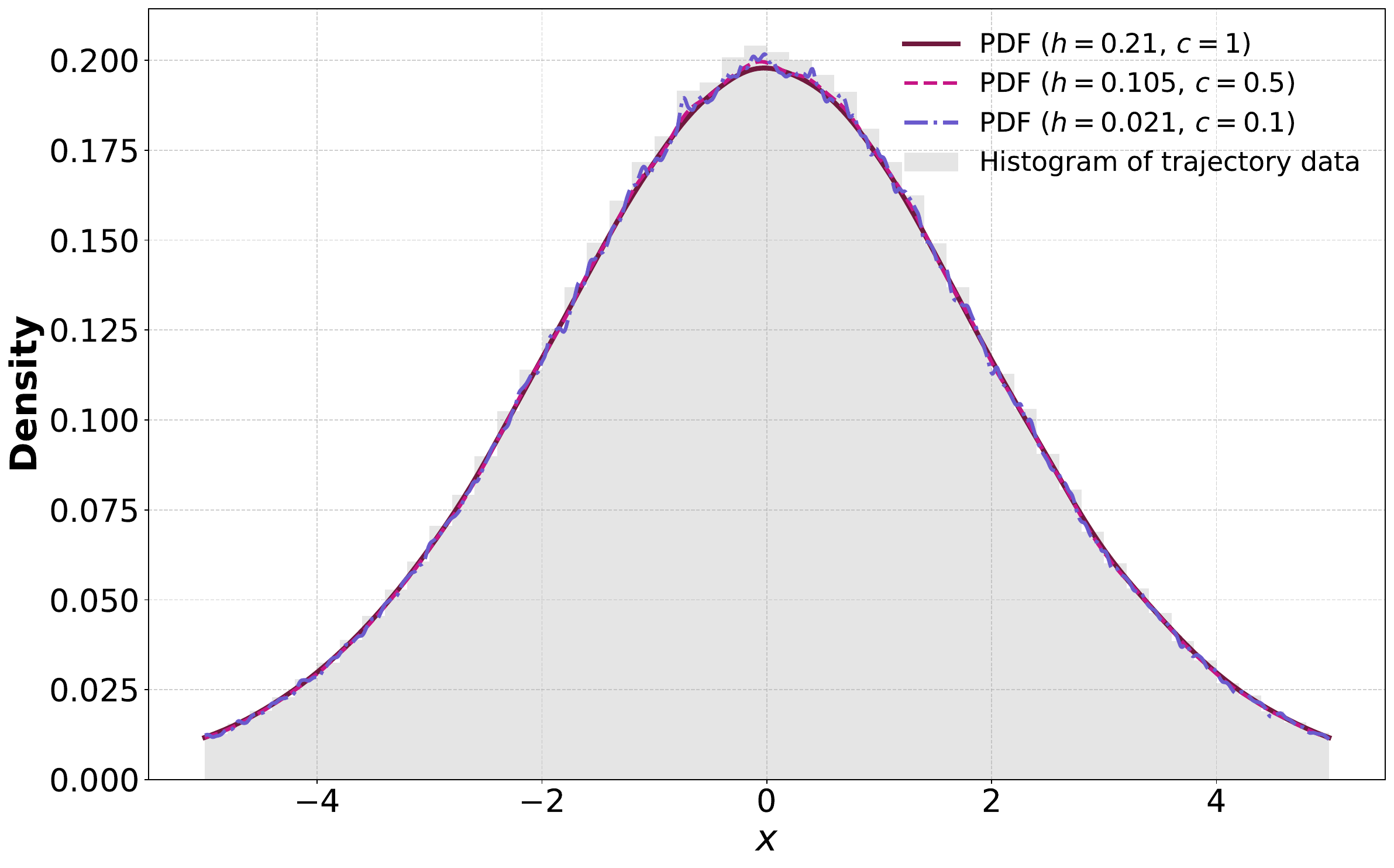}
    \label{fig:KDEpdf1}}
    \subfigure[]{\includegraphics[width=0.38\textwidth,height=0.28\textwidth]{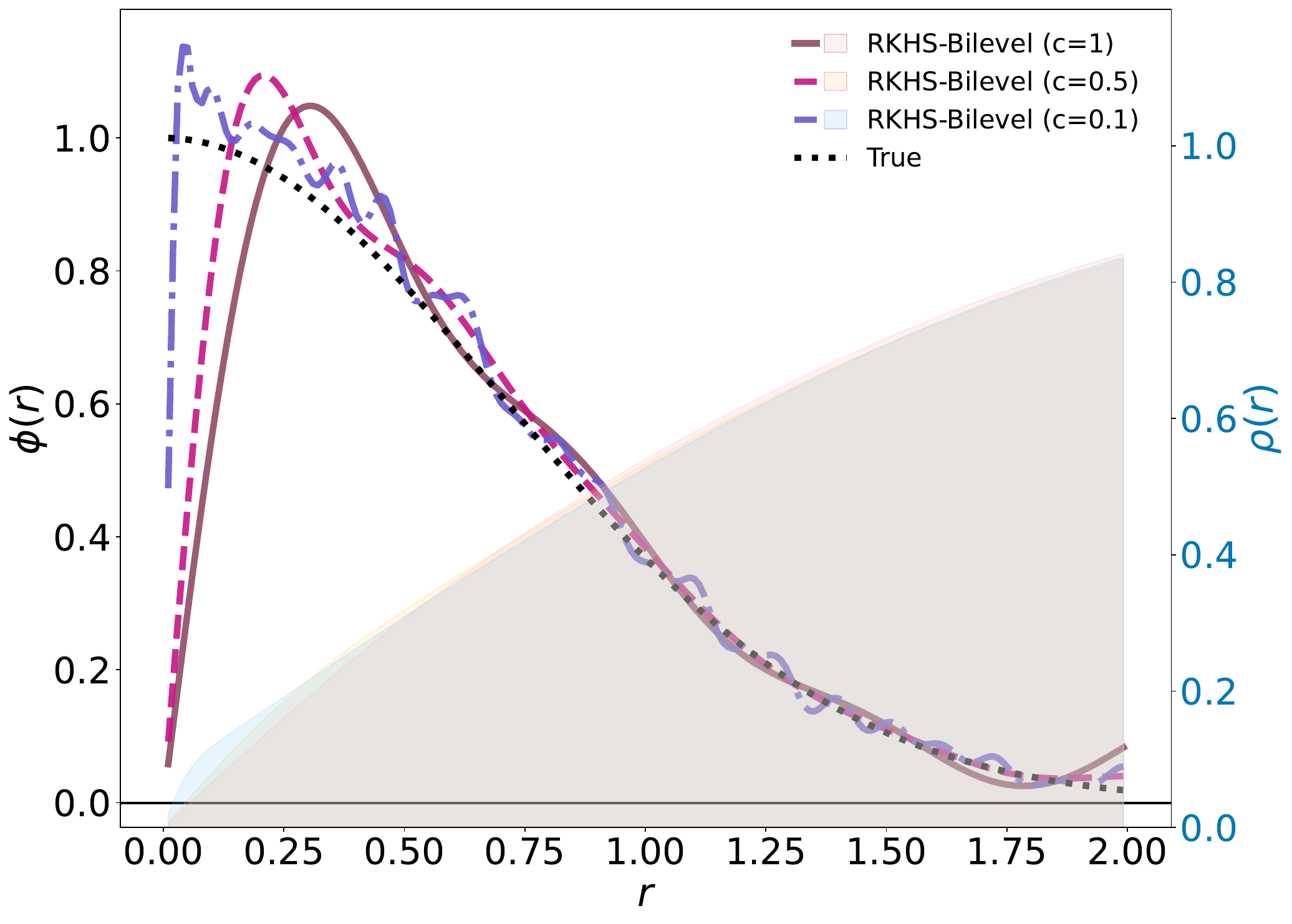}
    \label{fig:KDEband2}}
     \subfigure[]{\includegraphics[width=0.38\textwidth,height=0.28\textwidth]{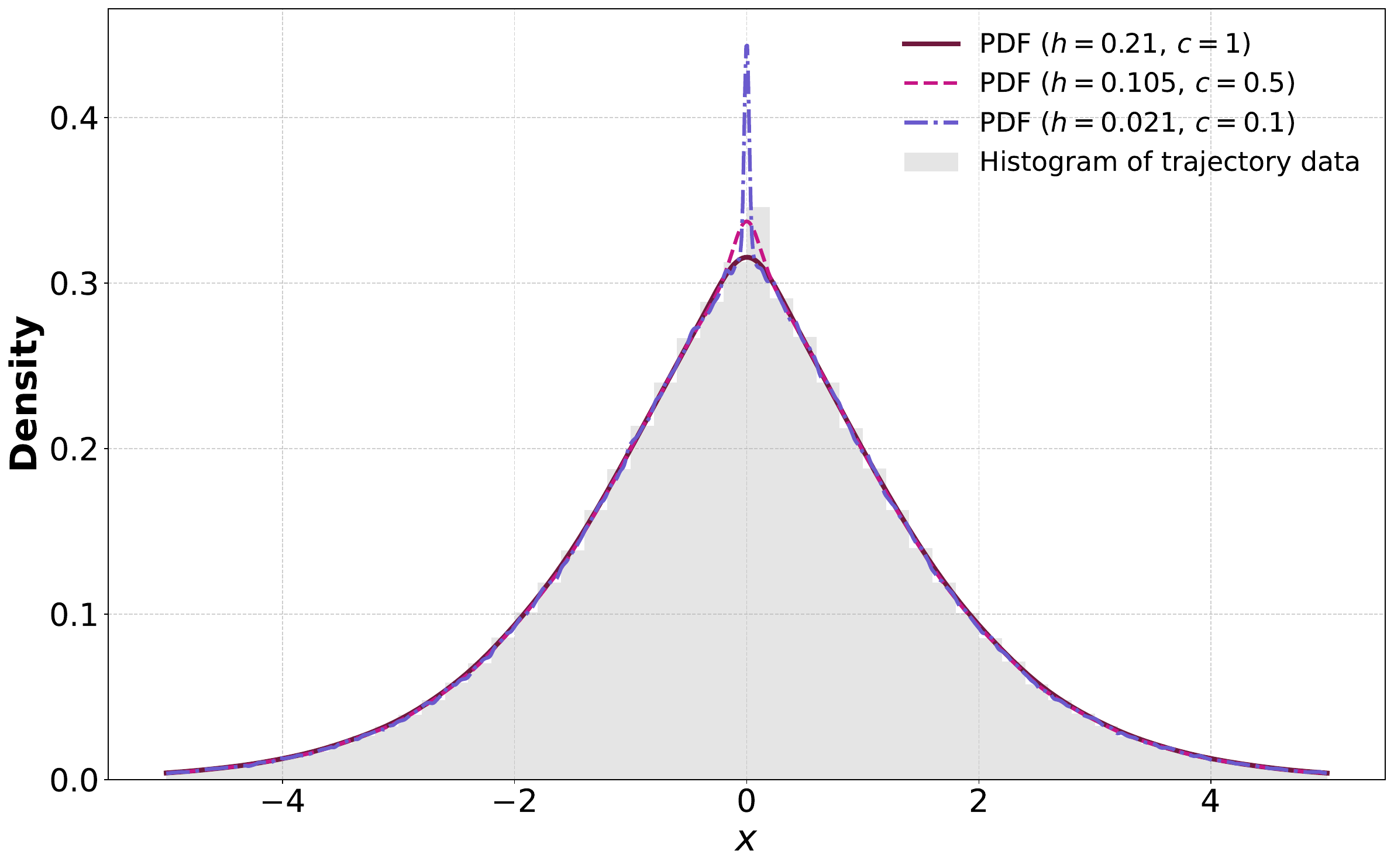}
    \label{fig:KDEpdf3}}
    \subfigure[]{\includegraphics[width=0.38\textwidth,height=0.28\textwidth]{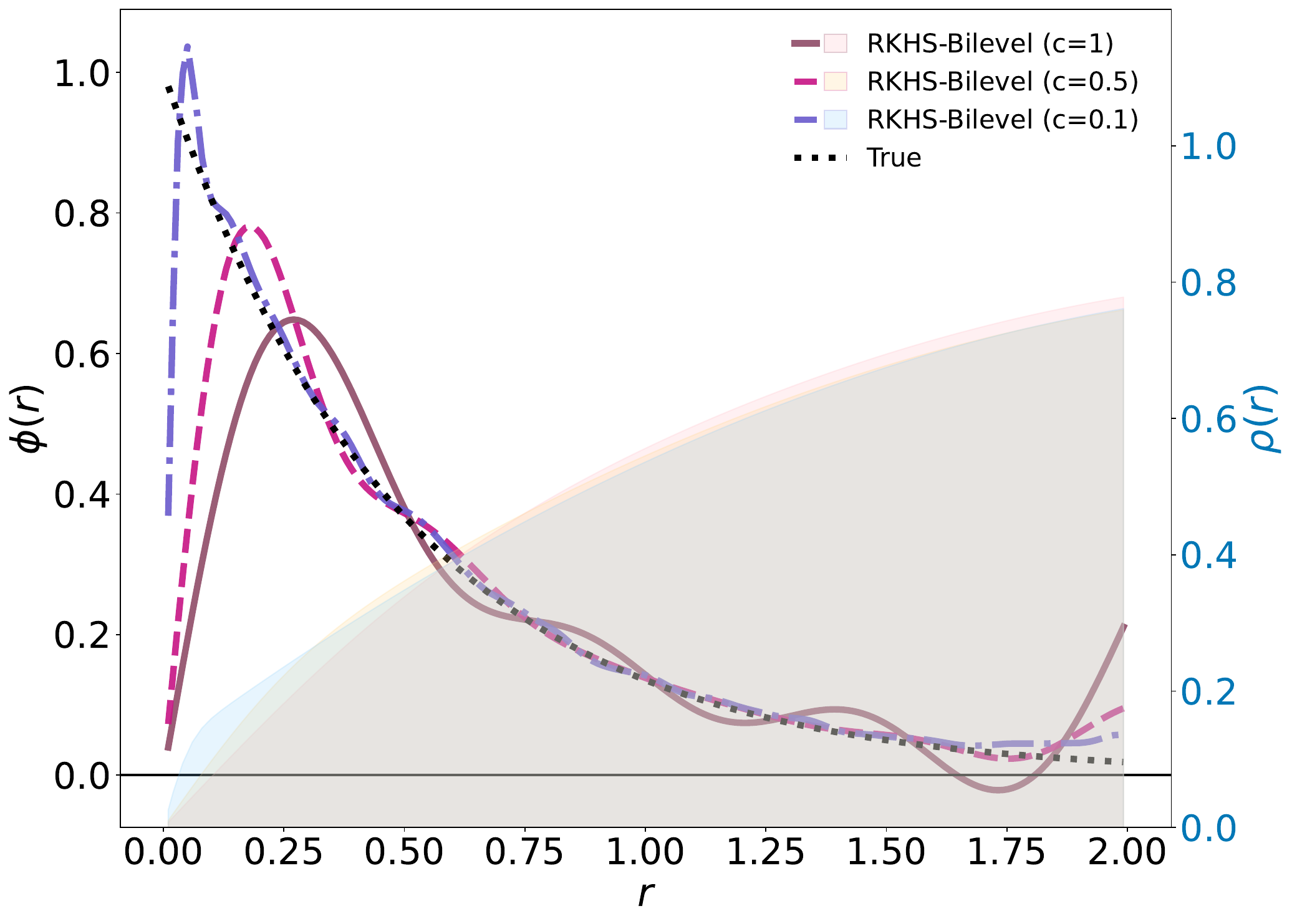}
    \label{fig:KDEband4}}
     \caption{Effect of the KDE bandwidth on density and L\'evy–density estimation for both Gaussian- and exponential-decay L\'evy measures. (a,c): KDE approximations of the PDF at a fixed time for several bandwidths $h$ (encoded by the constant $c$). Smaller $h$ produces more oscillatory PDFs; larger $h$ yields smoother but more biased PDFs near the origin. (b,d): corresponding L\'evy–density estimates: smaller $h$ leads to oscillatory estimates, while larger $h$ gives smoother but less accurate recovery.}   

    \label{fig:cbandwidth}
\end{figure}


\section{Conclusion}
\label{sec:conclusion}

We developed a nonparametric framework for learning the L\'evy density $\phi$ in the Fokker–Planck equation using an adaptive RKHS and bilevel optimization. Our main theoretical result is a mesh-dependent error bound for the DA–RKHS Tikhonov estimator. Under a power spectral decay condition on the normal operator, we proved that the reconstruction error decays in the spatial mesh size, with an optimal choice of the regularization parameter balancing regularization bias and numerical/approximation errors. When numerical errors are negligible, the resulting rates are comparable to minimax-optimal rates for inverse statistical learning, while making explicit how spatial discretization degrades accuracy.

Algorithmically, we proposed a GSVD-based bilevel scheme (\emph{bilevel-RKHS}) to select the regularization parameter by minimizing a validation loss. The lower level solves the Tikhonov problem in closed form via GSVD, and the same decomposition provides an explicit inverse-Hessian for efficient hypergradient computation in the upper level. 

Numerically, we demonstrated that the bilevel-RKHS method yields more accurate and stable L\'evy density estimates than L-curve and GCV across multiple L\'evy densities and drift fields. Also, the adaptive RKHS norm outperforms $L^2_\rho$- and $\ell^2$-based regularizers within the same bilevel framework.


\section*{Acknowledge}
L. Yang is partially supported by the CSC Fellowship, and would like to thank Prof. Xin Tong and Prof. Haibo Li for helpful discussions. The work of F.Lu is partially supported by the National Science Foundation under grant DMS-2238486 and DMS-2511283. This work of J. Duan and T. Gao was supported in part by the National Natural Science Foundation of China (12401233 and 12141107), the NSFC  International Collaboration Fund  for Creative Research Teams (Grant W2541005), the National Key Research and Development Program of China (2021ZD0201300), Guangdong Provincial Key Laboratory of Mathematical and Neural Dynamical Systems (2024B1212010004), the GuangDong Basic and Applied Basic Research Foundation (2025A1515011645), the Guangdong-Dongguan Joint Research Fund (2023A1515140016), the Cross Disciplinary Research Team on Data Science and Intelligent Medicine (2023KCXTD054).

\appendix
\section{Appendix: technical results and proofs}
\label{append:A}

\subsection{Proof of lemmas on RKHS}
\begin{proof}[Proof of Lemma \ref{lemma:rk}]
By definition, $\Gbar$ is positive semi-definite, i.e., for any finite set of points $\{r_i\}_{i=1}^m \subset \mathbb{R}^1$ and any coefficients $c_1, c_2, \ldots, c_m \in \mathbb{R}$, we have $\sum_{k=1}^n \sum_{k=1}^m c_k c_k \Gbar\left(r_j, r_k\right) \geqslant 0$.

To show $\Gbar\in L^2(\rho\otimes \rho)$, note that by {\rm Assumption \ref{assump:data}}, 
\[
\max_{i} \sup_{x\in \Omega,r\in [0,R_0]} \big| Q[p_{t_i}](x,r) \big| \leq 4  \max_{i} \sup_{x\in \Omega} p_{t_i}(x) \leq 4 C_{max}.  
\]

Then, for any $r,s\in [0,R_0]$, we have 
\begin{equation}
\begin{aligned}
\label{eq: boundGrs}
G(r, s) & =\frac{1}{N} \sum_{i=1}^N \int \left(Q[p_{t_i}](x, r) Q[p_{t_i}](x, s) \right) d x \\
        & \le 4C_{max}\frac{1}{N}\sum_{i=1}^N \int \left| Q[p_{t_i}](x, r) \right| dx  \le 4 C_{max} Z \rho(r). 
\end{aligned}
\end{equation}
Then, by the symmetry of $\Gbar$, we obtain 
\begin{equation}\label{eq:Gbarint}
\begin{aligned}
\iint|\Gbar(r, s)|^2 \rho(r) \rho(s) d r d s &= \iint \left|\frac{G(r, s)}{\rho(r) \rho(s)}\right|^2 \rho(r) \rho(s) d r d s \leq 16 C_{max}^2 Z^2 < \infty.
\end{aligned}
\end{equation}
That is, $\Gbar\in L^2(\rho\otimes \rho)$.

To prove \eqref{eq:lossFn_LG}, by definition of $\LGbar$ in \eqref{eq:LGbar}, we have 
\begin{equation*}   
\begin{aligned}
\innerp{\LGbar\phi,\psi}_{L^2_\rho}  & =\frac{1}{N} \sum_{i=1}^N \int_{\Omega_{R_0}} R_\phi\left[p_{t_i}\right]R_\psi\left[p_{t_i}\right] dx .\\
& =\frac{1}{N} \sum_{i=1}^N \int_{\Omega_{R_0}} \left(\int_0^{R_{0}} \phi(r) Q[p_t](x,r)dr \int_0^{R_{0}} \psi(s) Q[p_t](x,s)ds\right) dx .\\
& =\int_0^{R_{0}} \int_0^{R_{0}} \phi(r) \psi(s) G(r, s) d r d s  =\int_0^{R_{0}} \int_0^{R_{0}} \phi(r) \psi(s) \Gbar(r, s) \rho(d r) \rho(d s). 
\end{aligned}
\end{equation*}
Then, the loss function can be written as 
\begin{equation*}
\begin{aligned}
\mathcal{E}_\infty \left({\phi}\right) 
  & =\frac{1}{N} \sum_{i=1}^N \int_{\Omega_{R_0}}\left|R_{\phi}\left[p_{t_i}\right](x)-f_{t_i}(x)\right|^2 d x  = \innerp{\LGbar \phi, \phi}_{L^2_\rho} - 2\innerp{\phi^D, \phi}_{L^2_\rho} + C. 
\end{aligned}
\end{equation*}
Its Fr\'echet derivative in $L^2_\rho$ is $\nabla \mathcal{E} \left(\phi\right) =2\left(\mathcal{L}_{\Gbar} \phi-\phi^D\right)$. Thus, its unique minimizer in $\mathrm{Null}(\LGbar)^\perp$ is $\widehat \phi = \LGbar^{-1} \phi^D$. 
\end{proof}

\begin{proof}[Proof of Lemma \ref{lem:opma} (Operator form $\Leftrightarrow$ Matrix form)]
Since $\mathcal{L}_{\Gbar^M}$ is defined by $
(\mathcal{L}_{\Gbar^M}\phi)(r)
= \int_0^{R_0} \Gbar^M(r,s)\,\phi(s)\,\rho(s)\,ds,
$ 
using the definition of $G^M$ in \eqref{eq: discreteG} and $\Gbar^M$ in \eqref{eq:disbasis} and interchanging sums and integrals yields 
$$
\mathbf{\Abar}^M(k, k')
= \int_0^{R_0}\!\!\int_0^{R_0} G^M(r,s)\,\phi_k^M(r)\,\phi_{k'}^M(s)\,dr\,ds
= \langle \mathcal{L}_{\Gbar^M}\phi_k^M,\phi_{k'}^M\rangle_{L_\rho^2}.
$$
The reproducing property of $\Gbar^M$ implies that 
$$
\mathbf{\Gbar}^M(k,k')
= \Gbar^M(r_k,r_{k'})
= \langle \phi_k^M,\phi_{k'}^M\rangle_{H_{\Gbar^M}}
= \langle \mathcal{L}_{\Gbar^M}^{-1}\phi_k^M,\phi_{k'}^M\rangle_{L_\rho^2}.
$$
The identity $\mathbf{\bbar} ^M(k)
= \langle \phi^{D,M},\phi_k^M\rangle_{L_\rho^2}$ follows from the definition of $\phi^{D,M}$ and $R_{\phi_k^M}^M$ in \eqref{eq:RieszM}.

\noindent\emph{Operator form $\Rightarrow$ Matrix form} The operator form  in \eqref{eq:defphiMop} is equivalent to 
$$
(\mathcal{L}_{\Gbar^M} + \lambda \mathcal{L}_{\Gbar^M}^{-1})\hat\phi_\lambda^{n,M}
= \phi^{D,M}.
$$
Expand $\hat\phi_\lambda^{n,M} = \sum_{j=1}^n c_j\phi_j^M$ and test against each basis function $\phi_k^M$:
$$
\sum_j c_j\langle \mathcal{L}_{\Gbar^M}\phi_j^M,\phi_k^M\rangle_{L_\rho^2}
+ \lambda \sum_j c_j\langle \mathcal{L}_{\Gbar^M}^{-1}\phi_j^M,\phi_k^M\rangle_{L_\rho^2}
= \langle \phi^{D,M},\phi_k^M\rangle_{L_\rho^2}.
$$
In matrix form, this is $\bigl(\mathbf{\Abar}^M + \lambda \mathbf{\Gbar}^M\bigr)\mathbf c
= \mathbf{\bbar} ^M$. Thus, the solution is  
$ 
\hat{\mathbf c}_\lambda^{n,M}
= \bigl(\mathbf{\Abar}^M + \lambda \mathbf{\Gbar}^M\bigr)^{-1}\mathbf{\bbar} ^M
$. 

\noindent\emph{Matrix form $\Rightarrow$ operator form. }
Conversely, suppose $\hat\phi_\lambda^{n,M} = \sum_j \hat c_j \phi_j^M$ with $\hat{\mathbf c}_\lambda^{n,M}$ solving $
\bigl(\mathbf{\Abar}^M + \lambda \mathbf{\Gbar}^M\bigr)\hat{\mathbf c}_\lambda^{n,M}
= \mathbf{\bbar} ^M$. 
Then for each $k$,
$$
\sum_j \hat c_j\langle \mathcal{L}_{\Gbar^M}\phi_j^M,\phi_k^M\rangle_{L_\rho^2}
+ \lambda \sum_j \hat c_j\langle \mathcal{L}_{\Gbar^M}^{-1}\phi_j^M,\phi_k^M\rangle_{L_\rho^2}
= \langle \phi^{D,M},\phi_k^M\rangle_{L_\rho^2}. 
$$
Since $\{\phi_k^M\}$ spans $H_{\Gbar^M}$, this implies
$$
\langle \mathcal{L}_{\Gbar^M}\hat\phi_\lambda^{n,M}
+ \lambda \mathcal{L}_{\Gbar^M}^{-1}\hat\phi_\lambda^{n,M}
- \phi^{D,M},\psi\rangle_{L_\rho^2} = 0
\quad\forall \psi\in H_{\Gbar^M}.
$$
Since $\mathcal{L}_{\bar{G}^M}(L^2_\rho)  \subset H_{\bar{G}^M}$ and $\mathcal{L}_{\Gbar^M}$ is self-adjoint, replacing the test function $\psi $ by $\mathcal{L}_{\Gbar^M}\phi$ with $\phi\in L^2_\rho$,  we obtain  
$$
\langle  ( \mathcal{L}_{\Gbar^M}^2 
+ \lambda I ) \hat\phi_\lambda^{n,M}
- \mathcal{L}_{\Gbar^M} \phi^{D,M},\phi\rangle_{L_\rho^2} = 0
\quad\forall \psi\in L^2_\rho, 
$$ 
i.e., $
(\mathcal{L}_{\Gbar^M}^2 + \lambda I)\hat\phi_\lambda^{n,M}
= \mathcal{L}_{\Gbar^M}\phi^{D,M},
$
or equivalently, the operator form in \eqref{eq:defphiMop}. 
\end{proof}

\subsection{Proofs of three key Lemmas}\label{sec:tech_lemmas}
In this subsection we prove three key lemmas used in the proof of Theorem \ref{thm:conv}.

\begin{proof}[Proof of Lemma \ref{lem: regerror} (Regularization Error)] Note that 
\begin{equation}
\label{eq:diffphireg}
{\phi}_\lambda^{\infty}-\phi^*=\left(\mathcal{L}_{\Gbar}^2+\lambda I\right)^{-1} \mathcal{L}_{\Gbar}^2 \phi^*-\phi^*=-\lambda\left(\mathcal{L}_{\Gbar}^2+\lambda I\right)^{-1} \phi^*. 
\end{equation}
Since $\phi^*=\mathcal{L}_{\Gbar}^{\beta / 2} w$ with $w \in L_\rho^2$ and  $\LGbar$ is compact, by spectral representation we obtain,
$$
\left\|{\phi}_\lambda^{\infty}-\phi^*\right\|_{L_\rho^2} \leq \sup _{t \in \sigma(\LGbar^2)} \frac{\lambda t^{\beta / 4}}{t+\lambda}\|w\|_{L_\rho^2} .
$$
Since $\left\|\mathcal{L}_{\Gbar}\right\| \leq C_1$ by Lemma \ref{Lem: HSoperator}, we have $\sigma(\LGbar^2) \subset\left[0, C_1^2\right]$. 
Set $h_\lambda(t):=\lambda t^{\beta / 4} /(t+\lambda)$. Write $t=\lambda s$ to get
$
h_\lambda(\lambda s)=\lambda^{\beta / 4} \frac{s^{\beta / 4}}{1+s}$. 
\begin{itemize}
    \item If $0<\beta<4$, the function $s^\alpha /(1+s)$ with $\alpha:=\beta / 4 \in(0,1)$ attains its maximum at $s= \alpha /(1-\alpha)$ with value $\alpha^\alpha(1-\alpha)^{1-\alpha}$. Hence
$
\sup _{t \geq 0} h_\lambda(t)=\left(\frac{\beta}{4}\right)^{\beta / 4}\left(1-\frac{\beta}{4}\right)^{1-\beta / 4} \lambda^{\beta / 4} $. 
\item If $\beta=4, h_\lambda(t)=\lambda t /(t+\lambda) \leq \lambda$.
\item If $\beta>4, h_\lambda(t) \leq \lambda t^{\beta/4-1} \leq \lambda C_1^{\beta / 2-2}$.
\end{itemize}
Combine these and $\|w\|_{L_\rho^2} \leq R$ to obtain the stated bound.
\end{proof}

\begin{proof}[Proof of Lemma \ref{lem:approxerror} (Approximation Error)]
Recall that 
 $ {\phi}_\lambda^{\infty}=\left(\mathcal{L}_{\Gbar}^2+\lambda I\right)^{-1} \mathcal{L}_{\Gbar} \phi^D, \quad \hat{\phi}_\lambda^n=\left(\mathcal{L}_{\Gbar^n}^2+\lambda I\right)^{-1} \mathcal{L}_{\Gbar^n} \phi^{D,n} $.  
 Fix $\lambda>0$, denoting 
$$
\mathcal{A}:=\mathcal{L}_{\Gbar}^2+\lambda I, \mathcal{A}_n:=\mathcal{L}_{\Gbar^n}^2+\lambda I, 
$$ 
and using the resolvent identity $\left(\mathcal{A}_n\right)^{-1}- \mathcal{A}^{-1} = \left(\mathcal{A}_n\right)^{-1}\left(\mathcal{A}-\mathcal{A}_n\right) \mathcal{A}^{-1}$, we have 
\begin{align}
\hat{\phi}_\lambda^n-{\phi}_\lambda^{\infty}  
 =& \left(\mathcal{A}_n\right)^{-1} \mathcal{L}_{\Gbar^n}\left(\phi^{D,n}- \phi^D\right)+\left(\mathcal{A}_n\right)^{-1}\left(\mathcal{L}_{\Gbar^n}-\mathcal{L}_{\Gbar}\right) \phi^D+\left(\mathcal{A}_n\right)^{-1}\left(\mathcal{A}-\mathcal{A}_n\right) \mathcal{A}^{-1} \mathcal{L}_{\Gbar} \phi^D \notag  \\
=&\left(\mathcal{A}_n\right)^{-1} \mathcal{L}_{\Gbar^n}\left(\phi^{D,n}- \phi^D\right)+\left(\mathcal{A}_n\right)^{-1}\left(\mathcal{L}_{\Gbar^n}-\mathcal{L}_{\Gbar}\right) \phi^D+\left(\mathcal{A}_n\right)^{-1}\left(\mathcal{A}-\mathcal{A}_n\right){\phi}_\lambda^{\infty}. \label{eq:decompapproxerror}
\end{align}

We first bound $\left\|\mathcal{A}_n^{-1} \mathcal{L}_{\Gbar^n}\left(\phi^{D,n}- \phi^D\right)\right\|_{L_\rho^2}$. 
Recall the definitions of $\phi^{D,n}$ and $\phi^D$ in \eqref{eq:Rieszn} and \eqref{eq:phiD}, we have,  for any $\phi \in L_\rho^2$,
$$
\begin{aligned}
\left|\left\langle\phi^{D, n}-\phi^D, \phi\right\rangle_{L^2_{\rho}}\right| 
& = 
\bigg | \int_\delta^{R_0} \int \big(\sum_{\ell=1}^n Q(x,r_{\ell})\mathbf{1}_{I_{\ell}}(r) -Q(x,r)\big) \phi(r) d r f_{t_i}(x) d x \bigg|   \\
& \leq\left\|f_{t_i}\right\|_{\infty}\left|\Omega_{R_0}\right| \sup _x \sum_{\ell} \int_{I_{\ell}}\left|Q\left(x, r_{\ell}\right)-Q(x, r)\right||\phi(r)| d r \\
& \leq 2C_0\left|\Omega_{R_0}\right| C'_{\max} \Delta r \int_{\delta}^{R_0}|\phi(r)| d r \leq 2C_0\left|\Omega_{R_0}\right| C'_{\max} \Delta r \sqrt{\frac{R_0}{\rho_{0 }}}\|\phi\|_{L^2_{\rho}}.
\end{aligned}
$$
Here, under Assumption \ref{H:constant}, we get $\rho(r) = \rho_0$. Therefore, since $\Delta r= \Delta x$, we have
$$
\left\|\phi^{D, n}-\phi^D\right\|_{L^2_{\rho}}=\sup _{\|\phi\|_{L^2_{\rho}}=1}\left|\left\langle\phi^{D, n}-\phi^D, \phi\right\rangle\right| \leq 2 C_0\left|\Omega_{R_0}\right| C'_{\max} \sqrt{\frac{R_0}{\rho_{0 }}} \Delta x.
$$
Meanwhile, Lemma \ref{lem:effdim} gives $\|\mathcal{A}_n^{-1} \mathcal{L}_{\Gbar^n} \phi\|_{L^2_\rho}\leq \sqrt{\frac{\mathcal{N}_n(\lambda)}{\lambda}} \|\phi\|_{L^2_\rho}$ for any $\phi\in L^2_\rho$, and Lemma \ref{lem:diff-effdim} implies $|\mathcal{N}_n(\lambda)-\mathcal{N}(\lambda)| \leq \frac{2C_1C_2}{\lambda}\,\Delta x $. Thus, $\mathcal{N}_n(\lambda)
\le \mathcal{N}(\lambda) + |\mathcal{N}_n(\lambda)-\mathcal{N}(\lambda)| $ and $\sqrt{a+b}\le \sqrt{a}+\sqrt{b}$ imply    
\begin{equation}
\label{eq:Nnlambda}
\sqrt{\frac{\mathcal{N}_n(\lambda)}{\lambda}}
\;\le\;
\sqrt{\frac{\mathcal{N}(\lambda)}{\lambda}}
\;+\;
\sqrt{\frac{2C_1C_2}{\lambda^2}\,\Delta x}.
\end{equation}
Cobmining the above estimates with notation $C_L := 2 C_0\left|\Omega_{R_0}\right| C'_{\max} \sqrt{\frac{R_0}{\rho_{0}}}$, we obtain, 
$$
\begin{aligned}
\bigl\|\mathcal{A}_n^{-1} \mathcal{L}_{\Gbar^n}\bigl(\phi^{D,n}- \phi^D\bigr)\bigr\|_{L_\rho^2}
&\leq \sqrt{\frac{\mathcal{N}_n(\lambda)}{\lambda}}\,
      \bigl\|\phi^{D,n}- \phi^D\bigr\|_{L^2_{\rho}}\\
&\leq C_L\,\Delta x\,\sqrt{\frac{\mathcal{N}(\lambda)}{\lambda}} +
C_L\sqrt{2C_1C_2}\,\frac{\Delta x^{3/2}}{\lambda}. 
\end{aligned}
$$
Next, to bound the last two terms in \eqref{eq:decompapproxerror}, observe that
$$
\mathcal{A}_n^{-1}\left(\mathcal{L}_{\bar{G}^n}-\mathcal{L}_{\bar{G}}\right) \phi^D+\mathcal{A}_n^{-1}\left(\mathcal{A}-\mathcal{A}_n\right) \phi_\lambda^{\infty}=\mathcal{A}_n^{-1} \mathcal{L}_{\bar{G}^n}\left(\mathcal{L}_{\bar{G}}-\mathcal{L}_{\bar{G}^n}\right) \phi^*+\mathcal{A}_n^{-1}\left(\mathcal{L}_{\bar{G}^n}^2-\mathcal{L}_{\bar{G}}^2\right)\left(\phi^*-\phi_\lambda^{\infty}\right) .
$$
Using $\left\|\mathcal{L}_{\bar{G}^n}-\mathcal{L}_{\bar{G}}\right\| \leq C_2 \Delta x$ in Lemma \ref{Lem: HSoperator} and $\left\|\mathcal{A}_n^{-1} \mathcal{L}_{\bar{G}^n}\right\|\leq \sup _{t \geq 0} \frac{t}{t^2+\lambda}=\frac{1}{2 \sqrt{\lambda}}$, we have 
$$
\left\|\mathcal{A}_n^{-1} \mathcal{L}_{\bar{G}^n}\left(\mathcal{L}_{\bar{G}}-\mathcal{L}_{\bar{G}^n}\right) \phi^*\right\| \leq \frac{C_2}{2} \frac{\Delta x}{\sqrt{\lambda}}\left\|\phi^*\right\| .
$$
Since we get 
$
\left\|\mathcal{L}_{\bar{G}}\left(\phi^*-\phi_\lambda^{\infty}\right)\right\|=\left\|\lambda \mathcal{L}_{\bar{G}}\left(\mathcal{L}_{\bar{G}}^2+\lambda I\right)^{-1} \phi^*\right\| \leq \sup _{t \geq 0} \frac{\lambda t}{t^2+\lambda}\left\|\phi^*\right\| \leq \frac{\sqrt{\lambda}}{2}\left\|\phi^*\right\| 
$
from \eqref{eq:diffphireg},
$$
\begin{aligned}
\left\|\mathcal{A}_n^{-1}\left(\mathcal{L}_{\bar{G}^n}^2-\mathcal{L}_{\bar{G}}^2\right)\left(\phi^*-\phi_\lambda^{\infty}\right)\right\| & \le \left\|\mathcal{A}_n^{-1}\mathcal{L}_{\bar{G}^n}\left(\mathcal{L}_{\bar{G}^n} -\mathcal{L}_{\bar{G}}\right)\left(\phi^*-\phi_\lambda^{\infty}\right)\right\| + \left\|\mathcal{A}_n^{-1}\left(\mathcal{L}_{\bar{G}^n} -\mathcal{L}_{\bar{G}}\right)\mathcal{L}_{\bar{G}}\left(\phi^*-\phi_\lambda^{\infty}\right)\right\|\\
& \le  \frac{C_2}{2} \frac{\Delta x}{\sqrt{\lambda}}\left\|\phi^*-\phi_\lambda^{\infty}\right\| + \frac{C_2}{2} \frac{\Delta x}{\sqrt{\lambda}}\left\|\phi^*\right\|,
\end{aligned}
$$
based on $\mathcal{L}_{\bar{G}^n}^2-\mathcal{L}_{\bar{G}}^2=\mathcal{L}_{\bar{G}^n}\left(\mathcal{L}_{\bar{G}^n}-\mathcal{L}_{\bar{G}}\right)+\left(\mathcal{L}_{\bar{G}^n}-\mathcal{L}_{\bar{G}}\right) \mathcal{L}_{\bar{G}}$. By source condition in Assumption \ref{H:source_condi} and regularization error in Lemma \ref{lem: regerror}, $\left\|\phi^*\right\| \leq C_1^{\beta / 2} R$ and $\left\|\phi^*-\phi_\lambda^{\infty}\right\| \leq C\left(\beta, C_1\right) R \lambda^{\min \{\beta / 4,1\}}$, so
$$
\left\|\mathcal{A}_n^{-1}\left(\mathcal{L}_{\bar{G}^n}-\mathcal{L}_{\bar{G}}\right) \phi^D+\mathcal{A}_n^{-1}\left(\mathcal{A}-\mathcal{A}_n\right) \phi_\lambda^{\infty}\right\| \leq C_2 C_1^{\beta / 2} R \frac{\Delta x}{\sqrt{\lambda}}+\frac{C_2}{2} C\left(\beta, C_1\right)R \Delta x \lambda^{-1 / 2+\min \{\beta / 4,1\}}.
$$
For $0<\lambda \leq 1$, $\Delta x \lambda^{-1 / 2+\min \{\beta / 4,1\}} \leq \Delta x \lambda^{-1 / 2}$ with $\min \{\beta / 4,1\} \in[0,1]$. Hence, the term $\Delta x \lambda^{-1 / 2+\min \{\beta / 4,1\}}$ can be absorbed into $\Delta x \lambda^{-1 / 2}$ by adjusting the constant, yielding the final bound. Combing the bounds,
$$
\left\|\hat{\phi}_\lambda^n-\phi_\lambda^{\infty}\right\|_{L_\rho^2} \leq C_L \Delta x \sqrt{\frac{\mathcal{N}(\lambda)}{\lambda}}+C_L \sqrt{2 C_1 C_2} \frac{\Delta x^{3 / 2}}{\lambda}+ C_\beta \frac{\Delta x}{\sqrt{\lambda}},
$$
where $C_\beta = C_2 C_1^{\beta / 2} R + \frac{C_2}{2} C\left(\beta, C_1\right)R$.
\end{proof}


\begin{proof}[Proof of Lemma \ref{lem:numerroroperator} (Numerical error)]
Given by the definitions of $\hat{\phi}_\lambda^n$ and $\hat{\phi}_\lambda^{n,M}$ in \eqref{eq:defphi}, add and subtract $\mathcal{A}_M^{-1}\mathcal{L}_{\bar{G}^M}\phi^{D,n}$:
$$
\begin{aligned}
\hat{\phi}_\lambda^{n,M}-\hat{\phi}_\lambda^n
&= \mathcal{A}_M^{-1}\mathcal{L}_{\bar{G}^M}(\phi^{D,M}-\phi^{D,n}) + \bigl(\mathcal{A}_M^{-1}\mathcal{L}_{\bar{G}^M} - \mathcal{A}_n^{-1}\mathcal{L}_{\bar{G}^n}\bigr)\phi^{D,n},
\end{aligned}
$$
where $\mathcal{A}_n := \mathcal{B}_n + \lambda I,\, \mathcal{A}_M := \mathcal{B}_M + \lambda I$ with $\mathcal{B}_n := \mathcal{L}_{\bar{G}^n}^2,\,
\mathcal{B}_M := \mathcal{L}_{\bar{G}^M}^2$. For the second term, use the resolvent identity:
$$
\mathcal{A}_M^{-1}\mathcal{L}_{\bar{G}^M}
 - \mathcal{A}_n^{-1}\mathcal{L}_{\bar{G}^n}
= \mathcal{A}_M^{-1}(\mathcal{L}_{\bar{G}^M}-\mathcal{L}_{\bar{G}^n})
 + \mathcal{A}_M^{-1}(\mathcal{B}_n-\mathcal{B}_M)\mathcal{A}_n^{-1}
   \mathcal{L}_{\bar{G}^n}.
$$
Thus
$$
\begin{aligned}
\hat{\phi}_\lambda^{n,M}-\hat{\phi}_\lambda^n
&= \underbrace{\mathcal{A}_M^{-1}\mathcal{L}_{\bar{G}^M}(\phi^{D,M}-\phi^{D,n})}_{T_1}
 + \underbrace{\mathcal{A}_M^{-1}(\mathcal{L}_{\bar{G}^M}-\mathcal{L}_{\bar{G}^n})\phi^{D,n}}_{T_2}
\\
&\quad
 + \underbrace{\mathcal{A}_M^{-1}(\mathcal{B}_n-\mathcal{B}_M)\mathcal{A}_n^{-1}\mathcal{L}_{\bar{G}^n}\phi^{D,n}}_{T_3}.
\end{aligned}
$$

\medskip\noindent
\emph{Step 1: bound for $T_1$.}
By the effective-dimension inequality for $\mathcal{A}_M$ and $\mathcal{L}_{\bar{G}^M}$ in Lemma \ref{lem:effdim},
$$
\|T_1\|_{L^2_\rho}
= \|\mathcal{A}_M^{-1}\mathcal{L}_{\bar{G}^M}(\phi^{D,M}-\phi^{D,n})\|_{L^2_\rho}
\le \sqrt{\frac{\mathcal{N}_M(\lambda)}{\lambda}}\,\|\phi^{D,M}-\phi^{D,n}\|_{L^2_\rho}.
$$
By definitions of $\left\langle\phi^{D, n}, \phi\right\rangle_{L^2_{\rho}} $ in \eqref{eq:Rieszn} and $\left\langle \phi^{D,M},\phi\right\rangle_{L_\rho^2}$ in \eqref{eq:RieszM}, with $R_\phi^M[p_{t_i}](x_m) = R_\phi^n[p_{t_i}](x_m)$, 
$$
\begin{aligned}
\bigl|\langle\phi^{D, M}-\phi^{D,n}, \phi\rangle_{L^2_{\rho}}\bigr|
&= \Bigl| \frac{1}{N}\sum_{i=1}^N\sum_{m=1}^M 
  R_\phi^M[p_{t_i}](x_m)\,\tilde f_{t_i}(x_m)\,\Delta x 
 - \frac{1}{N} \sum_{i=1}^N \int_{\Omega_{R_0}} R^n_{\phi}\left[p_{t_i}\right](x) f_{t_i}(x)\, d x\Bigr| \\
&\leq A_1 + A_2,
\end{aligned}
$$
with
$$
\begin{aligned}
A_1 
&:= \left|\frac{1}{N}\sum_{i=1}^N\sum_{m=1}^M
  R_\phi^n[p_{t_i}](x_m)\bigl(\tilde f_{t_i}(x_m)-f_{t_i}(x_m)\bigr)\Delta x\right|,\\
A_2 
&:= \left|\frac{1}{N}\sum_{i=1}^N\sum_{m=1}^M
  R_\phi^n[p_{t_i}](x_m)f_{t_i}(x_m)\Delta x
 - \frac{1}{N}\sum_{i=1}^N\int_{\Omega_{R_0}}R_\phi^n[p_{t_i}](x)f_{t_i}(x)\,dx\right|.
\end{aligned}
$$
The key step in bounding $A_1$ is to control $f_{t_i}\left(x_m\right)-\tilde{f}_{t_i}\left(x_m\right)$, which decomposes into three contributions: the time derivative, the first spatial derivative and the second spatial derivative. Let $w_1(x) = b(x)p(x,t_i)$ and $w_1(x) = \sigma^2(x)p(x,t_i)$. By Taylor expansion in time and space, 
\begin{equation}
\label{eq:diffbound}
\begin{aligned}
&\left|\partial_t p\left(t_i, x_m\right)-\frac{p\left(t_{i+1}, x_m\right)-p\left(t_i, x_m\right)}{\Delta t}\right|=\frac{\Delta t}{2}\left|\partial^2_{t}p\left(\tau, x_m\right)\right| \le \frac{\Delta t}{2} \sup _{s \in\left[t_i, t_{i+1}\right]}\left|\partial^2_{t}p(s,  x_m)\right| \\
&\left|\partial_x w_1\left(x_m\right) -\frac{w_1(x_{m+1})-w_1(x_{m-1})}{2 \Delta x}\right| \leq \frac{\Delta x^2}{6} \sup _{\xi \in\left[x_{m-1}, x_{m+1}\right]}\left|\partial^3_xw_1\left(\xi\right)\right|\\
& \left|\partial^2_{x} w_2\left( x_m\right)-\frac{w_2\left(x_{m+1}\right) -2w_2\left(x_m\right)+w_2 \left(x_{m-1}\right)}{\Delta x^2}\right| \leq \frac{\Delta x^2}{12}\sup _{\zeta \in\left[x_{m-1}, x_{m+1}\right]}\left|\partial^4_{x} w_2 \left(\zeta\right)\right|
\end{aligned}
\end{equation}
By the triangle inequality, we get
\begin{equation}
\left|f_{t_i}\left(x_m\right)-\tilde{f}_{t_i}\left(x_m\right)\right| \le C_t \Delta t+C_w \Delta x^2,
\end{equation}
where $C_t$ and $C_w$ depend only on bounds of $\partial_t^2 p, \partial_x^3(b p)$ and $\partial_x^4\left(\sigma^2 p\right)$ from Remark \eqref{rem:regularity}.Using the bound $\left\|R_\phi^n\left[p_{t_i}\right]\right\|_{\infty}=\sup _x\left|\sum_{\ell=1}^n \int Q\left[p_{t_i}\right]\left(x, r_{\ell}\right) \mathbf{1}_{I_{\ell}}(r) \phi(r) d r\right| \leq 4 C_{\max } \sqrt{\frac{R_0}{\rho_{0 }}}\|\phi\|_{L_\rho^2}$, which follows from Cauchy-Schwarz inequality applied to $\int_{\delta}^{R_0}|\phi(r)| d r$, together with  \eqref{eq:diffbound} and Assumption \ref{H:constant}, we have
$$
\begin{aligned}
A_1 
&\le \frac{1}{N}\sum_{i=1}^N \|R_\phi^n[p_{t_i}]\|_\infty
     \bigl(C_t \Delta t+C_w \Delta x^2\bigr)\sum_{m=1}^M \Delta x \\
&\le 4 C_{\max } \sqrt{\frac{R_0}{\rho_{0}}} 
   |\Omega_{R_0}|\bigl(C_t \Delta t+C_w \Delta x^2\bigr)\,\|\phi\|_{L^2_{\rho}}.
\end{aligned}
$$
For $A_2$, set $b_i(x):=R^n_{\phi}\left[p_{t_i}\right](x) f_{t_i}(x)$. A standard quadrature error estimate yields
$$
\left|\int_{\Omega_{R_0}} b_{i}(x) d x-\sum_{m=1}^M b_{i} \left(x_m\right) \Delta x\right| 
\le \frac{1}{2}\|b_i'\|_\infty\,|\Omega_{R_0}|\,\Delta x.
$$
Using $\|b_i'\|_\infty 
\le \|\partial_x R^n_{\phi}[p_{t_i}]\|_\infty\|f_{t_i}\|_\infty +\|R^n_{\phi}[p_{t_i}]\|_\infty\|\partial_x f_{t_i}\|_\infty$ and the bounds
$$
\|R_{\phi}^n[p_{t_i}]\|_{\infty}
\le 4 C_{\max } \sqrt{\frac{R_0}{\rho_{0}}}\,\|\phi\|_{L^2_{\rho}},
\qquad
\|\partial_x R_{\phi}^n[p_{t_i}]\|_{\infty}
\le 4 C_{\max }^{'} \sqrt{\frac{R_0}{\rho_{0}}}\,\|\phi\|_{L^2_{\rho}},
$$
we obtain
$$
A_2 \le 4 \bigl( C_{\max }^{'} C_0 +  C_{\max } C_1\bigr) 
|\Omega_{R_0}|\sqrt{\frac{R_0}{\rho_{0}}}\,\|\phi\|_{L^2_{\rho}}\,\Delta x.
$$
Altogether,
$$
\bigl|\langle\phi^{D, M}-\phi^{D,n}, \phi\rangle_{L^2_{\rho}}\bigr|
\le 4 \sqrt{\frac{R_0}{\rho_{0}}}\,
|\Omega_{R_0}|\Bigl[
C_{\max}\bigl(C_t \Delta t+C_w \Delta x^2\bigr)  
+ \bigl( C_{\max }^{'} C_0 +  C_{\max } C_1\bigr)\Delta x
\Bigr]\|\phi\|_{L^2_{\rho}}.
$$
Taking the supremum over $\|\phi\|_{L^2_\rho}=1$ gives a bound for $\|\phi^{D,M} - \phi^{D,n}\|_{L^2_\rho}$ and thus, 
$$
\|T_1\|_{L^2_\rho}
\le 4 \sqrt{\frac{R_0}{\rho_{0}}}\,
|\Omega_{R_0}|\Bigl[
C_{\max}\bigl(C_t \Delta t+C_w \Delta x^2\bigr)  
+ \bigl( C_{\max }^{'} C_0 +  C_{\max } C_1\bigr)\Delta x
\Bigr]\sqrt{\frac{\mathcal{N}_M(\lambda)}{\lambda}}.
$$
\emph{Step 2: bound for $T_2$.} Under Assumption \ref{H:constant}, we have $\rho(r) =\rho_0$. From
$$
\begin{aligned}
\bigl|\langle \phi^{D,n},\phi\rangle_{L^2_\rho}\bigr|
&\le \frac{1}{N}\sum_{i=1}^N\int_{\Omega_{R_0}}
|R_\phi^n[p_{t_i}](x)|\,|f_{t_i}(x)|\,dx \le  4 C_{\max }  C_0  |\Omega_{R_0}|\sqrt{\frac{R_0}{\rho_{0}}}\,\|\phi\|_{L^2_\rho},
\end{aligned}
$$
we get $\|\phi^{D,n}\|_{L^2_\rho}
\le 4 C_{\max }  C_0 |\Omega_{R_0}| \sqrt{\frac{R_0}{\rho_{0}}}.$ Then, by Lemma \ref{Lem: HSoperator},
$$
\begin{aligned}
\|T_2\|_{L^2_\rho}
&= \|\mathcal{A}_M^{-1}(\mathcal{L}_{\bar{G}^M}-\mathcal{L}_{\bar{G}^n})\phi^{D,n}\|_{L^2_\rho} \\
&\le \|\mathcal{A}_M^{-1}\|\;\|\mathcal{L}_{\bar{G}^M}-\mathcal{L}_{\bar{G}^n}\|\;
     \|\phi^{D,n}\|_{L^2_\rho} 
\le 4 C_{\max }  C_0 C_2\,|\Omega_{R_0}| \sqrt{\frac{R_0}{\rho_{0}}}\,
   \frac{\Delta x}{\lambda}.
\end{aligned}
$$
\medskip\noindent
\emph{Step 3: bound for $T_3$.} By the source condition and $\left\|\mathcal{L}_{\Gbar}\right\| \leq C_1 $ in Lemma \ref{Lem: HSoperator}, we have 
$$
\left\|{\phi}_\lambda^{\infty}\right\|_{L_\rho^2} \leq \sup _{t \in \sigma(\mathcal{B})} \frac{t^{1 + \beta / 4}}{t+\lambda}\|w\|_{L_\rho^2} \leq \sup _{t \in \sigma(\mathcal{B})} t^{\beta / 4}\|w\|_{L_\rho^2} \leq C_1^{\beta/2}R. 
$$
Using the bound on $\left\|\hat{\phi}_\lambda^n-\phi_\lambda^{\infty}\right\|_{L_\rho^2}$ in Lemma \ref{lem:approxerror} and the bound on $\left\|\phi_\lambda^{\infty}\right\|_{L_\rho^2}$,
$$
\begin{aligned}
\left\|\mathcal{A}_n^{-1} \mathcal{L}_{\bar{G}^n} \phi^{D, n}\right\|_{L_\rho^2}=\left\|\hat{\phi}_\lambda^n\right\|_{L_\rho^2} & \leq\left\|\hat{\phi}_\lambda^n-\phi_\lambda^{\infty}\right\|_{L_\rho^2}+\left\|\phi_\lambda^{\infty}\right\|_{L_\rho^2} \\
& \leq C_L \Delta x \sqrt{\frac{\mathcal{N}(\lambda)}{\lambda}}+C_L \sqrt{2 C_1 C_2} \frac{\Delta x^{3 / 2}}{\lambda}+C_\beta \frac{\Delta x}{\sqrt{\lambda}}+C_1^{\beta / 2} R .
\end{aligned}
$$
Combining the bound of $\|\mathcal{B}_n-\mathcal{B}_M\|$ from Lemma \ref{Lem: HSoperator} and the bound of $\left\|\mathcal{A}_n^{-1} \mathcal{L}_{\bar{G}^n} \phi^{D, n}\right\|_{L_\rho^2}$, 
$$
\begin{aligned}
\left\|T_3\right\|_{L_\rho^2} & =\left\|\mathcal{A}_M^{-1}\left(\mathcal{B}_n-\mathcal{B}_M\right) \mathcal{A}_n^{-1} \mathcal{L}_{\bar{G}^n} \phi^{D, n}\right\|_{L_\rho^2} \\
& \leq\left\|\mathcal{A}_M^{-1}\right\|\left\|\mathcal{B}_n-\mathcal{B}_M\right\|\left\|\mathcal{A}_n^{-1} \mathcal{L}_{\bar{G}^n} \phi^{D, n}\right\|_{L_\rho^2} \\
& \leq \frac{\Delta x C_1C_2}{\lambda}\left(C_L \Delta x \sqrt{\frac{\mathcal{N}(\lambda)}{\lambda}}+C_L \sqrt{2 C_1 C_2} \frac{\Delta x^{3 / 2}}{\lambda}+C_\beta \frac{\Delta x}{\sqrt{\lambda}}+C_1^{\beta / 2} R\right)\\ 
\end{aligned}
$$

\medskip\noindent
\emph{Step 4: collect the bounds.} From  Lemma \ref{lem:diff-effdim} and \eqref{eq:Nnlambda} we know that $\sqrt{\frac{\mathcal{N}_M(\lambda)}{\lambda}} \le \sqrt{\frac{\mathcal{N}_n(\lambda)}{\lambda}} + \sqrt{\frac{2C_1C_2}{\lambda^2}\,\Delta x} \le \sqrt{\frac{\mathcal{N}(\lambda)}{\lambda}} + 2\sqrt{\frac{2C_1C_2}{\lambda^2}\,\Delta x}$. Putting together the bounds for $T_1$, $T_2$, and $T_3$ and substituting the above estimates for 
$\sqrt{\mathcal{N}_M(\lambda)/\lambda}$ and $\sqrt{\mathcal{N}_n(\lambda)/\lambda}$, we obtain
$$
\begin{aligned}
\|\hat{\phi}_\lambda^{n,M}-\hat{\phi}_\lambda^n\|_{L^2_\rho}
&\le \|T_1\|_{L^2_\rho} + \|T_2\|_{L^2_\rho} + \|T_3\|_{L^2_\rho} \\
& \leq K_1\left(\Delta t+\Delta x+\Delta x^2\right)\left(\sqrt{\frac{\mathcal{N}(\lambda)}{\lambda}}+\sqrt{\frac{\Delta x}{\lambda^2}}\right) \\ & +K_2 \frac{\Delta x}{\lambda}+K_3 \frac{\Delta x^2}{\lambda} \sqrt{\frac{\mathcal{N}(\lambda)}{\lambda}}+K_4 \frac{\Delta x^{5 / 2}}{\lambda^2} + K_5\frac{\Delta x^{2}}{\lambda^{3/2}}
\end{aligned}
$$
where the constants $K_i >0$ are independent of $\Delta x$, $\Delta t$ and $\lambda$.
This completes the proof.
\end{proof}

\subsection{Technical lemmas and their proofs}
\begin{lemma} \label{Lem: HSoperator}
Under Assumptions \ref{H:boundf} and \ref{H:constant}, the kernels $\Gbar, \Gbar^n, \Gbar^M$, defined in \eqref{eq:barG}, \eqref{eq:Gbar^n}, \eqref{eq: discreteG}, are in $L^2(\rho \otimes \rho)$ with norms bounded by $C_1$, so their integral operators $\LGbar,\mathcal{L}_{\Gbar^n},\mathcal{L}_{\Gbar^M} $ have norms bounded by $C_1$. 
Moreover,
\begin{align}
\left\|\mathcal{L}_{\Gbar}-\mathcal{L}_{\Gbar^n}\right\| & \leq\left\|\Gbar-\Gbar^n\right\|_{L^2(\rho \otimes \rho)} \leq C_2 \Delta x,\\
\left\|\mathcal{L}_{\Gbar^n}- \mathcal{L}_{\Gbar^M} \right\| & \leq\left\|\Gbar^n-\Gbar^M\right\|_{L^2(\rho \otimes \rho)} \leq C_2 \Delta x,
\end{align}
with $C_1=4 C_{\max }Z$, $C_2=\frac{16 C_{\max }\left|\Omega_{R_0}\right| C'_{\max}}{\rho_0^2}$ and $\Delta x  = \Delta r$.
\end{lemma}

\begin{proof}[Proof of Lemma \ref{Lem: HSoperator}]
Based on the definition of $\Gbar$ in \eqref{eq:barG} and Assumption  \ref{assump:data},
$$
\|\Gbar\|_{L^2(\rho \otimes \rho)}^2=\iint\left|\Gbar(r,s)\right|^2 \rho(r) \rho(s) d r d s \leq  16 C_{\max }^2 Z^2
$$
so $\|\Gbar\|_{L^2(\rho \otimes \rho)} \leq C_1$. Similarly,  $\left\|\Gbar^n\right\|_{L^2(\rho \otimes \rho)}, \left\|\Gbar^{M} \right\|_{L^2(\rho \otimes \rho)} \leq C_1$. Then, by the definition of $\mathcal{L}_{\Gbar}$ in \eqref{eq:LGbar}, for any $\phi \in L_{\rho}^2$
$$
\left|\left(\mathcal{L}_{\Gbar} \phi\right)(r)\right|=\left|\int \Gbar(r, s) \phi(s) \rho(s) d s\right| \leq\left(\int|\Gbar(r, s)|^2 \rho(s) d s\right)^{1 / 2}\|\phi\|_{L_\rho^2}
$$
by Cauchy-Schwarz. Squaring and integrating in $r$ against $\rho(r) d r$ gives
$$
\left\|\mathcal{L}_{\Gbar} \phi\right\|_{L_\rho^2}^2 \leq\left(\iint|\Gbar(r, s)|^2 \rho(r) \rho(s) d r d s\right)\|\phi\|_{L_\rho^2}^2=\|\Gbar\|_{L^2(\rho \otimes \rho)}^2\|\phi\|_{L_\rho^2}^2
$$
Taking $\sup _{\phi \neq 0}$ yields $\left\|\mathcal{L}_{\Gbar}\right\| \leq\|\Gbar\|_{L^2(\rho \otimes \rho)}$. Applying the same argument to $\Gbar^n$ proves $\left\|\mathcal{L}_{\Gbar^n}\right\| \leq\left\|\Gbar^n\right\|_{L^2(\rho\otimes \rho)}$. The same argument with $\Gbar$ replaced by $\Gbar-\Gbar^n$ gives $\left\|\mathcal{L}_{\Gbar}-\mathcal{L}_{\Gbar^n}\right\| \leq\left\|\Gbar-\Gbar^n\right\|_{L^2(\rho \otimes \rho)}$. Finally, by Remark \ref{rem:regularity}, for any $x$,
$$
\left|Q\left[p_{t_i}\right]\left(x, s_k\right)-Q\left[p_{t_i}\right](x, s)\right| \leq\left|p_{t_i}\left(x+s_k\right)-p_{t_i}(x+s)\right|+\left|p_{t_i}\left(x-s_k\right)-p_{t_i}(x-s)\right| \leq 2C'_{\max}\Delta r.
$$
Thus, for $(r, s) \in I_{k^{'}} \times I_k$,
$$
\begin{aligned}
\left|G\left(r_{k^{'}}, s_k\right)-G(r, s)\right| \leq & \frac{1}{N} \sum_{i=1} \int_{\Omega_{R_0}}\left|Q\left[p_{t_i}\right]\left(x, r_{k^{'}}\right)\right|\left|Q\left[p_{t_i}\right]\left(x, s_k\right)-Q\left[p_{t_i}\right](x, s)\right| d x \\
& +\frac{1}{N} \sum_{i=1}^N \int_{\Omega_{R_0}}\left|Q\left[p_{t_i}\right](x, s)\right|\left|Q\left[p_{t_i}\right]\left(x, r_{k^{'}}\right)-Q\left[p_{t_i}\right](x, r)\right| d x \\
\
\leq & 16 C_{\max } \left|\Omega_{R_0}\right| C'_{\max} \Delta r,
\end{aligned}
$$
with $\left|Q\left[p_{t_i}\right]\left(x, r\right)\right| \leq 4 C_{\max} ,\, \forall  x,\,r$. Under Assumption \ref{H:constant}, $\rho(r) = \rho_0$. Using the definition of $\Gbar^n$ in \eqref{eq:Gbar^n},
$$
\begin{aligned}
\left\|\Gbar^n-\Gbar\right\|_{L^2(\rho \otimes \rho)}^2  
& =\sum_{k^{'}, k} \iint_{I_{k^{'}} \times I_k} \frac{\left|G\left(r_{k^{'}}, s_k\right)-G(r, s)\right|^2}{\rho(r) \rho(s)} d r d s \\
& \leq\left(\frac{16 C_{\max }\left|\Omega_{R_0}\right| C'_{\max} \Delta r}{\rho_0^2}\right)^2 \sum_{k^{'}, k} \iint_{I_{k^{'}} \times I_k} \rho(r) \rho(s) d r d s \\
& =\left(\frac{16 C_{\max }\left|\Omega_{R_0}\right| C'_{\max}}{\rho_0^2}\right)^2(\Delta r)^2
\end{aligned}
$$
with  $\sum_{k^{'}, k} \iint_{I_{k^{'}} \times I_k} \rho(r) \rho(s) d r d s=\iint \rho(r) \rho(s) d r d s=1$.
Taking square roots gives $\left\|\Gbar^n-\Gbar\right\|_{L^2(\rho \otimes \rho)} \leq C_2 \Delta r$. 

The bound for $\left\|\mathcal{L}_{\Gbar^n}- \mathcal{L}_{\Gbar^M} \right\|$ follows by the same argument, using the definition of $\Gbar^M$ in \eqref{eq: discreteG} and noting that the same estimate for $\left|G\left(r_{k^{'}}, s_k\right)-G(r, s)\right|$ holds. 
\end{proof}

\begin{lemma}\label{lem:effdim}
Let $\mathcal{L}:L_\rho^2\to L_\rho^2$ be self-adjoint, nonnegative and Hilbert--Schmidt and set $\mathcal{A}:=\mathcal{B}+\lambda I$ with $ \mathcal{B}:=\mathcal{L}^2$ and $\lambda>0$. 
Denote the effective dimension by $
\mathcal{N}(\lambda)
:= \operatorname{Tr}\bigl(\mathcal{B}(\mathcal{B}+\lambda I)^{-1}\bigr).
$
Then, for every $\phi\in L_\rho^2$,
$$
\|\mathcal{A}^{-1}\mathcal{L}\phi\|_{L_\rho^2}
\le \sqrt{\frac{\mathcal{N}(\lambda)}{\lambda}}\;\|\phi\|_{L_\rho^2}.
$$

In particular, the same estimate holds for 
$\mathcal{L}=\mathcal{L}_{\Gbar}$, for each piecewise-constant operator 
$\mathcal{L}_{\Gbar^n}$ in $r$, and for the fully discrete operator 
$\mathcal{L}_{\Gbar^M}$ in $(x,r)$, with effective dimensions
$\mathcal{N}(\lambda)$, $\mathcal{N}_n(\lambda)$ and $\mathcal{N}_M(\lambda)$
defined in the same way from $\mathcal{L}_{\Gbar}$, $\mathcal{L}_{\Gbar^n}$ 
and $\mathcal{L}_{\Gbar^M}$, respectively.
\end{lemma}

\begin{proof}[Proof of Lemma \ref{lem:effdim}]
If $\mathcal{L}\equiv 0$, then $\mathcal{B}=0$, $\mathcal{A}=\lambda I$ and $\mathcal{N}(\lambda)=0$, so both sides of the inequality are zero and the claim is trivial. Hence we assume $\mathcal{L}\not\equiv 0$.

Since $\mathcal{L}$ is compact, self-adjoint and nonnegative, the spectral theorem yields an orthonormal basis $\{e_j\}_{j\ge1}\subset L_\rho^2$ and eigenvalues $\sigma_j\ge0$ such that
$
\mathcal{L} e_j = \sigma_j e_j,\qquad j=1,2,\dots$. 
Then
$ 
\mathcal{B}=\mathcal{L}^2 \quad\Rightarrow\quad \mathcal{B}e_j=\sigma_j^2 e_j =:\mu_j e_j,
$
so $\{e_j\}$ is also an eigenbasis of $\mathcal{B}$ with eigenvalues $\mu_j=\sigma_j^2$, and
$$
\mathcal{A}=\mathcal{B}+\lambda I \quad\Rightarrow\quad \mathcal{A}e_j=(\mu_j+\lambda)e_j,\qquad \lambda>0.
$$
The Hilbert--Schmidt norm of $\mathcal{A}^{-1}\mathcal{L}$ is
$$
\|\mathcal{A}^{-1}\mathcal{L}\|_{\mathrm{HS}}^2
= \sum_{j\ge1}\|\mathcal{A}^{-1}\mathcal{L}e_j\|_{L_\rho^2}^2
= \sum_{j\ge1}\frac{\mu_j}{(\mu_j+\lambda)^2}.
$$
For each $j$ we have 
$$
\frac{\mu_j}{(\mu_j+\lambda)^2}
= \frac{1}{\mu_j+\lambda}\cdot\frac{\mu_j}{\mu_j+\lambda}
\le \frac{1}{\lambda}\cdot\frac{\mu_j}{\mu_j+\lambda},
$$
since $\mu_j+\lambda\ge\lambda$. Summing over $j$ gives
$$
\|\mathcal{A}^{-1}\mathcal{L}\|_{\mathrm{HS}}^2
\le \frac{1}{\lambda}\sum_{j\ge1}\frac{\mu_j}{\mu_j+\lambda}
= \frac{\mathcal{N}(\lambda)}{\lambda}.
$$
Thus, 
$
\|\mathcal{A}^{-1}\mathcal{L}\|
\le \|\mathcal{A}^{-1}\mathcal{L}\|_{\mathrm{HS}}
\le \sqrt{\frac{\mathcal{N}(\lambda)}{\lambda}}$ and 
$
\|\mathcal{A}^{-1}\mathcal{L}\phi\|_{L_\rho^2}
\le \sqrt{\frac{\mathcal{N}(\lambda)}{\lambda}}\;\|\phi\|_{L_\rho^2}$ for any $\phi\in L_\rho^2$.  

The cases $\mathcal{L}=\mathcal{L}_{\Gbar}$, $\mathcal{L}=\mathcal{L}_{\Gbar^n}$ and $\mathcal{L}=\mathcal{L}_{\Gbar^M}$ follow by applying the same argument to each operator, which only changes the eigenvalues and hence replaces $\mathcal{N}(\lambda)$ by $\mathcal{N}(\lambda)$, $\mathcal{N}_n(\lambda)$ and $\mathcal{N}_M(\lambda)$ respectively.
\end{proof}

\begin{lemma}\label{lem:diff-effdim}
Let $\mathcal{B} := \mathcal{L}_{\Gbar}^2$ and $\mathcal{B}_n := \mathcal{L}_{\Gbar^n}^2$, and assume 
$\mathcal{L}_{\Gbar}\not\equiv 0$. Define the effective dimensions: $ \mathcal{N}(\lambda) := \operatorname{Tr}\bigl(\mathcal{B}(\mathcal{B}+\lambda I)^{-1}\bigr),\, \mathcal{N}_n(\lambda) := \operatorname{Tr}\bigl(\mathcal{B}_n(\mathcal{B}_n+\lambda I)^{-1}\bigr).$ Then, for any fixed $\lambda>0$,
$$
\bigl|\mathcal{N}_n(\lambda)-\mathcal{N}(\lambda)\bigr|
\le \frac{2 C_1 C_2}{\lambda}\,\Delta r.
$$
Moreover, if $\mathcal{B}_M := \mathcal{L}_{\Gbar^M}^2$ and $\mathcal{N}_M(\lambda) := \operatorname{Tr}\bigl(\mathcal{B}_M(\mathcal{B}_M+\lambda I)^{-1}\bigr)$, then the same bound applies to $\bigl|\mathcal{N}_M(\lambda)-\mathcal{N}_n(\lambda)\bigr|$.
\end{lemma}
\begin{proof}[Proof of Lemma \ref{lem:diff-effdim}]
Fix $\lambda>0$ and define $f_\lambda(t) := t/(t+\lambda)$ for $t\ge0$. Since $\mathcal{B}$ and $\mathcal{B}_n$ are self-adjoint, trace-class operators, let $\{\mu_j\}_{j\ge1}$ and $\{\mu_{j,n}\}_{j\ge1}$ denote their nonnegative eigenvalues, listed in non-increasing order. Then
\begin{equation}
\label{eq:diffNn}
\begin{aligned}
|\mathcal{N}_n(\lambda)-\mathcal{N}(\lambda)|
&= \left|\sum_{j\ge1} f_\lambda(\mu_{j,n}) - \sum_{j\ge1} f_\lambda(\mu_j)\right| \le \sum_{j\ge1}\bigl|f_\lambda(\mu_{j,n}) - f_\lambda(\mu_j)\bigr|
 \le \frac{1}{\lambda}\sum_{j\ge1}|\mu_{j,n}-\mu_j|,
\end{aligned}
\end{equation}
since
$
\bigl|f_\lambda(a)-f_\lambda(b)\bigr|
= \left|\frac{a}{a+\lambda} - \frac{b}{b+\lambda}\right|
= \lambda\,\frac{|a-b|}{(a+\lambda)(b+\lambda)}
\le \frac{1}{\lambda}|a-b|,
$
for all $a,b\ge0$. Next, by the Schatten Hölder inequality $\|AB\|_1 \le \|A\|_2\|B\|_2$ for Hilbert--Schmidt operators $A,B$, we obtain
$$
\begin{aligned}
\|\mathcal{B}_n-\mathcal{B}\|_1
&\le \|\mathcal{L}_{\Gbar^n}(\mathcal{L}_{\Gbar^n}-\mathcal{L}_{\Gbar})\|_1
 + \|(\mathcal{L}_{\Gbar^n}-\mathcal{L}_{\Gbar})\mathcal{L}_{\Gbar}\|_1 \\
&\le \|\mathcal{L}_{\Gbar^n}\|_2\|\mathcal{L}_{\Gbar^n}-\mathcal{L}_{\Gbar}\|_2
 + \|\mathcal{L}_{\Gbar^n}-\mathcal{L}_{\Gbar}\|_2\|\mathcal{L}_{\Gbar}\|_2 \le 2 C_1 C_2\,\Delta r
\end{aligned}
$$
by using $\|\mathcal{L}_{\Gbar^n}\|_2,\|\mathcal{L}_{\Gbar}\|_2\le C_1$ and $\|\mathcal{L}_{\Gbar^n}-\mathcal{L}_{\Gbar}\|_2 \le C_2\Delta r$ in Lemma \ref{Lem: HSoperator}. By the Mirsky inequality \cite{foucart2018concave},
\begin{equation}
\label{eq:diffmu}
\sum_{j\ge1}|\mu_{j,n}-\mu_j|
\le \|\mathcal B_n-\mathcal B\|_1
\le 2 C_1 C_2\,\Delta r.
\end{equation}
Combining \eqref{eq:diffNn} and \eqref{eq:diffmu} yields
$$
|\mathcal N_n(\lambda)-\mathcal N(\lambda)|
\le \frac{2 C_1 C_2}{\lambda}\,\Delta r.
$$
Similarly, applying the same argument with $(\mathcal B,\mathcal B_n)$ replaced by $(\mathcal B_n,\mathcal B_M)$ gives
$$
|\mathcal{N}_M(\lambda)-\mathcal{N}_n(\lambda)|
\le \frac{2 C_1 C_2}{\lambda}\,\Delta r.
$$
\end{proof}

The next lemma shows that when the eigenvalues decay polynomially with a rate $\lambda_k\sim k^{-\varsigma}$ for $\varsigma>1/2$, the effective dimension $\mathcal{N}_{2\varsigma}(\lambda)$ increases at the order $O(\lambda^{-1/2\varsigma})$ as $\lambda\downarrow 0$.

 \begin{lemma}\label{lemma:effdim_Bd}
Assume that there exist constants $a\geq 0$, $b>0$ and $\varsigma>1/4$ such that
$ a\,k^{-2\varsigma} \le \lambda_k \le  b\,k^{-2\varsigma}, \quad k\in\mathbb N$. 
Then, for all $0<\lambda<1$, we have  
\[
C_1\,\lambda^{-1/(4\varsigma)}-\frac1 2
 \le 
\mathcal{N}_{4\varsigma}(\lambda) := \sum_{k\ge1}\frac{\lambda^2_k}{\lambda^2_k+\lambda}
 \le 
C_2\,\lambda^{-1/(4\varsigma)}+1.
\]
where $C_1 = \frac{a^{1/(2\varsigma)}}{2}$
and $C_2 = b^{1/(2\varsigma)}\Big(1+\frac{1}{4\varsigma-1}\Big)$. 
In particular, we have 
$\mathcal{N}_{4\varsigma}(\lambda) \asymp \lambda^{-1/(4\varsigma)} \text{ as }\lambda\downarrow 0.$ 
\end{lemma}

\begin{proof}
Let $m(\lambda)$ be the number of $\lambda^2_k>\lambda$, i.e.,
$$m(\lambda) := \#\{k\ge1:\ \lambda^2_k\ge \lambda\}=\#\{k\ge1:\ \lambda_k\ge \sqrt{\lambda}\}.$$  
The two-sided decay condition implies that 
\[
(a/\sqrt{\lambda})^{1/(2\varsigma)}-1
\;\le\;
m(\lambda)
\;\le\;
(b/\sqrt{\lambda})^{1/(2\varsigma)}+1
=:\,K_b(\lambda).
\tag{1}
\]

\noindent Since $\frac12\,\mathbf 1_{\{x\ge1\}}  \le  \frac{x}{x+1}  \le 
\mathbf 1_{\{x\ge1\}} + x\,\mathbf 1_{\{x<1\}}$ for all $x\ge0$, with $x=\lambda_k^2/\lambda$, we have 
\begin{align*}
\frac12\,m(\lambda)  \le \mathcal{N}_{4\varsigma}(\lambda) = \sum_{k\ge1}\frac{\lambda_k^2/\lambda}{\lambda_k^2/\lambda+1}
& \le 
m(\lambda)+\frac{1}{\lambda}\sum_{k:\,\lambda_k^2<\lambda}\lambda_k^2.
\end{align*}
The lower bound with $C_1  = \frac{a^{1/(2\varsigma)}}{2}$ then follows directly from 
\[
\mathcal{N}_{4\varsigma}(\lambda) \ge \frac12\,m(\lambda)
 \ge \frac12\Big(\,(a/\sqrt{\lambda})^{1/(2\varsigma)}-1\Big)
 = \frac{a^{1/(2\varsigma)}}{2}\,\lambda^{-1/(4\varsigma)}-\frac12.
\]
To control the upper bound, note that $\{k:\lambda_k^2<\lambda\}\subset\{k>K_b(\lambda)\}$. Hence, using $\lambda_k\le b k^{-2\varsigma}$,
\[
\sum_{k:\,\lambda_k^2<\lambda}\lambda_k^2
\le
\sum_{k>K_b(\lambda)}\lambda_k^2
\le
b^2\sum_{k>K_b(\lambda)}k^{-4\varsigma}.
\]
By the integral test (valid since $4\varsigma>1$),
$\sum_{k>K}k^{-4\varsigma}\le \int_K^\infty x^{-4\varsigma}\,dx
=\frac{K^{\,1-4\varsigma}}{4\varsigma-1}$. Therefore,
\[
\mathcal{N}_{4\varsigma}(\lambda)
 \le 
K_b(\lambda)+\frac{b^2}{(4\varsigma-1)\lambda}\,K_b(\lambda)^{\,1-4\varsigma}
 \le 
b^{1/(2\varsigma)}\Bigl(1+\frac{1}{4\varsigma-1}\Bigr)\lambda^{-1/(4\varsigma)}+1, 
\]
which gives the upper bound with $C_2= b^{1/(2\varsigma)}\Big(1+\frac{1}{4\varsigma-1}\Big) $. 
\end{proof}

\makeatletter
\let\origthebibliography\thebibliography
\renewcommand{\thebibliography}[1]{%
    \origthebibliography{#1}%
    \setlength{\itemsep}{0.2pt}
}
\makeatother
{\small 
\bibliographystyle{unsrt}
\bibliography{referencesRKHS}
} 
\end{document}